\documentclass[12pt,reqno]{amsart}
\usepackage[paper=letterpaper,margin=1in]{geometry}

\usepackage{amsmath,amssymb,mathrsfs,amsthm,amsfonts,braket}
\usepackage[inline]{enumitem} 
\usepackage[usenames,dvipsnames]{xcolor}
\usepackage{hyperref}
\hypersetup{%
	colorlinks=true, linkcolor= blue, 
	citecolor= ForestGreen
}

\numberwithin{equation}{section}
\allowdisplaybreaks

\newtheorem{theorem}{Theorem}[section]
\newtheorem{lemma}[theorem]{Lemma}
\newtheorem{proposition}[theorem]{Proposition}
\newtheorem{corollary}[theorem]{Corollary}
\theoremstyle{definition}

\newtheorem{remark}[theorem]{Remark}

\usepackage{acronym}
\acrodef{BDG}{Burkholder--Davis--Gundy}
\acrodef{PDE}{Partial Differential Equation}

\usepackage{graphicx}
\newcommand*{\Cdot}{{\raisebox{-0.5ex}{\scalebox{1.8}{$\cdot$}}}} 

\renewcommand{\Set}[1]{ {\left\{ #1 \right\}} }

\DeclareMathOperator{\Ex}{\mathbf{E}}			
\renewcommand{\Pr}{\operatorname{\mathbf{P}}} 	
\DeclareMathOperator{\ind}{\mathbf{1}}			

\DeclareMathOperator{\Exp}{Exp}					
\DeclareMathOperator{\PPP}{PPP}					
\renewcommand{\hat}{\widehat}

\newcommand{\bbN}{ \mathbb{N} }
\newcommand{\bbR}{ \mathbb{R} }
\newcommand{\bbZ}{ \mathbb{Z} }

\newcommand{\Csp}{\mathcal{C}}					

\newcommand{\Dsp}{\mathcal{D}}					
\newcommand{\barDsp}{\overline{\mathcal{D}}}	

\newcommand{\pN}{p^\text{N}}			
\newcommand{\hatpN}{\widehat{p}^\text{N}}	
\newcommand{\erf}{\Phi}				
\newcommand{\erfc}{\widetilde{\Phi}}
\newcommand{\G}{G}					
\newcommand{\Gc}{\widetilde{G}}		
\newcommand{\W}{W}					
\newcommand{\U}{U}					
\newcommand{\Uc}{\widetilde{U}}		
\newcommand{\V}{V}					

\newcommand{\mgY}{N}					
\newcommand{\mg}{M}				

\newcommand{\BKi}{B^K_i}			
\newcommand{\XKi}{X^K_i}			
\newcommand{\YKi}{Y^K_i}			
\newcommand{\sfX}{\mathcal{X}}		
\newcommand{\sfJ}{\mathcal{J}}		

\newcommand{\emK}{\mu^K}			
\newcommand{\emKY}{\nu^K}			

\newcommand{\tKi}{\tau^K_i}			

\newcommand{\testf}{\Psi}					

\newcommand{\rdY}{R'}		
\newcommand{\rd}{R}			

\newcommand{\emrd}{E}		
\newcommand{\emrdY}{E'}		

\newcommand{\prd}{Q}		
\newcommand{\prdY}{Q'}		

\newcommand{\e}{\varepsilon}

\newcommand{\Yup}{ \overline{Y} }
\newcommand{\Ylw}{ \underline{Y} }
\newcommand{\YKup}{ \overline{Y}^{K} }
\newcommand{\YKlw}{ \underline{Y}^{K} }
\newcommand{\Wup}{ \overline{W} }
\newcommand{\Wlw}{ \underline{W} }
\newcommand{\Vup}{ \overline{V} }
\newcommand{\Vlw}{ \underline{V} }
\newcommand{\uup}{ \overline{u} }
\newcommand{\ulw}{ \underline{u} }
\newcommand{\Gup}{ \overline{G} }

\newcommand{\norm}{ | }

\renewcommand{\u}{u_{\star}}			
\newcommand{\z}{z_{\star}}			
\renewcommand{\t}{t_{\star}}
\newcommand{\Us}{U_{\star}}
\newcommand{\Ucs}{\widetilde{U}_{\star}}			
\newcommand{\Uss}{U_{\star\star}}

\title[Optimal Surviving Strategy for Drifted BM]
{Optimal Surviving Strategy for Drifted Brownian Motions with Absorption}
\author[W.\ Tang]{Wenpin Tang} 
	\address{Statistics department, University of California, Berkeley, CA 94720}
	\email{wenpintang@stat.berkeley.edu}	

\author[L.-C.\ Tsai]{Li-Cheng Tsai}
\address{L.-C.\ Tsai,
	Departments of Mathematics, Columbia University,
	\newline\hphantom{\quad\quad L-C Tsai}
	2990 Broadway, New York, NY 10027}
\email{lctsai.math@gmail.com}


\begin{document}
\maketitle

\begin{abstract}
We study the `Up the River' problem formulated by Aldous \cite{aldous02},
where a unit drift is distributed among 
a finite collection of Brownian particles on $ \bbR_+ $, 
which are annihilated once they reach the origin. 
Starting $ K $ particles at $ x=1 $,
we prove Aldous' conjecture \cite{aldous02}
that the `push-the-laggard' strategy of distributing the drift
asymptotically (as $ K\to\infty $) maximizes 
the total number of surviving particles,
with approximately $ \frac{4}{\sqrt{\pi}} \sqrt{K} $ surviving particles.
We further establish the hydrodynamic limit of the particle density, 
in terms of a two-phase \ac{PDE} with a moving boundary,
by utilizing certain integral identities and coupling techniques. 
\end{abstract}

\section{Introduction}\label{sect:intro}
In this paper we study the `Up the River' problem formulated by Aldous \cite{aldous02}.
That is, we consider $ K $ independent Brownian particles,
which all start at $ x=1 $,
and are absorbed (annihilated) once they hit $ x=0 $.
Granted a unit drift,
we ask what is the optimal strategy of
dividing and allocating the drift among all surviving particles 
in order to maximize the number of particles that survive forever.
More precisely, 
letting $ B_i(t) $, $ i=1,\ldots,K $, denote independent standard Brownian motions,
we define the model as an $ \bbR^K_+ $-valued diffusion $ ( X_i(t);t\geq 0)_{i=1}^K $,
satisfying
\begin{align}\label{eq:X:SDE}
	X_i(t) = 1 + B_i(t\wedge\tau_i) + \int_0^{t\wedge\tau_i} \phi_i(s) ds.
\end{align}
Here $ \tau_{i}:=\inf\{t>0:X_i(t)=0\} $ denotes the absorption time of the $ i $-th particle,
and the strategy is any $ [0,1]^K $-valued, $ \{B_i(t)\}_{i=1}^K $-progressively 
measurable function $ (\phi_i(t);t\geq 0)_{i=1}^K $
such that $ \sum_{i=1}^K \phi_i(t) \leq 1 $, $ \forall t\geq 0 $.
Our goal is to maximize $ \Uc(\infty) $, where
\begin{align*}
	\Uc(\infty) := \lim_{t\to\infty} \Uc(t),
	\quad
	\Uc(t) := \# \{ i: X_i(t) >0 \}.
\end{align*}
Here $ \Uc(t) $ actually depends on $ K $,
but we \emph{suppress} the dependence in this notation,
and reserve notations such as $ \Uc_K(t) $ for \emph{scaled} quantities.
Inspired by the `Up the River: Race to the Harbor' board game,
this simple model serves as a natural optimization problem 
for a random environment with limited resources.
For $ K=2 $, \cite{mckean06} obtains an explicit expression of the law of $ \Uc(t) $,
and for large $ K $, 
numerical results are obtained in \cite{han} for the discrete analog of \eqref{eq:X:SDE}.

Focusing on the asymptotic behavior as $ K\to\infty $, 
we prove that the optimal strategy is the na\"{\i}ve \textbf{push-the-laggard strategy}
\begin{align}\label{eq:pushlaggard}
	\phi_i(t) := \ind_\Set{ X_i(t) = Z(t) },
	\quad
	\text{ where }
	Z(t) := \min\{X_i(t) : X_i(t)>0\},
\end{align}
which allocates all the unit drift on the \textbf{laggard} $ Z(t) $.
\begin{remark}
Due to the recursive nature of Brownian motions in one-dimension,
ties do occur in \eqref{eq:pushlaggard}, 
namely $ \Pr(\#\{i: X_{i}(s)=Z(s)\} >1, \text{for some } s\leq t) > 0 $, 
for all large enough $ t $.
Here we break the ties in an \emph{arbitrarily} fixed manner.
That is, any strategy $ (\phi_i(t))_{i=1}^K $ satisfying
\begin{align}\label{eq:pushlaggard:}
	\sum_{ i: X_i(t) =Z(t) } \phi_i(t) = 1
\end{align}
is regarding as a push-the-laggard strategy.
As the analysis in this paper is independent of 
the exact choice of breaking the ties,
hereafter we fix some arbitrary way of breaking the ties
and refer to \eqref{eq:pushlaggard:} as \emph{the} push-the-laggard strategy.
\end{remark}
\noindent
Furthermore, we prove that, due to self-averaging,
$ \Uc(\infty) $ is in fact deterministic to the leading order, 
under the push-the-laggard strategy.
More explicitly, $ \Uc(\infty) \approx \frac{4}{\sqrt{\pi}} K^{1/2} $.
Define the scaled process 
\begin{align*}
	\Uc_K(t) := \tfrac{1}{\sqrt{K}} \Uc(tK).
\end{align*}
The following is our main result:
\begin{theorem}
	\begin{enumerate}[label=(\alph*)]
	\item[]
	\item\label{enu:aldous:upbd}
	Regardless of the strategy,
	for any fixed $ n<\infty $ and $ \gamma \in (0,\frac14) $, we have
	\begin{align}
		\label{eq:aldous:upbd}
		\Pr\big( \Uc_K(\infty) \leq \tfrac{4}{\sqrt{\pi}} + K^{-\gamma} \big) 
		\geq 1 - CK^{-n},	\quad \forall K<\infty,
	\end{align}
	where $ C=C(n,\gamma)<\infty $ depends only on $ n $ and $ \gamma $, 
	not on the strategy.

	\item\label{enu:aldous:optl}
	Under the push-the-laggard strategy,
	for any fixed $ \gamma\in(0,\frac{1}{96}) $ and $ n<\infty $, we have
	\begin{align}
		\label{eq:aldous:optl}
		\Pr\big( |\Uc_K(\infty) - \tfrac{4}{\sqrt{\pi}} | 
		\leq 
		K^{-\gamma} \big) \geq 1 - CK^{-n},	\quad \forall K<\infty,
	\end{align}
	where $ C=C(\gamma,n)<\infty $ depends only on $ \gamma $ and $ n $.
\end{enumerate}
\label{thm:aldous}
\end{theorem}
\begin{remark}
While the exponent $ \frac14^- $ of the error term in Theorem~\ref{thm:aldous}\ref{enu:aldous:upbd}
(originating from the control on the relevant martingales) 
is optimal, the choice of exponent $ \gamma\in(0,\frac{1}{96}) $ 
in Theorem~\ref{thm:aldous}\ref{enu:aldous:optl} is purely technical.
The latter may be improved by establishing sharper estimates, 
which we do not pursue in this paper.
\end{remark}

Theorem~\ref{thm:aldous} resolves Aldous' conjecture 
\cite[Conjecture~2]{aldous02} in a slightly different form.
The intuition leading to such a theorem, 
as well as the the main ingredient of proving it,
is the hydrodynamic limit picture given in \cite{aldous02}.
To be more precise,
we consider the diffusively scaled process $ \XKi(t) := \frac{1}{\sqrt{K}} X_i(tK) $
and let $ Z_K(t) := \tfrac{1}{\sqrt{K}} Z(tK) $
denote the scaled process of the laggard.
Consider further the scaled complementary distribution function
\begin{align}\label{eq:Uc}
	\Uc_K(t,x) := \tfrac{1}{\sqrt{K}} \# \big\{ \XKi(t) >x \big\},
\end{align}
and let $ p(t,x) := \frac{1}{\sqrt{2\pi t}} \exp(-\frac{x^2}{2t}) $ 
denote the standard heat kernel.
Under the push-the-laggard strategy,
we expect $ (\Uc_K(t,x),Z_K(t)) $ to be well-approximated by $ (\Ucs(t,x), \z(t)) $.
Here $ \Ucs(t,x) $ and $ \z(t) $ are deterministic functions,
which are defined in \emph{two separated phases} as follows.
For $ t \leq \frac12 $, the \textbf{absorption phase}, we define
\begin{align}
	\label{eq:Ucs:abs}
	&
	\Ucs(t,x) := 2 p(t,x) + \int_0^t 2p(t-s,x) ds,
	\quad
	\forall t \leq \tfrac12, \ x \geq 0,
\\
	&
	\label{eq:z=0}
	\z(t) :=0,
	\quad
	\forall t \leq \tfrac12.
\end{align}
For $ t > \frac12 $, the \textbf{moving boundary phase}, 
letting $ \pN(t,y,x) := p(t,y-x) + p(t,y+x) $ denote the Neumann heat kernel,
we define
\begin{align}
	\label{eq:Ucs:move}
	\Ucs(t,x) := 2 p(t,x) + \int_0^t \pN(t-s,\z(s),x) ds,
	\quad
	\forall t \geq \tfrac12, \ x \geq \z(t),
\end{align}
where  $ \z(t) $ is the unique solution to the following integral equation:
\begin{align}\label{eq:zeq}
	\left\{\begin{array}{l@{}l}
			\z(\Cdot+\tfrac12) \in \Csp(\bbR_+), \ \text{nondecreasing },
			\z(\tfrac12) = 0,	
			\vspace{5pt}			
		\\
			\displaystyle
			\int_0^\infty p(t-\tfrac12,\z(t)-y) 
			\big( \Ucs(\tfrac12,0) - \Ucs(\tfrac12,y) \big) dy 
		\\
			\displaystyle
			\quad\quad\quad\quad\quad
			= \int_\frac12^t p(t-s,\z(t)-\z(s)) ds,
			\quad \forall t \in (\tfrac12,\infty),
	\end{array}\right.
\end{align}
As we show in Section~\ref{sect:Stef},
the integral equation~\eqref{eq:zeq} admits a unique solution.

The pair $ (\Ucs,\z) $, defined by \eqref{eq:Ucs:abs}--\eqref{eq:zeq},
is closely related to certain \ac{PDE} problems, as follows.
Let $ \erfc(t,y) := \Pr(B(t) > y) $ denote the Brownian tail distribution function.
For $ t\leq\frac12 $,
a straightforward calculation (see Remark~\ref{rmk:cal})
shows that the function $ \Ucs(t,x) $ in \eqref{eq:Ucs:abs}
is written as the tail distribution function of $ u_1(t,x) $:
\begin{align}
	\label{eq:U1}
	\Ucs(t,x)
	=
	\int_{x}^\infty u_1(t,y) dy,
	\quad
	\forall t\leq \tfrac12,
\end{align}
where $ u_1(t,x) $ is defined as
\begin{align}
	\label{eq:u1}
	u_1(t,x) &:= -2\partial_x p(t,x) + 4 \erfc(t,x).
\end{align}
It is straightforward to check that this density function $ u_1 $
solves the heat equation on $ x>0 $ with a boundary condition $ u_1(t,0)=2 $:
\begin{subequations}\label{eq:PDE<}
	\begin{align}
	&
	\label{eq:HE<}
	\partial_t u_1 = \tfrac12 \partial_{xx} u_1 	\quad \forall 0< t< \tfrac12, \ x> 0,
\\
	&
	\label{eq:DiriBC<}
	u_1(t,0) =2,	\quad \forall 0< t<\tfrac12,		
\\
	&
	\label{eq:PDEic}
	\lim_{t\downarrow 0} (u_1(t,x) + 2 \partial_x p(t,x)) =0,	\quad\forall x\geq 0.
	\end{align}
\end{subequations}
For $ t >\frac12 $, we consider the following \textbf{Stefan problem},
a \ac{PDE} with a \emph{moving boundary}:
\begin{subequations}\label{eq:PDE>}
	\begin{align}
	&
	\label{eq:z2}
	z_2\in \Csp([\tfrac12,\infty)) \text{ nondecreasing},  \ z_2(\tfrac12)=0, 
	\\
	&
	\label{eq:HE>}
	\partial_t u_2 = \tfrac12 \partial_{xx} u_2, \quad \forall t > \tfrac12, \ x > z_2(t)
	\\
	&
	u_2(\tfrac12,x) = u_1(\tfrac12,x),	\quad	\forall x\geq 0,
	\\
	&
	\label{eq:DiriBC>}
	u_2(t,z_2(t)) = 2,	\quad \forall t\geq \tfrac12,
	\\
	&
	\label{eq:StefBC}
	2 \tfrac{d~}{dt}z_2(t) + \tfrac12 \partial_x u_2(t,z_2(t)) = 0,	\quad \forall t> \tfrac12.
	\end{align}
\end{subequations}
As we show in Lemma~\ref{lem:StefInt}, 
for each sufficiently smooth solution $ (u_2,z_2) $ to \eqref{eq:PDE>},
the functions $ \Ucs(t,x) := \int_{x }^\infty u_2(t,y) \ind_{\{ y\geq \z(t)\}} dy $ 
and $ \z(t) := z_2(t) $ satisfy \eqref{eq:Ucs:move}--\eqref{eq:zeq} for $ t \geq \frac12 $.

\begin{remark}\label{rmk:cal}
To see why \eqref{eq:U1} holds,
differentiate \eqref{eq:Ucs:abs} in $ x $ to obtain
$ \partial_x \Ucs(t,x) = 2\partial_x p(t,x) - 2 \int_0^t \frac{x}{t-s} p(t-s,x) ds $.
Within the last integral,
performing the change of variable $ y:= \frac{x}{\sqrt{t-s}} $,
we see that $ \int_0^t \frac{x}{t-s} p(t-s,x) ds = 2\erfc(t-s,x) $.
From this \eqref{eq:U1} follows.
\end{remark}

\begin{remark}\label{rmk:Stefan}
Note that for Equation~\eqref{eq:PDE>}
to make sense classically,
one needs $ u_2(t,x) $ to be $ \Csp^1 $ \emph{up to}
the boundary $ \{(t,z_2(t)):t\geq 0\} $ 
and needs $ z_2(t) $ to be $ \Csp^1 $.
Here, instead of defining the hydrodynamic limit
classically through \eqref{eq:PDE>},
we take the integral identity and integral equation 
\eqref{eq:Ucs:move}--\eqref{eq:zeq} 
as the \emph{definition} of the hydrodynamic limit equation.
This formulation is more convenient for our purpose,
and in particular it requires
neither the smoothness of $ \u $ onto the boundary
nor the smoothness of $ \z $.
We note that, however, it should be possible 
to establish classical solutions to \eqref{eq:PDE>},
by converting \eqref{eq:PDE>} to a parabolic variational inequality.
See, for example, \cite{friedman10}. We do not pursue this direction here.
\end{remark}

Before stating the precise result on hydrodynamic limit,
we explain the intuition of how \eqref{eq:PDE<}--\eqref{eq:PDE>}
arise from the behavior of the particle system.
Indeed, the heat equations \eqref{eq:HE<} and \eqref{eq:HE>}
model the diffusive behavior of $ (X^K_i(t))_i $ away from $ Z_K(t) $.
In view of the equilibrium measure of gaps of the infinite Atlas model \cite{pal08},
near $ Z_K(t) $ we expect the particle density to be $ 2 $ to balance the drift exerted on $ Z_K(t) $,
yielding the boundary conditions \eqref{eq:DiriBC<} and \eqref{eq:DiriBC>}.
The function $ -2\partial_x p(t,x) $ is the average density
of the system without the drift.
(The singularity of $ -2\partial_x p(t,x) $ at $ t=0 $ captures the overabundance 
of particles at $ t=0 $ compared to the scaling $ K^{1/2} $.)
As the drift affects little of the particle density near $ t=0 $,
we expect the entrance law~\eqref{eq:PDEic}.
The absorption phase ($ t\leq \frac12 $) describes
the initial state of the particle system with a high density,
where particles are constantly being absorbed,
yielding a fixed boundary $ Z_K(t) \approx 0 $.
Under the push-the-laggard strategy,
the system enters a new phase at $ t\approx \frac12 $,
where the density of particles is low enough ($ \leq 2 $ everywhere)
so that the drift carries all remaining particles away from $ 0 $.
This results in a moving boundary $ Z_K(t) $,
with an additional boundary condition~\eqref{eq:StefBC},
which simply paraphrases the conservation of particles
$ \frac{d~}{dt} \int_{z_2(t)}^\infty u_2(t,y) dy =0 $.

The following is our result on the hydrodynamic limit of $ (\Uc_K(t,x), Z_K(t)) $:
\begin{theorem}[hydrodynamic limit]
\label{thm:hydro}
Under the push-the-laggard strategy,
for any fixed $ \gamma\in(0,\frac{1}{96}) $ and $ T,n<\infty $,
there exists $ C=C(T,\gamma,n)<\infty $ such that
\begin{align}
	&
	\label{eq:hydro:U}
	\Pr \Big( 
		\sup_{ t\in[0,T], x\in\bbR } 
		\big\{ |\Uc_K(t,x)-\Ucs(t,x)| t^{\frac{3}{4}} \big\} \leq CK^{-\gamma} 
	\Big)
	\geq
	1 - CK^{-n},
	\quad
	\forall K<\infty,
\\
	&
	\label{eq:hydro:Z}
	\Pr \Big( \sup_{t\in[0,T]}| Z_K(t)-\z(t)| \leq CK^{-\gamma} \Big)
	\geq
	1 - CK^{-n},
	\quad
	\forall K<\infty.
	\end{align}
\end{theorem}
\begin{remark}
The factor of $ t^{\frac{3}{4}} $ in \eqref{eq:hydro:U} is in place 
to regulate the singularity of $ \Uc_K(t,x) $ and $ \Ucs(t,x) $ near $ t= 0 $.
Indeed, with $ \Ucs(t,x) $ defined in \eqref{eq:U1} for $ t\leq\frac12 $,
it is standard to verify that $ \sup_{x\in\bbR} \Ucs(t,x) $
diverges as $ \frac{2}{\sqrt{2\pi t}} $ as $ t\downarrow 0 $.
With $ \Uc_K(t,x) $ defined in \eqref{eq:Uc},
we have that $ \Uc_K(0,x) = \sqrt{K} \ind_{\{x<1/\sqrt{K}\}} $,
which diverges at $ x=0 $ as $ K\to\infty $.
This singularity at $ t=0 $ of $ \Uc_K(t,x) $ propagates into $ t>0 $,
resulting in a power law singularity of the form $ |t|^{-\frac12} $.

The choice of the exponent $ \frac34 $ in \eqref{eq:hydro:U} is technical,
and may be sharpened to $ \frac12 $ as discussed in the preceding,
but we do not pursue this direction here.
\end{remark}

Under the push-the-laggard strategy \eqref{eq:pushlaggard},
the process $ (X^K_i(t))_i $ is closely related 
to the Atlas model \cite{fernholz02}.
The latter is a simple special case of diffusions with rank-dependent drift:
see \cite{banner05,chatterjee10,chatterjee11,ichiba11,ichiba10,ichiba13},
for their ergodicity and sample path properties,
and \cite{dembo12,pal14} for their large deviations properties as the dimension tends to infinity.
In particular, the hydrodynamic limit and fluctuations
of the Atlas-type model have been analyzed 
in \cite{cabezas15, dembo15, hernandez15}.

Here we take one step further and analyze
the combined effect of rank-dependent drift and absorption,
whereby demonstrating the two-phase behavior.
With the absorption at $ x=0 $,
previous methods of analyzing the large scale behaviors 
of diffusions with rank-dependent drift do not apply.
In particular, the challenge of proving Theorem~\ref{thm:hydro}
originates from the lack of invariant measure (for the absorption phase)
and the singularity at $ t=0 $,
where a rapid transition from $ K $ particles
to an order of $ K^{1/2} $ particles occurs.
Here we solve the problem by adopting a \emph{new} method
of exploiting certain  integral identities of the particle system
that mimic \eqref{eq:Ucs:abs}--\eqref{eq:zeq}.
Even though here we mainly focus on
the push-the-laggard strategy under the initial condition $ X_i(0)=1 $, $ \forall i $,
the integral identities apply to general rank-dependent drifts and initial conditions,
and may be used for analyzing for general models with both rank-dependent drifts and absorption.

\subsection*{Outline}
In Section~\ref{sect:Int},
we develop certain integral identities of the particle system $ X $
that are crucial for our analysis,
and in Section~\ref{sect:Stef},
we establish the necessary tools pertaining to the integral equation~\eqref{eq:zeq}.
Based on results obtained in Sections~\ref{sect:Int}--\ref{sect:Stef},
in Sections~\ref{sect:hydro} and \ref{sect:aldous}
we prove Theorems~\ref{thm:hydro} and \ref{thm:aldous}, respectively.

\subsection*{Acknowledgment}
We thank David Aldous for suggesting this problem for research.
WT thanks Jim Pitman for helpful discussion throughout this work,
and Craig Evans for pointing out the relation between \eqref{eq:PDE>}
and parabolic variational inequalities.
LCT thanks Amir Dembo for enlightening discussion at the early stage of this work.
LCT was partially supported by the NSF through DMS-0709248.

We thank the anonymous reviewers for their careful reading of the manuscript.

\section{Integral Identities}
\label{sect:Int}
Recall that $ \pN(t,x,y) $ denotes the Neumann heat kernel,
and let $ \erf(t,x) := \Pr(B(t)\leq x) = 1 - \erfc(t,x) $ 
denote the Brownian distribution function.
With $ \z(t) $ as in \eqref{eq:z=0} and \eqref{eq:zeq},
we unify the integral identities \eqref{eq:Ucs:abs} and \eqref{eq:Ucs:move} 
into a single expression as
\begin{align}\label{eq:Ucs}
	\Ucs(t,x) = 2 p(t,x) + \int_0^t \pN(t-s,\z(s),x) ds,
	\quad
	\forall t >0, \ x \geq \z(t).
\end{align}
Essential to our proof of Theorems~\ref{thm:aldous} and \ref{thm:hydro}
are certain integral identities of the \emph{particle system} $ X=(X(t);t\geq 0) $
that mimic the integral identities~\eqref{eq:Ucs}.
This section is devoted to deriving such identities of the particle system,
particularly Proposition~\ref{prop:intXY} in the following.

As it turns out,
in addition to the particle system $ X $,
it is helpful to consider also the Atlas models.
We say that $ Y=(Y_i(t); t\geq 0 )_{i=1}^m $ is 
an \textbf{Atlas model} with $ m $ particles 
if it evolves according to the following system of stochastic differential equations:
\begin{equation} \label{eq:alt}
	dY_i(t)= \ind_\Set{ Y_i(t)=\W(t)} dt + dB_i(t) \quad 
	\text{ for } 1 \leq i \leq m,
	\quad
	\W(t) := \min\{ Y_i(t) \}.
\end{equation}
We similarly define 
the scaled processes $ \YKi(t) := \frac{1}{\sqrt{K}} Y_i(tK) $ 
and $ \W_K(t) := \frac{1}{\sqrt{K}} \W(tK) $.
Note that here $ K $ is just a scaling parameter, 
not necessarily related to the number of particles in $ Y $.

To state the first result of this section,
we first prepare some notations.
Define the scaled empirical measures of $ X $ and $ Y $ as:
\begin{align}
	\label{eq:em}
	\emK_{t}(\Cdot) 
	&:= 
	\frac{1}{\sqrt{K}} \sum\nolimits_{ \{i: X^K_i(t) >0\} } \delta_{X^K_i(t)}(\Cdot).
\\	
	\label{eq:emY}
	\emKY_{t}(\Cdot) 
	&:= \frac{1}{\sqrt{K}} \sum_{i} \delta_{\YKi(t)}(\Cdot).
\end{align}
For any fixed $ x\geq 0 $,
consider the tail distribution function 
$ \testf(t,y,x) := \Pr( B^\text{ab}_x(t) >y ) $, $ y>0 $,
of a Brownian motion $ B^\text{ab}_x $, 
starting at $ B^\text{ab}_x(0)=x $ and absorbed at $ 0 $.
More explicitly,
\begin{align}	
	\label{eq:testf}
	\testf(t,y,x) &:= \erf(t,y-x) - \erfc(t,y+x),
\end{align}
which is the unique solution to the following equation
\begin{subequations}
\label{eq:testf:Eq}
\begin{align}
	\label{eq:testf:HE}
	\partial_t \testf(t,y,x) &= \tfrac12 \partial_{yy} \testf(t,y,x), \forall t,y>0,
\\
	\testf(t,0,x) &= 0, \forall t >0,
\\
	\testf(0,y,x) &= \ind_{(x,\infty)}(y), \ \forall y>0.
\end{align}
\end{subequations}
Adopt the notations $ t_K:=t+\frac1K $, $ \tau^K_i := K^{-1}\tau_i $
and $ \phi^K_i(t) := \phi_i(Kt) $ hereafter.

\begin{lemma}
\label{lem:int}
\begin{enumerate}[label=(\alph*)]
\item[]
\item For the particle system $ (X(t);t\geq 0) $,
under any strategy, we have the following integral identity:
\begin{align}
\label{eq:int:abs:e}
\begin{split}
	\langle \emK_{t}, &\testf(\tfrac1K,\Cdot,x) \rangle 
	= 
	\Gc_K(t_K,x)
\\
	&+ \sum_{i=1}^K \int_{0}^{t} \phi^K_i(s) \pN(t_K-s,\XKi(s),x) ds
	+ \mg_{K}(t,x),
	\quad
	\forall t \in\bbR_+, \ x\geq 0,
\end{split}
\end{align}
where
\begin{align}
	\label{eq:Gc}
	\Gc_K(t,x) &:= 
	\sqrt{K} \testf(t,\tfrac{1}{\sqrt{K}},x),
\\
	\label{eq:mg}
	\mg_{K}(t,x) 
	&:= 
	\frac{1}{\sqrt{K}} \sum_{i=1}^K \int_{0}^{t\wedge\tKi} \pN(t_K-s,\XKi(s),x) d\BKi(s).
\end{align}
\item Let  $ (Y_i(t);t\geq 0)_{i} $ be an Altas model.
We have the following integral identity:
\begin{align}
\begin{split}
	\langle \emKY_{t}, &\erf(\tfrac{1}{K}, x-\Cdot) \rangle 
	= 
	\langle \emKY_{0}, \erf(t_K,x-\Cdot) \rangle 
\\
	\label{eq:int:atl:K}
	&
	- \int_{0}^{t} p(t_K-s,x-\W_K(s)) ds
	- \mgY_{K}(t,x),
	\quad
	\forall t\in\bbR_+, \ x\in\bbR,
\end{split}
\end{align}
where
\begin{align}
	\label{eq:mgY}
	\mgY_{K}(t,x) 
	&:= \frac{1}{\sqrt{K}} \sum_{i} \int_{0}^t p(t_K-s,Y^K_i(s)-x) dB^K_i(s).
\end{align}
\end{enumerate}
\end{lemma}

\begin{remark}
\label{rmk:int:meaning}
To motivate our analysis in the following,
here we explain the meaning of each term in the integral identity~\eqref{eq:int:abs:e}.
From the definitions \eqref{eq:Uc} and  \eqref{eq:em}
of $ \Uc_K(t,x) $ and $ \emK_{t} $,
we have that $ \lim_{\e\to 0 } \langle \emK_{t}, \testf(\e,\Cdot,x) \rangle = \Uc_K(t,x) $,
so it is reasonable to expect the term 
$ \langle \emK_{t}, \testf(\tfrac1K,\Cdot,x) \rangle $ on the l.h.s.\
to approximate $ \Uc_K(t,x) $ as $ K\to\infty $.

Next,
consider a system $ (X^{\text{ab}}_i(t);t\geq 0)_{i=1}^K $ of independent Brownian particles
starting at $ x=1 $ and absorbed at $ x=0 $, \emph{without} drifts.
Letting $ X^{\text{ab},K}_i(t) := \frac{1}{\sqrt{K}} X^{\text{ab}}_i(Kt) $ denote the diffusively scaled process, 
with the corresponding scaled tailed distribution function
\begin{align}
	\label{eq:Ucab}
	\Uc^{\text{ab}}_K(t,x) := \tfrac{1}{\sqrt{K}} \# \{ i: X^{\text{ab},K}_i(t) > x \},
\end{align}
it is standard to show that
\begin{align}
	\label{eq:Ucab:Gc}
	\Ex(\Uc^{\text{ab}}_K(t,x)) = \sqrt{K} \Pr( X^\text{ab}_1(Kt) > \sqrt{K} x ) 
	= \sqrt{K} \testf(Kt,1,\sqrt{K}x)  
	= \Gc_K(t,x). 
\end{align}
That is, the term $ \Gc_K(t,x) $ on the r.h.s.\ \eqref{eq:int:abs:e}
accounts for the contribution (in expectation) of the \emph{absorption}.

Subsequent, the time integral term
$ \sum_{i=1}^K \int_0^{\tKi}(\ldots) ds $ arises from the contribution
of the drifts $ (\phi_i(t))_{i=1}^K $ allocated to the particles,
while the martingale term $ \mg(t,x) $ encodes the random fluctuation
due to the Brownian nature of the particles.
\end{remark}

\begin{proof}
Under the diffusive scaling $ \XKi(t) := \frac{1}{\sqrt{K}} X_i(tK) $,
we rewrite the SDE~\eqref{eq:X:SDE} as
\begin{align}\label{eq:XK:SDE}
	d\XKi(t) = \phi^K_i(t) \sqrt{K} d(t\wedge\tKi) + d \BKi(t\wedge\tKi).
\end{align}
Fixing arbitrary $ t<\infty $, $ x\geq 0 $,
with $ \testf $ solving~\eqref{eq:testf:HE},
we apply It\^{o}'s formula to $ F_i(s) := \testf(t_K-s,\XKi(s),x) $ using \eqref{eq:XK:SDE} to obtain
\begin{align}\label{eq:int:abs:F}
	F_i(t\wedge\tKi) - F_i(0)
	=
	\sqrt{K} \int_{0}^{t\wedge\tKi} \phi^K_i(s) \pN(t_K-s,\XKi(s),x) ds
	+
	\mg_{i,K}(t,x),
\end{align}
where $ \mg_{i,K}(t,x) := \int_{0}^{t\wedge\tKi} \pN(t_K-s,\XKi(s),x) d\BKi(s) $.
With $ \testf(s,0,x)=0 $, we have $ F_i(t\wedge\tKi) = \testf(\frac1K,\XKi(t),x) $.
Using this in \eqref{eq:int:abs:F}, summing the result over $ i $, 
and dividing both sides by $ \sqrt{K} $,
we conclude the desired identity~\eqref{eq:int:abs:e}.
Similarly, the identity~\eqref{eq:int:atl}
follows by applying It\^{o}'s formula with the test function $ \erf(t_K-s,y-x) $.
\end{proof}

Based on the identities \eqref{eq:int:abs:e} and \eqref{eq:int:atl:K},
we proceed to establish bounds on the empirical measures $ \emK_{t} $ and $ \emKY_{t} $.
Hereafter, we use $ C=C(\alpha,\beta,\ldots)<\infty $ 
to denote a generic deterministic finite constant
that may change from line to line,
but depends only on the designated variables.
In the following, we will use the following estimates
of the heat kernel $ p(t,x) $.
The proof is standard and we omit it here.
\begin{align}
	&
	\label{eq:p:Holdx}
	|p(t,x) - p(t,x')| \leq C(\alpha) |x-x'|^{\alpha} t^{-\frac{1+\alpha}{2}},&  
	& \alpha\in(0,1],
\\
	&
	\label{eq:p:Holdt}
	|p(t,x) - p(t',x)| \leq C(\alpha) |t-t'|^{\frac{\alpha}{2}} (t')^{-\frac{1+\alpha}{2}},& & \alpha\in(0,1], \ t'<t<\infty.
\end{align}
We adopt the standard notations $ \Vert \xi \Vert_n := (\Ex|\xi|^n)^{\frac1n} $ 
for the $ L^n $-norm of a give random variable $ \xi $
and $ \norm f \norm_{L^\infty(\Omega)} := \sup_{x\in \Omega} |f(x)| $  
for the uniform norm over the designated region $ \Omega $.

\begin{lemma}\label{lem:emYbd}
Let $ (Y_i(t);t \geq 0)_{i} $ be an Atlas model.
The total number $ \#\{Y_i(0)\} $ of particles may be 
random but is independent of $ \sigma(Y_i(t)-Y_i(0);t\geq 0, i=1,\ldots) $.
Let $ \emKY_{t} $ to be as in \eqref{eq:emY}.
Assume $ (Y^K_i(0))_i $ satisfies the following initial condition:
given any $\alpha\in(0,1) $ and $ n <\infty $, there exist $ D_* ,D_{\alpha,n} <\infty $ such that
\begin{align}
	\label{eq:D*}
	\Pr\big( \#\{ Y_i(0) \} \leq K \big)
	&\geq 1 - \exp(-\tfrac{1}{D_*}K^{\frac12}),
\\
	\label{eq:Dan}
	\big\Vert \langle \emKY_{0},\ind_{[a,b]}\rangle \big\Vert_n	
	&\leq 
	D_{\alpha,n}|b-a|^{\alpha}, 
	\quad 
	\forall |b-a| \geq \tfrac{1}{\sqrt{K}}.
\end{align}	
For any given $ T<\infty $,
we have
\begin{align}
	\label{eq:emY:bd}
	&
	\Vert \langle \emKY_{s},\ind_{[a,b]}\rangle \Vert_{n} 
	\leq 
	C |b-a|^{\alpha} \Big( \big( \tfrac{|b-a|}{\sqrt{s_K}} \big)^{1-\alpha} +1 \Big), &
	&
	\forall \tfrac{1}{\sqrt{K}}\leq |b-a|,	\ s\leq T,	
\\
	\label{eq:emPY:bd}
	&
	\Vert \langle \emKY_{s}, p(t_K,\Cdot-x) \rangle \Vert_{n} 
	\leq 
	C t_K^{\frac{\alpha-1}{2}} \Big( \big( \tfrac{t_K}{s_K} \big)^{\frac{1-\alpha}{2}} +1 \Big), &
	&
	\forall x\in\bbR, \ s,t<T,
\end{align}
where $ C=C(T,\alpha,n,D_*, D_{\alpha,n})<\infty $.
\end{lemma}

\begin{proof}
Fixing such $ T,\alpha,n $ and $ [a,b] $,
throughout this proof we use $ C=C(T,\alpha,n,D_*, D_{\alpha,n})<\infty $
to denote a generic finite constant.
To the end of showing \eqref{eq:emY:bd},
we begin by estimating 
$ 
	\Vert\langle \emKY_{s}, \ind_{[a,b]} \rangle \Vert_1
	=\Ex(\langle \emKY_{s}, \ind_{[a,b]} \rangle ).
$
To this end, we set $ x=b,a $ in \eqref{eq:int:atl:K},
take the difference of the resulting equation,
and take expectations of the result to obtain
\begin{align}
	\label{eq:int:atl:Ex:}
	\Ex(\langle \emKY_{s}, \erf(\tfrac1K,b-\Cdot)-\erf(\tfrac1K,a-\Cdot) \rangle )
	=
	\Ex( J_1 ) + \Ex(J_2),
\end{align}
where
\begin{align*}
	J_1 &:= \langle \emKY_{0}, \erf(s_K,b-\Cdot)-\erf(s_K,a-\Cdot) \rangle,
\\
	J_2 &:= -\int_{0}^s \big( p(s_K-u,b-\W_K(u)) - p(s_K-u,a-\W_K(u)) \big) du.
\end{align*}
Further, with $ |b-a|\geq K^{-\frac12} $,
it is straightforward to verify that 
\begin{align}
	\label{eq:erf:ind}
	\erf(\tfrac1K,b-y)-\erf(\tfrac1K,a-y) \geq \tfrac{1}{C} \ind_{[a,b]}(y).
\end{align}
Combining \eqref{eq:erf:ind} and \eqref{eq:int:atl:Ex:} yields
\begin{align}
	\label{eq:int:atl:Ex}
	\Vert \langle \emKY_{s}, \ind_{[a,b]} \rangle \Vert_1
	\leq
	C\Ex( J_1 ) + C\Ex(J_2).
\end{align}

With \eqref{eq:int:atl:Ex},
our next step is to bound $ \Ex(J_1) $ and $ \Ex(J_2) $.
For the former, 
we use $ \erf(t_K,b-y)-\erf(t_K,a-y) = \int_{a}^{b} p(t_K,z-y) dz $
to write $ J_1 = \int_{a}^{b} \langle \emKY_0,p(s_K,x-\Cdot) \rangle dx $.
Taking the $ L^m $-norm of the last expression yields
\begin{align}
	\label{eq:J1:Ln:}
	\Vert J_1 \Vert_m 
	\leq 
	\int_{a}^{b} \big\Vert \langle \emKY_{0},p(s_K,x-\Cdot) \rangle \big\Vert_m dx,
	\quad
	\forall m \in\bbN.
\end{align}
Further,
as the heat kernel $ p(t,y-x) = \frac{1}{ \sqrt{t} } p(1,\frac{y-x}{\sqrt{t}}) $ 
decreases in $ |y-x| $, 
letting $ I_j(t,x) := x+[j\sqrt{t},(j+1)\sqrt{t}] $
and $ j_* := |j| \wedge |j+1| $,
we have
\begin{align}
\begin{split}
	\Vert \langle \emKY_{s}, p(t,\Cdot-x) \rangle \Vert_m
	&\leq
	\Big\Vert 
		\sum_{j\in\bbZ} \frac{1}{ \sqrt{t} } p(1,j_*) 
		\langle \emKY_{s}, \ind_{I_j(t,x)} \rangle 
	\Big\Vert_m
\\
	\label{eq:onion:Ln}	
	&\leq
	\sum_{j\in\bbZ} \frac{1}{ \sqrt{t} } p(1,j_*) 
	\Vert\langle \emKY_{s}, \ind_{I_j(t,x)} \rangle \Vert_m.
\end{split}
\end{align}
Set $ m=n $, $ s=0 $ and $ t=s_K $ in \eqref{eq:onion:Ln}.
Then, for each $ j $-th term within the sum,
use \eqref{eq:Dan} to bound 
$ \Vert \langle \emKY_{0}, \ind_{I_j(s_K,x)} \rangle  \Vert_n \leq C|\sqrt{s_K}|^{\alpha} $,
followed by using $ \sum_{j} p(1,j_*) <\infty $.
This yields
\begin{align}
	\label{eq:onion:Ln:}	
	\Vert \langle \emKY_{0}, p(s_K,\Cdot-x) \rangle \Vert_n
	\leq
	C s_K^{\frac{\alpha-1}{2}}.
\end{align}
Inserting~\eqref{eq:onion:Ln:} into \eqref{eq:J1:Ln:}, we then obtain
\begin{align}
	\label{eq:J1:Ln}
	\Vert J_1 \Vert_n
	\leq 
	\int_{a}^{b} C s_K^{\frac{\alpha-1}{2}} dx
	\leq
	C |b-a| s_K^{\frac{\alpha-1}{2}}.
\end{align}
As for $ J_2 $, by \eqref{eq:p:Holdx} we have
\begin{align}
	\label{eq:J2:bd}
	|J_2| \leq C \int_{0}^{s} |b-a|^{\alpha}(u_K)^{-\frac{1+\alpha}{2}} du
	\leq
	C |b-a|^{\alpha}.
\end{align}
Inserting \eqref{eq:J1:Ln}--\eqref{eq:J2:bd} in \eqref{eq:int:atl:Ex},
we see that \eqref{eq:emY:bd} holds for $ n=1 $.

To progress to $ n>1 $,
we use induction,
and assume \eqref{eq:emY:bd} has been established for an index $ m\in[1,n) $.
To setup the induction,
similarly to the proceeding,
we set $ x=b,a $ in \eqref{eq:int:atl:K},
take the difference of the resulting equation,
and take the $ L^{m+1} $-norm of the result to obtain
\begin{align*}
	\Vert \langle \emKY_{s}, \erf(\tfrac1K,b-\Cdot)-\erf(\tfrac1K,a-\Cdot) \rangle \Vert_{m+1}
	\leq
	\Vert J_1 \Vert_{m+1} + \Vert J_2 \Vert_{m+1} + \Vert J_3 \Vert_{m+1},
\end{align*}
where $ J_3 :=\mgY_K(s,b) - \mgY_K(s,a) $.
Further combining this with \eqref{eq:erf:ind} yields
\begin{align}
	\label{eq:atl:induc}
	\Vert \langle \emKY_{s}, \ind_{[a,b]} \rangle \Vert_{m+1}
	\leq
	C \Vert J_1 \Vert_{m+1} + C \Vert J_2 \Vert_{m+1} + C \Vert J_3 \Vert_{m+1}.
\end{align}
For $ \Vert J_1 \Vert_{m+1} $ and $ \Vert J_2 \Vert_{m+1} $
we have already established the bounds~\eqref{eq:J1:Ln}--\eqref{eq:J2:bd},
so it suffices to bound $ \Vert J_3 \Vert_{m+1} $.
As $ J_3 $ is a martingale integral of quadratic variation
$
	\frac{1}{\sqrt{K}} \int_0^{s} \langle \emKY_{u}, \hat{p}^2(u,\Cdot) \rangle du,
$
where $ \hat{p}(u,y) := p(s_K-u,a-y) - p(s_K-u,b-y) $,
we applying the~\ac{BDG} inequality to obtain
\begin{align}
	\label{eq:J3:bd1}
	\Vert J_3 \Vert^2_{m+1}
	\leq
	\frac{C}{\sqrt{K}} 
	\int_{0}^{s} 
	\Vert \langle \emKY_{u}, \hat{p}^2(u,\Cdot) \rangle \Vert_{\frac{m+1}{2}} du.
\end{align}
The induction hypothesis asserts the bound
\eqref{eq:emY:bd} for $ n=m $. 
With this in mind, 
within the integral in~\eqref{eq:J3:bd1},
we use $ \frac{m+1}{2} \leq m $ to bound the $ \Vert\Cdot\Vert_{\frac{m+1}{2}} $ norm
by the $ \Vert\Cdot\Vert_m $ norm, and write
\begin{align}
	\label{eq:J3:bd2}
	\Vert \langle \emKY_{u}, \hat{p}^2(u,\Cdot) \rangle \Vert_{\frac{m+1}{2}}
	\leq
	\Vert \langle \emKY_{u}, \hat{p}^2(u,\Cdot) \rangle \Vert_{m}
	\leq
	\norm \hat{p}(u,\Cdot) \norm_{L^\infty(\bbR)}
	\Vert \langle \emKY_{u}, \hat{p}(u,\Cdot) \rangle \Vert_{m}.
\end{align}
To bound the factor $ \norm \hat{p}(u,\Cdot) \norm_{L^\infty(\bbR)} $
on the r.h.s.\ of~\eqref{eq:J3:bd2},
fixing $ (2\alpha-1)_+<\beta<\alpha $,
we use~\eqref{eq:p:Holdx} to write
\begin{align}
	\label{eq:J3:bd3}
	\norm \hat{p}(u,\Cdot) \norm_{L^\infty(\bbR)} 
	\leq C |b-a|^{\beta} (s_K-u)^{-\frac{1+\beta}{2}}.
\end{align}
Now, within the r.h.s.\ of~\eqref{eq:J3:bd2},
using \eqref{eq:J3:bd3},
\begin{align*}
	|\langle \emKY_{u}, \hat{p}(u,\Cdot) \rangle| 
	\leq  
	\langle \emKY_{u}, p(s_K-u,b-\Cdot) \rangle 
	+ \langle \emKY_{u}, p(s_K-u,a-\Cdot) \rangle
\end{align*}
and \eqref{eq:onion:Ln}, we obtain
\begin{align}
	\notag
	\Vert \langle \emKY_{u}, & \hat{p}^2(u,\Cdot) \rangle \Vert_{\frac{m+1}{2}}
	\leq
	C |b-a|^{\beta} (s_K-u)^{-\frac{1+\beta}{2}}
\\
	\label{eq:J3:bd4}
	&\sum_{j\in\bbZ} \frac{1}{ \sqrt{s_K-u} } p(1,j_*) 
	\Big(
		\sum_{x\in a,b} 
		\Vert \langle \emKY_{u}, \ind_{I_j(s_K-u,x)} \rangle \Vert_{m}
	 \Big).
\end{align}
By the induction hypothesis,
$ 
	\Vert \langle \emKY_{u}, \ind_{I_j(s_K-u,x)} \rangle \Vert_{m}
	\leq 
	C (\sqrt{s_K-u})^{\alpha} ( (\frac{\sqrt{s_K-u}}{\sqrt{u_K}})^{1-\alpha}+1 ). 
$
Using this for $ x=a,b $ in \eqref{eq:J3:bd4}, 
and combining the result with \eqref{eq:J3:bd1},
followed by $ \sum_{j\in\bbZ} p(1,j_*)\leq C $,
we obtain
\begin{align}
	\label{eq:J3:bd5}
	\Vert \langle \emKY_{u}, \hat{p}^2(u,\Cdot) \rangle \Vert_{\frac{m+1}{2}}
	\leq
	C |b-a|^{\beta}
	\Big(
		(s_K-u)^{-\frac{1+\beta}{2}} {u_K}^{\frac{\alpha-1}{2}} 
		+ (s_K-u)^{-1+\frac{\alpha-\beta}{2}} 
	\Big).
\end{align}
Inserting this bound~\eqref{eq:J3:bd5} back into \eqref{eq:J3:bd1},
followed by using
$ \frac{1}{\sqrt{K}} \leq K^{-(2\alpha-\beta)/2} \leq |b-a|^{2\alpha-\beta} $,
we arrive at
\begin{align*}
	\Vert J_3 \Vert^2_{m+1}
	\leq
	C  |b-a|^{2\alpha}
	\int_0^{s} 
	\big( 
		(s_K-u)^{-\frac{1+\beta}{2}} {u_K}^{\frac{\alpha-1}{2}} 
		+ (s_K-u)^{-1+\frac{\alpha-\beta}{2}} 
	\big)
	du.
\end{align*}
Within the last expression,
using the readily verify inequality:
\begin{align}
	\label{eq:int:power}
	\int_0^s (s_K-u)^{-\delta_1} {u_K}^{-\delta_2} du \leq C(\delta_1,\delta_2) s_K^{1-\delta_1-\delta_2},
	\quad
	\forall \delta_1,\delta_2<1,
\end{align}
we obtain 
$ 	
	\Vert J_3 \Vert^2_{m+1}
	\leq
	C |b-a|^{2\alpha} s_K^{\frac{\alpha-\beta}{2}}
	\leq
	C |b-a|^{2\alpha}. 
$
Using this bound and the bounds \eqref{eq:J1:Ln}--\eqref{eq:J2:bd} in \eqref{eq:atl:induc},
we see that \eqref{eq:emY:bd} holds for the index $ m+1 $.
This completes the induction and hence concludes~\eqref{eq:emY:bd}.

The bound~\eqref{eq:emPY:bd} follows by 
combining \eqref{eq:onion:Ln} and \eqref{eq:emY:bd}.
\end{proof}

Next we establish bounds on $ \emK_{t} $.
\begin{lemma}
\label{lem:embd}
Let $ n,T<\infty $ and $ \alpha\in(0,1) $.
Given any strategy,
\begin{align}
	\label{eq:em:bd}
	&\Vert \langle \emK_{s}, \ind_{[a,b]} \rangle \Vert_{n} 
	\leq 
	C |b-a|^{ \alpha } s_K^{ -\frac{1+\alpha}{2} },
	\quad
	\forall [a,b]\subset\bbR_+ \text{ with } |b-a|\geq \tfrac{1}{\sqrt{K}},	
	\ \forall s\leq T,
\\
	\label{eq:emP:bd}
	&\Vert \langle \emK_{s}, p(t_K,\Cdot-x) \rangle \Vert_{n} 
	\leq 
	C t_K^{-\frac{1-\alpha}{2}} s_K^{-\frac{1+\alpha}{2}} ,
	\quad
	\forall x\in\bbR, \ s,t\leq T,
\end{align}
where $ C=C(T,\alpha,n)<\infty $, which, in particular, is independent of the strategy.
\end{lemma}

\begin{proof}
With $ \testf(\tfrac1K,y,x) $ defined in the proceeding,
it is straightforward to verify that 
\begin{align}
	\label{eq:testf:ind}
	\tfrac1C \ind_{[a,b]}(y) \leq \testf(\tfrac1K,y,a)-\testf(\tfrac1K,y,b),
	\quad
	\forall [a,b] \subset [\tfrac1K,\infty),
	\text{ satisfying } |b-a|\geq \tfrac{1}{\sqrt{K}}.
\end{align}
The idea of the proof is to follow the same general strategy 
as in the proof of Lemma~\ref{lem:emYbd}.
However, unlike \eqref{eq:erf:ind},
here the inequality~\eqref{eq:testf:ind} does \emph{not} hold 
for all desired interval $ [a,b]\subset\bbR_+ $,
but only for $ [a,b] \subset [\tfrac1K,\infty) $.
This is due to the fact that $ \testf(t,0,x)=0 $.
To circumvent the problem,
we consider the \emph{shifted} process
$ (\sfX^{m}_i(t); t \geq 0)_{i=1}^K $
\begin{align*}
	\sfX^m_i(t) = 1+n + B_i(t\wedge\sigma^n_i) + \int_0^{t\wedge\sigma^n_i} \phi_i(s) ds,
	\quad
	\text{where } \sigma^n_i := \inf\{t : \sfX^m_i(t)=0 \}.
\end{align*}
That is, $ \sfX^{m}_i(t) $, $ i=1,\ldots,K $, are driven by the same Brownian motion as $ X(t) $,
drifted under \emph{the same strategy} $ \phi(t) $ as $ X(t) $,
and absorbed at $ x=0 $,
but started at $ x=m+1 $ instead of $ x=1 $.
Define the analogous scaled variables as
$ \sfX^{K,m}_i(t) := \frac{1}{\sqrt{K}} \sfX^m_i(Kt) $,
$ \sigma^{K,m}_i := K^{-1}\sigma^m_i $,
\begin{align*}
	\mu^{K,m}_t(\Cdot)
	:= 
	\frac{1}{\sqrt{K}} \sum_{\sfX^{K,m}_i(t)>0} \delta_{\sfX^{K,m}(t)}(\Cdot).
\end{align*}
We adopt the convention that $ \sfX^{K,0}_i(t) = \XKi(t) $ and $ \mu^{K,0}_t = \emK_t $.
Under the proceeding construction, we clearly have that 
$ \sfX^{m-1}_i(t) = \sfX^m_i(t)-1 $, $ \forall t\leq\sigma^{m-1}_i $,
so in particular
\begin{align}
	\label{eq:mu:cmp}
	\mu^{K,m-1}_t([a-\tfrac{1}{\sqrt{K}},b-\tfrac{1}{\sqrt{K}}]) \leq \mu^{K,m}_t([a,b]),
	\quad
	\forall [a,b] \subset \bbR_+.
\end{align}
For the shifted process $ \sfX^{K,m}(t) = (\sfX^{K,m}_i(t))_{i=1}^K $,
by the same procedure of deriving \eqref{eq:int:abs:e},
we have the following integral identity:
\begin{align}
\label{eq:int:abs:e'}
\begin{split}
	\langle \mu^{K,m}_{t}, &\testf(\tfrac1K,\Cdot,x) \rangle 
	= 
	\Gc_{K,m}(t_K,x)
\\
	&+ \sum_{i=1}^K \int_{0}^{t} \phi^K_i(s) \pN(t_K-s,\sfX^{K,m}_i(s),x) ds
	+ \mg_{K,n}(t,x),
	\quad
	\forall t \in\bbR_+, \ x\geq 0,
\end{split}
\end{align}
where
\begin{align}
	\label{eq:Gcn}
	\Gc_{K,m}(t,x) &:= 
	\sqrt{K} \testf(t,\tfrac{1+m}{\sqrt{K}},x),
\\
	\label{eq:mgn}
	\mg_{K,m}(t,x) 
	&:= 
	\frac{1}{\sqrt{K}} \sum_{i=1}^K \int_{0}^{t\wedge\tKi} \pN(t_K-s,\sfX^{K,m}_i(s),x) d\BKi(s).
\end{align}

Having prepared the necessary notations,
we now begin the proof of \eqref{eq:em:bd}.
Instead of proving~\eqref{eq:em:bd} directly,
we show
\begin{align}
\begin{split}
	\label{eq:em:bd:m}
	\Vert \langle \mu^{K,n-m+1}_{s}, \ind_{[a,b]} \rangle \Vert_{m} 
	&\leq 
	C |b-a|^{ \alpha } s_K^{ -\frac{1+\alpha}{2} },
\\
	&\forall [a,b]\subset[\tfrac{1}{\sqrt{K}},\infty) \text{ with } |b-a|\geq \tfrac{1}{\sqrt{K}},	
	\ \forall s\leq T,
\end{split}
\end{align}
for all $ m=1,\ldots,n $.
Once this is established, 
combining \eqref{eq:em:bd:m} for $ m=n $
and \eqref{eq:mu:cmp} for $ m=1 $,
the desired result~\eqref{eq:em:bd} follows.

Fixing $ [a,b]\subset[\frac1{\sqrt{K}},\infty) $ with $ |b-a|\geq \tfrac{1}{\sqrt{K}} $.
We begin by settling \eqref{eq:em:bd:m} for $ m=1 $.
Similarly to the procedure for obtaining \eqref{eq:int:atl:Ex},
using \eqref{eq:int:abs:e'} for $ m=n $ 
and \eqref{eq:testf:ind} in place of \eqref{eq:int:atl:K} and \eqref{eq:erf:ind},
respectively, here we have
\begin{align}
	\label{eq:int:abs:Ex}
	\Vert \langle \mu^{K,n}_{s}, \ind_{[a,b]} \rangle \Vert_1
	\leq
	C \sfJ^n_{1} + C\Ex(\sfJ^n_{2}),
\end{align}
where $ \sfJ^m_{1},\sfJ^m_{2} $ is defined for $ m=1,\ldots, n $ as
\begin{align*}
	\sfJ^m_{1} &:= \Gc_{K,m}(s_K,a)-\Gc_{K,m}(s_K,b),
\\
	\sfJ^m_{2} &:= \sum_{i=1}^K \int_0^{\sigma^{K,n}_i\wedge s} \phi^K_i(u) 
		\big( \pN(s_K-u,\sfX^{K,m}(u),a)-\pN(s_K-u,\sfX^{K,m}_i(u),b) \big) du.
\end{align*}
As noted in Remark~\ref{rmk:int:meaning},
expressions of the type $ \sfJ^m_{1} $ account 
for the contribution of the system with only absorption,
while $ \sfJ^m_2 $ encodes the contribution of the drifts $ \phi^K_i(s) ds $.
The singular behavior of the empirical measure $ \mu^{K,m}_s $ at $ s=0 $ 
(due to having $ K \gg \sqrt{K} $ particles) is entirely encoded in $ \sfJ^m_1 $.
In particular,
recalling from \eqref{eq:Ucab} the notation $ \Uc^\text{ab}_K(t,x) $,
by \eqref{eq:Ucab:Gc} we have
\begin{align*}
	\Gc_{K,m}(0,a)-\Gc_{K,m}(0,b) 
	= \tfrac{1}{\sqrt{K}} \# \{ X^{\text{ab},K}_i(0)+m\in (a,b] \} 
	= \sqrt{K} \ind_{\frac{1+m}{\sqrt{K}}\in(a,b]}.
\end{align*}
While this expression diverges (for $ a<\frac{1+m}{\sqrt{K}} $) as $ K\to\infty $,
for any fixed $ s>0 $
the absorption mechanism remedies the divergence,
resulting in converging expression for the fixed $ s>0 $.
To see this, with $ \Gc_K(t,x) $ defined in \eqref{eq:Gc},
we use $ \partial_y\testf(s,y,x)=\pN(s,y,x) $ and $ \testf(s,0,x)=0 $
to write 
\begin{align}
	\label{eq:Gc:pNm}
	\Gc_{K,m}(t,x) 
	= \sqrt{K} \int_{0}^{\frac{1+m}{\sqrt{K}}}
	\pN(s,y,x) dy.			
\end{align}
Letting $ x=a,b $ in \eqref{eq:Gc:pN},
taking the difference of the resulting equations,
followed by applying the estimate \eqref{eq:p:Holdx},
we obtain the following bound of $ \sfJ^m_1 $,
which stays bounded as $ K\to\infty $ for any fixed $ s>0 $: 
\begin{align}
	\notag
	\sfJ^m_1 
	&= 
	\sqrt{K} \int_{0}^{\frac{1+m}{\sqrt{K}}}
	\big( \pN(s_K,y,a)-\pN(s_K,y,b)\big) dy
	\leq
	C \sqrt{K} \int_{0}^{\frac{1+m}{\sqrt{K}}} s_K^{-\frac{1+\alpha}{2}} |b-a|^{\alpha} dy
\\
	\label{eq:J1':bd}
	&\leq
	C s_K^{-\frac{1+\alpha}{2}} |b-a|^{\alpha},
	\quad
	\forall m=1,\ldots, n.
\end{align}
As for $ \sfJ^m_2 $, similarly to \eqref{eq:J2:bd},
by using \eqref{eq:p:Holdx} and $ \sum_{i=1}^K \phi^K_i(s) \leq 1 $ here we have
\begin{align}
	\label{eq:J2':bd}
	|\sfJ^m_2| \leq 
	\sum_{i=1}^K \int_0^{\sigma^{K,n}_i\wedge s} \phi^K_i(u)
	C|b-a|^{\alpha}(u_K)^{-\frac{1+\alpha}{2}} du
	\leq
	C |b-a|^{\alpha},
	\quad
	\forall m=1,\ldots, n.	
\end{align}
Combining \eqref{eq:J1':bd}--\eqref{eq:J2':bd} with \eqref{eq:int:abs:Ex},
we conclude \eqref{eq:em:bd:m} for $ m=1 $.

Having establishing \eqref{eq:em:bd:m} for $ m=1 $,
we use induction to progress,
and assume \eqref{eq:em:bd:m} has been established for some index $ m\in[1,n) $.
Similarly to \eqref{eq:atl:induc}, here we have
\begin{align}
	\label{eq:abs:induc}
	\Vert \langle \mu^{K,n-m}_{s}, \ind_{[a,b]} \rangle \Vert_{m+1}
	\leq
	C \sfJ^{n-m}_1 + C \Vert \sfJ^{n-m}_2 \Vert_{n-m} + C \Vert \sfJ^{n-m}_3 \Vert_{m+1}.
\end{align}
where $ \sfJ^{m}_3 :=\mg_{K,m}(s,b) - \mg_{K,m}(s,a) $.
To bound $ \Vert \sfJ^{n-m}_3 \Vert_{m+1} $,
similarly to \eqref{eq:J3:bd1}--\eqref{eq:J3:bd2},
by using the~\ac{BDG} inequality here we have
\begin{align}
	\label{eq:J3':bd1}
	\Vert \sfJ^{n-m}_3 \Vert^2_{m+1}
	&\leq
	\frac{C}{\sqrt{K}} 
	\int_{0}^{s} 
	\Vert \langle \mu^{K,n-m}_{u}, \hatpN(u,\Cdot)^2 \rangle \Vert_{\frac{m+1}{2}} du,
\end{align}
where 
$ \hatpN(u,y) := \pN(s_K-u,a,y)-\pN(s_K-u,a,y) $.
For the expression 
$ \Vert \langle \mu^{K,m+1}_{u}, \hatpN(u,\Cdot)^2 \rangle \Vert_{\frac{m+1}{2}} $,
following the same calculations in \eqref{eq:J3:bd2}--\eqref{eq:J3:bd4},
here we have
\begin{align}
	\notag
	\Vert \langle \mu^{K,n-m}_{u}, & \hatpN(u,\Cdot)^2 \rangle \Vert_{\frac{m+1}{2}}
	\leq
	C |b-a|^{\beta} (s_K-u)^{-\frac{1+\beta}{2}}
\\
	\label{eq:J3':bd4}
	&\sum_{j\in\bbZ} \frac{1}{ \sqrt{s_K-u} } p(1,j_*) 
	\Big( 
		\sum_{x=\pm a,\pm b} 
		\Vert \langle \mu^{K,n-m}_{u}, \ind_{I'_j(x)}\rangle \Vert_{m}  \Big),
\end{align}
where $ \beta \in ((2\alpha-1)_+,\alpha) $ is fixed,
$ j_*:=|j|\wedge|j+1| $ and
\begin{align*}
	I'_j(x):= I_j(\sqrt{s_K-u},x) = [x+j\sqrt{s_K-u},x+(j+1)\sqrt{s_K-u}].
\end{align*}
Since the empirical measure~$ \mu^{K,m+1}_{t} $ is supported on $ \bbR_+ $,
letting\\ 
$ I''_j(x):=[(x+j\sqrt{s_K-u})_+,(x+j\sqrt{s_K-u})_++\sqrt{s_K-u}] $,
we have
\begin{align*}
	\langle \mu^{K,n-m}_{u}, \ind_{I'_j(x)}\rangle 
	= 
	\langle \mu^{K,n-m}, \ind_{I'_j(x)\cap\bbR_+} \rangle
	\leq
	\langle \mu^{K,n-m}_{u}, \ind_{I''_j(x)}\rangle.
\end{align*}
Further using \eqref{eq:mu:cmp} yields
\begin{align}
	\label{eq:em:II'}
	\langle \mu^{K,n-m}_{u}, \ind_{I'_j(x)}\rangle 
	\leq
	\langle \mu^{K,n-m+1}_{u}, \ind_{\frac{1}{\sqrt{K}}+I''_j(x)}\rangle.
\end{align}
By the induction hypothesis,
\begin{align}
	\label{eq:em:induc}
	\Vert \langle \mu^{K,n-m+1}_{u}, \ind_{\frac{1}{\sqrt{K}}+I''_j(x)} \rangle\Vert_{m} 
	\leq 
	C (\sqrt{s_K-u})^{\alpha} (u_K)^{-\frac{1+\alpha}{2}}. 
\end{align}
Using~\eqref{eq:em:induc} for $ x=\pm a,\pm b $ in \eqref{eq:J3':bd4}, 
and combining the result with \eqref{eq:J3':bd1}, we arrive at
\begin{align*}
	\Vert \sfJ^{n-m}_3 \Vert^2_{m+1}
	\leq
	C  \frac{|b-a|^{\beta}}{\sqrt{K}}
	\int_0^{s} 
	(s_K-u)^{\frac{-2-\beta+\alpha}{2}} {u_K}^{\frac{-1-\alpha}{2}} 
	du.
\end{align*}
Within the last expression,
further using
$ \frac{1}{\sqrt{K}} \leq K^{-(2\alpha-\beta)/2} \leq |b-a|^{2\alpha-\beta} $
and \eqref{eq:int:power},
we obtain 
$ 	
	\Vert \sfJ^{n-m}_3 \Vert^2_{m+1}
	\leq
	C |b-a|^{2\alpha} s_K^{-\frac{1+\beta}{2}}
	\leq
	C |b-a|^{2\alpha} s_K^{-\frac{1+\alpha}{2}}. 
$
Using this bound and the bounds \eqref{eq:J1':bd}--\eqref{eq:J2':bd} in \eqref{eq:abs:induc},
we see that \eqref{eq:em:bd} holds for the index $ m+1 $.
This completes the induction and hence concludes~\eqref{eq:em:bd}.

The bound~\eqref{eq:emP:bd} follows by \eqref{eq:em:bd} and \eqref{eq:J3':bd4}.
\end{proof}

Lemmas~\ref{lem:emYbd}--\ref{lem:embd} establish bounds
that are `pointwise' in time,
in the sense that they hold at a fixed time $ s $ within the relevant interval.
We next improve these pointwise bounds to bounds
that hold for \emph{all} time within a relevant interval.
\begin{lemma}\label{lem:em}
Let $ T,L,n<\infty $ and
\begin{align}\label{eq:Ixa}
	I_{x,\alpha} := [-K^{-\alpha}+x,x+K^{-\alpha}].
\end{align}
	\begin{enumerate}[label=(\alph*)]
	\item\label{enu:em}
	For any given $ \gamma\in(0,\frac14) $, $ \alpha\in(2\gamma,\frac12] $ 
	and any strategy,
	\begin{align}
		\label{eq:em:bdT}
		\Pr \Big(  
			\langle \emK_{t}, \ind_{I_{x,\alpha}} \rangle 
			\leq 
			t^{-\frac34} K^{-\gamma},
			\ 
			\forall t \leq T, |x| \leq L  
		\Big) 
		\geq 1 - CK^{-n},
	\end{align}
	where $ C=C(T,L,\alpha,\gamma,n)<\infty $,
	which, in particular, is independent of the strategy.
	\item\label{enu:emY}
	Letting $ \emKY_{t} $ be as in Lemma~\ref{lem:emYbd}, we have,
	for any $ \alpha\in(\frac14,\frac12] $,
	\begin{align}\label{eq:emY:bdT}
		\Pr \Big(  
			\langle \emKY_{t}, \ind_{I_{x,\alpha}}\rangle \leq K^{-\frac14},
			\ 
			\forall t \leq T, |x| \leq L  
		\Big) 
		\geq 1 - CK^{-n},
	\end{align}
	where $ C=C(T,L,\alpha,n,D_*, D_{\alpha,n})<\infty $,
	for $ D_*, D_{\alpha,n} $ as in \eqref{eq:D*}--\eqref{eq:Dan}.
	\end{enumerate}
\end{lemma}

\begin{proof}
We first prove Part~\ref{enu:emY}. 
Fixing $ L,T,n<\infty $ and $ \alpha\in(\frac14,\frac12] $,
to simplify notations
we use, $ C(a_1,a_2,\ldots)<\infty $ to denote generic finite constants that may depend on 
$ L,T, \alpha, n, D_*, D_{\alpha,j} $ and the designated variable $ a_1,a_2,\ldots $. 
To the end of proving \eqref{eq:emY:bdT}, we cover $ [-L,L] $ 
by intervals $ I_j $ of length $ K^{-\alpha} $:
\begin{align*}
	I_j := [jK^{-\alpha}, (j+1)K^{-\alpha}],
	\quad
	|j| \leq LK^{\alpha}.
\end{align*}
Indeed, each $ I_{x,\alpha} $ is contained 
in the union of three consecutive such intervals $ I_j $,
so it suffices to prove
\begin{align}\label{eq:emY:goal}
	\Pr \big( 
		\langle \emKY_{t}, \ind_{I_{j}}\rangle 
		\leq 
		\tfrac13 K^{-\frac14}, 
		\ 
		\forall |j|\leq LK^{\alpha},  t\leq T
	\big) 
	\geq 
	1- C K^{-n}.
\end{align}
By \eqref{eq:emY:bd} we have, 
for any $ t\in [\frac1K,T] $, $ \beta\in(0,1) $ and $ k<\infty $,
\begin{align}\label{eq:emY:norm}
	\Vert \langle \emKY_{t}, \ind_{I_{j}}\rangle \Vert_{k} 
	\leq 
	C(k,\beta) |I_j|^{\beta} \big( (|I_j|K^{\frac12})^{1-\beta} + 1 \big)
	\leq
	C(k,\beta) (K^{-\alpha+\frac{1-\beta}{2}}+ K^{-\alpha\beta}).
\end{align}
With $ \alpha > \frac14 $, 
fixing $ \beta $ close enough to $ 1 $ 
we have $ \Vert \langle \emKY_{t}, \ind_{I_{x,\alpha}}\rangle \Vert_{k} \leq C(k) K^{-\frac14-\e} $,
for some fixed $ \e >0 $.
With this, applying Markov's inequality we obtain
\begin{align}\label{eq:Mrk:ineq}
	\Pr( \langle \emKY_{t}, \ind_{I_{j}}\rangle \geq \tfrac19 K^{-\frac14} ) \leq C(k) K^{-k\e}.
\end{align}
Now, fixing $ k \geq (n+\alpha+2)\e^{-1} $ and
taking the union bound of \eqref{eq:Mrk:ineq} 
over $ |j| \leq LK^{\alpha} $
and $ t= t_{\ell} := \ell K^{-2} $, $ 1\leq \ell \leq TK^2 $, 
we arrive at
\begin{align}\label{eq:emY:goal:}
	\Pr( 
		\langle \emKY_{t_\ell}, \ind_{I_{j}}\rangle \leq \tfrac19 K^{-\frac14}, 
		\ \forall |j|\leq LK^{\alpha}, 1\leq \ell\leq TK^{2} 
		) 
	\geq 
	1- C K^{-n}.
\end{align}
To move from the `discrete time' $ t_\ell $ to `continuous time' $ t\in[0,T] $,
we need to control $ \emKY_{s}(I_j) $
within each time interval $ s\in [t_{\ell-1},t_{\ell}] := J_\ell $.
Within each $ J_\ell $, 
since each $ \YKi(s) $ evolves as a drifted Brownian motion with drift $ \leq \sqrt{K} $,
we have that 
\begin{align} \label{eq:driftBMest}
	\Pr( |\YKi(s)-\YKi(t_\ell)| \leq K^{-\alpha}, \forall s\in J_{\ell} )
	\geq 
	1- \exp(-\tfrac{1}{C} K^{1-\alpha})
	\geq
	1- C K^{-n-3}.
\end{align}
By \eqref{eq:D*}, 
we assume without lost of generality the total number of $ Y $-particles
is at most $ K $.
Hence, taking the union bound of \eqref{eq:driftBMest} over $ \ell\leq TK^{-2} $ 
and over all particles $ i = 1,2,\ldots \leq K $,
we obtain
\begin{align*}
	\Pr\Big( 
		\sup_{s\in J_\ell} |\YKi(s)-\YKi(t_\ell)| 
		\leq 
		K^{-\alpha}, \ \forall i, \forall 1 \leq \ell \leq TK^{2} 
		\Big)
	\geq 
	1 - C K^{-n}.
\end{align*}
That is, with high probability, no particle travels farther 
than distance $ |I_j| $ within each time interval $ J_{\ell} $.
Therefore,
\begin{align*}
	\Pr\Big( 
		\sup_{s\in J_\ell} \langle \emKY_{s}, \ind_{I_{j}}\rangle 
		\leq 
		\langle \emKY_{t_\ell}, \ind_{I_{j-1}\cup I_{j}\cup I_{j+1}}\rangle, 
		\  
		\forall 1 \leq \ell \leq TK^{2} 
	\Big)
	\geq 
	1 - C K^{-n}.
\end{align*}
Combining this with \eqref{eq:emY:goal:} 
yields the desired result~\eqref{eq:emY:goal}.

Part~\ref{enu:em} is proven by similar argument as in the preceding.
The only difference is that, instead of a moment bound of the form \eqref{eq:emY:norm},
we have from \eqref{eq:em:bd} the moment bound 
\begin{align}\label{eq:em:norm}
	\Vert \langle \emK_{t}, I_j \rangle \Vert_{k} 
	\leq 
	C(k) |I_j|^\frac12 t^{-\frac34}
	\leq 
	C(k) K^{-\frac{\alpha}{2}} t^{-\frac34},
\end{align}
for all $ k<\infty $.
With $\frac{\alpha}{2} > \gamma $,
\eqref{eq:em:norm} yields \eqref{eq:em:bdT} by same argument we obtained \eqref{eq:emY:norm}.
\end{proof}

Equipped with Lemmas~\ref{lem:emYbd}--\ref{lem:em},
we proceed to the main goal of this section:
to develop integral identities 
(in different forms from \eqref{eq:int:abs:e} and \eqref{eq:int:atl:K})
that are convenient for proving the hydrodynamic limits.
Recall from \eqref{eq:alt} that $ \W(t) $ 
is the analogous laggard of the Atlas model $ (Y_i(t);t\geq 0)_i $
and that $ \W_K(t) $ denotes the scaled process.
For any fixed $ t $, we define the scaled distribution function of $ Y $ as
\begin{align}
	\label{eq:V}
	\V_K(t,x) := \tfrac{1}{\sqrt{K}} \# \{ Y^K_i(t) \leq x \}
	=
	\langle \emKY_t, \ind_{(-\infty,x]} \rangle.
\end{align}

\begin{proposition}
\begin{enumerate}[label=(\alph*)]
\item []
\item	\label{enu:intX}
Let $ (\phi_i(t);t\geq 0)_{i=1}^K $ be any given strategy.
The following integral identity holds for all $ t<\infty $ and $ x\geq 0 $:
\begin{align}\label{eq:int:abs}
	\Uc_K(t,x) = \Gc_K(t,x)
	+ \sum_{i=1}^K \int_{0}^{t\wedge\tau_i^K} \phi^K_i(s) \pN(t-s,\XKi(s),x) ds
	+ \rd_K(t,x).
\end{align}
Here $ \rd_K(t,x) $ is a remainder term such that,
for given any $ T,n<\infty $ and $ \gamma\in(0,\frac14) $,
\begin{align}\label{eq:rd:bd}
	\Pr \Big(  
		|\rd_K(t,x)| \leq 
		K^{-\gamma} t^{-\frac34}, \ \forall t \leq T, \ x\in\bbR
	\Big) 
	\geq 1 - CK^{-n},
\end{align}
where $ C=C(T,\gamma,n)<\infty $, and is in particular independent of the strategy.
\item \label{enu:intY}
Let $ (Y_i(t);t \geq 0)_{i} $ be an Atlas model,
and let $ \W_K(t) $ and $ \V_K(t,x) $ be as in the preceding,
and assume $ (Y^K_i(0))_i $ satisfies the conditions \eqref{eq:D*}--\eqref{eq:Dan}.
Then, the following integral identity holds for all $ t<\infty $ and $ x\in\bbR $:
\begin{align}\label{eq:int:atl}
	\V_K(t,x) 
	= 
	\int_0^\infty p(t,x-y) \V_K(0,y) dy
	- \int_{0}^{t}  p(t-s,\W_K(s),x) ds
	+ \rdY_K(t,x).
\end{align}	
Here $ \rdY_K(t,x) $ is a remainder term such that,
given any $ T,n<\infty $ and $ \gamma\in(0,\frac14) $,
\begin{align}\label{eq:rdY:bd}
	\Pr \Big(  
		 |\rdY_K(t,x)|
		\leq K^{-\gamma},
		\
		\forall t \leq T, \ x\in\bbR
	\Big) 
	\geq 1 - CK^{-n},
\end{align}
where $ C<\infty $ depends only on $ T,n $ and 
$ D_*, D_{\alpha,n} $.
\end{enumerate}
\label{prop:intXY}
\end{proposition}
\noindent
The proof of Proposition~\ref{prop:intXY} requires 
a Kolmogorov-type estimate,
which we recall from \cite{kunita97} as follows.
\begin{lemma}[{\cite[Theorem~1.4.1]{kunita97}}]
\label{lem:kol}
Let $ T <\infty $, $ a\in\bbR $, and let $ F $ be a 
$ \Csp([0,\infty)\times\bbR) $-valued process.
If, for some $ \alpha_1,\alpha_2 $, $ k\in\bbN $ and $ C_1<\infty $ with $ \frac{1}{k\alpha_1}+\frac{1}{k\alpha_2}<1 $,
\begin{align}
	\label{eq:kol:0}
	\Vert F(0,0) \Vert_{k} &\leq C_1,
\\
	\label{eq:kol:holder}
	\Vert F(t,x)-F(t',x') \Vert_k &\leq C_1 (|t-t'|^{\alpha_1}+|x-x'|^{\alpha_2}),
\end{align}
$\forall t,t'\in [0,T]$,  $x,x'\in[a,a+1]$,
then
$
	\big\Vert \norm F \norm_{L^\infty([0,T]\times[a,a+1])} \big\Vert_k 
	\leq 
	C_2=C_2(C_1,T,\alpha_1,\alpha_2)<\infty.
$
\end{lemma}
\noindent
Note that, although the dependence of $ C_2 $ is not explicitly designated in \cite[Theorem~1.4.1]{kunita97},
under the present setting, it is clear from the proof of \cite[Lemma~1.4.2, Lemma~1.4.3]{kunita97} 
that $ C_2=C_2(C_1,T,\alpha_1,\alpha_2,k) $.

\begin{proof}[Proof of Proposition~\ref{prop:intXY}]
The first step of the proof is to rewrite 
\eqref{eq:int:abs:e} and \eqref{eq:int:atl:K} in a form similar 
to \eqref{eq:int:abs} and \eqref{eq:int:atl}.
To motivate this step,
recall from Remark~\ref{rmk:int:meaning} 
that the term $ \langle \emK_{t}, \testf(\frac1K,x,\Cdot) \rangle $
should approximate $ \Uc_K(t,x) $ as $ K\to\infty $.
In view of this, we write
\begin{align}
	\notag
	\langle \emK_{t}&, \testf(\tfrac1K,x,\Cdot) \rangle
	=
	\Uc_K(t,x) +\emrd_K(t,x),
\\
	\label{eq:emrd}
	&\text{ where }
	\emrd_K(t,x)
	:= \langle \emK_{t}, \testf(\tfrac1K,x,\Cdot)-\testf(0,x,\Cdot) \rangle
	= \langle \emK_{t}, \testf(\tfrac1K,x,\Cdot)-\ind_{(x,\infty)} \rangle.
\end{align}
Similarly, for the the the first two terms on the r.h.s.\ of \eqref{eq:int:abs:e},
we write
\begin{align*}
	&\Gc_K(t_K,x) = \Gc_K(t,x) + \big( \Gc_K(t_K,x)-\Gc_K(t,x) \big)
\\
	&\sum_{i=1}^K \int_{0}^{t} \phi^K_i(s) \pN(t_K-s,\XKi(s),x) ds
	=
	\sum_{i=1}^K \int_{0}^{t} \phi^K_i(s) \pN(t-s,\XKi(s),x) ds
	+
	\prd_N(t,x),
\end{align*}
where
\begin{align}
	\label{eq:pNrd}
	\prd_K(t,x) 
	:= 
	\sum_{i=1}^K \int_{0}^{t} \phi^K_i(s) 
	\big( \pN(t-s,\XKi(s),x)-\pN(t_K-s,\XKi(s),x) \big) ds.
\end{align}
Under these notations, we rewrite \eqref{eq:int:abs:e} as
\begin{align}
	\label{eq:int:abs::}
	\Uc_K(t,x) = \Gc_K(t,x) + \sum_{i=1}^K \int_{0}^{\tKi\wedge t} \phi^K_i(s) \pN(t-s,\XKi(s),x) ds
	+ \rd_K(t,x),
\end{align}
where
\begin{align}
	\label{eq:rd}
	\rd_K(t,x) := ( \Gc_K(t_K,x)-\Gc_K(t,x) ) - \emrd_K(t,x)+\prd_K(t,x)+\mg_K(t,x).
\end{align}

Equation~\eqref{eq:int:abs::} gives
the desired identity \eqref{eq:int:abs} with the explicit remainders $ \rd_K(t,x) $.
Similarly for the Atlas model $ Y $, we define
\begin{align}
	\label{eq:emrdY}
	\emrdY_K(u,t,x) &:= \langle \emKY_{u}, \erf(t_K,x-\Cdot)-\erf(t,x-\Cdot) \rangle
\\
	\label{eq:prd}
	\prdY_K(t,x) 
	&:= 
	\int_{0}^{t} 
	\big( p(t-s,x-\W_K(s))-p(t_K-s,x-\W_K(s)) \big) ds.
\\
	\label{eq:rdY}
	\rdY_K(t,x) &:= \emrdY_K(0,t,x) - \emrdY_K(t,0,x)-\prdY_K(t,x)+\mgY_K(t,x),
\end{align}
and rewrite \eqref{eq:int:atl:K} as 
\begin{align*}
	\V_K(t,x) 
	&= \langle\emKY_{0}, \erf(t,x-\Cdot)\rangle 
	- \int_{0}^{t} p(t-s,x-\W_K(s)) ds + \rdY_K(t,x).
\end{align*}
Further using integration by parts:
\begin{align*}
	\langle\emKY_{0}, \erf(t,x-\Cdot)\rangle 
	=
	\int_{\bbR}  \erf(t,x-y) d\V_K(0,y)
	=
	-\int_{\bbR}  \V_K(0,y) \partial_y \erf(t,x-y) dy,
\end{align*}
we write
\begin{align}
	\label{eq:int:atl::}
	\V_K(t,x)  &= \int_{\bbR} p(t,x-y) \V_K(0,x)  dy
	- \int_{0}^{t} p(t-s,x-\W_K(s)) ds + \rdY_K(t,x).
\end{align}

Equations~\eqref{eq:int:abs::} and \eqref{eq:int:atl::} give
the desired identities \eqref{eq:int:abs} and \eqref{eq:int:atl}
with the explicit remainders $ \rd_K(t,x) $ and $ \rdY_K(t,x) $,
as in \eqref{eq:rd} and \eqref{eq:rdY}.
With this, it suffices to show that these remainders do satisfy
the bounds \eqref{eq:rd:bd} and \eqref{eq:rdY:bd}.
To this end,
fixing arbitrary $ T,n < \infty $ and $ \gamma\in(0,\frac14) $, 
we let $ C(k)<\infty $ denote a generic constant depending only on 
$ T,n,\alpha,\gamma, D_*, D_{\alpha,n} $,
and the designated variable $ k $.

We begin with a reduction.
That is, in order to prove \eqref{eq:rd:bd} and \eqref{eq:rdY:bd},
we claim that it suffices to prove
\begin{align}
	\label{eq:rd:bd:}
	&
	\Pr \Big(  
		|\rd_K(t,x)| \leq K^{-\gamma} t^{-\frac34}, \ \forall t \leq T, x\in [a,a+1]
	\Big) 
	\geq 1 - CK^{-n},
\\
	\label{eq:rdY:bd:}
	&
	\Pr \Big(  
		|\rdY_K(t,x)| \leq K^{-\gamma}, \ \forall t \leq T, x\in [a,a+1]
	\Big) 
	\geq 1 - CK^{-n},
\end{align}
for all $ a\in\bbR $.
To see why such a reduction holds,
we assume that \eqref{eq:rd:bd:} has been established,
and take the union bound of \eqref{eq:rd:bd:}
over $ a\in \bbZ\cap[-K,K] $ to obtain
\begin{align}
	\label{eq:rd:bd::}
	&
	\Pr \Big(  
		|\rd_K(t,x)| \leq K^{-\gamma} t^{-\frac34}, \ \forall t \leq T, |x| \leq K
	\Big) 
	\geq 1 - CK^{-n+1}.
\end{align}
To cover the regime $ |x|>K $ that is left out by \eqref{eq:rd:bd::},
we use the fact that each $ X^K_i(t) $
evolves as a Brownian motion with drift at most $ \sqrt{K} $ (and absorption)
to obtain
\begin{align}
	\Pr(\Omega_K) \geq 1 - CK^{-n},
	\quad
	\Omega_K := \{ |X^K_i(t)| \leq \tfrac12 K, \ \forall t\leq T, \forall i \}.
\end{align}
That is, with a sufficiently high probability,
each particle $ \XKi(t) $ stays within $ [-\frac12K,\frac12K] $ for all time.
Use \eqref{eq:int:abs::} to
express $ \rd_K(t,x) = \Uc_K(t,x) - f(t,x)- \Gc_K(t,x) $,
where\\
$ f(t,x) := \sum_{i=1}^K \int_0^{t\wedge\tKi} \phi^K_i(s) \pN(t-s,\XKi(s),x) ds $.
On the event $ \Omega_K $, 
the function $ x\mapsto \Uc_K(t,x) $ remains constant on $ \bbR\setminus(-K,K) $;
and, for all $ x>K $,
\begin{align*}
	| f(t,\pm x) - f(t,\pm K) |
	&\leq
	\sum_{i=1}^K \int_0^{t\wedge\tKi} \phi^K_i(s) |\pN(t-s,\XKi(s),\pm x)-\pN(t-s,\XKi(s),\pm K)| ds
\\
	&\leq
	\int_0^{t} 4|p(t-s,\tfrac{K}{2})| ds
	\leq
	\frac{C}{K}.
\end{align*}
From these we conclude that, on $ \Omega_K $,
\begin{align*}
	\sup_{x\in\bbR} |\rd_K(t,x)| 
	\leq 
	\sup_{|x|\leq K} |\rd_K(t,x)| + \Big( \frac{C}{K} + \sup_{|x|\geq K} |\Gc_K(t,x)| \Big)
	\leq
	\sup_{|x|\leq K} |\rd_K(t,x)| + \frac{C}{K}.	
\end{align*}
Combining this with \eqref{eq:rd:bd::} gives the desired bound \eqref{eq:rd:bd}.
A similar argument shows that \eqref{eq:rdY:bd:} implies \eqref{eq:rdY:bd}.

Having shown that \eqref{eq:rd:bd:}--\eqref{eq:rdY:bd:}
imply the desired results,
we now return to proving \eqref{eq:rd:bd:}--\eqref{eq:rdY:bd:}.
This amounts to bounding each term on the r.h.s.\ of 
the explicit expressions \eqref{eq:rd} and \eqref{eq:rdY} of 
$ \rd_K(t,x) $ and $ \rdY_K(t,x) $.
To this end, fixing $ t \leq T $, $ a\in\bbR $ and $ x\in [a,a+1] $,
we establish bounds on the following terms in sequel.
\begin{enumerate}[label=\itshape\roman*\upshape)]
	\item $ |\Gc_K(t_K,x)- \Gc_K(t,x)| $;
	\item $ \prd_K(t,x) $ and $ \prdY_K(t,x) $;
	\item $ \emrd_K(t,x) $, $ \emrdY_K(0,t,x) $ and $ \emrdY_K(t,0,x) $;
	and
	\item $ \mgY_K(t,x) $ and $ \mg_K(t,x) $.
\end{enumerate}

(\textit{i})
By \eqref{eq:Gc:pNm} for $ m=0 $,
we have that
\begin{align}
	\label{eq:Gc:pN}
	\Gc_{K}(t,x) 
	= \sqrt{K} \int_{0}^{\frac{1}{\sqrt{K}}}
	\pN(s,y,x) dy.			
\end{align}
Applying the bound~\eqref{eq:p:Holdt} for $ \alpha=\frac{1}{4} $
within \eqref{eq:Gc:pN}, we obtain
\begin{align}\label{eq:Gc:bd}
	|\Gc_K(t_K,x)- 2p(t,x)| 
	\leq 
	C K^{-\frac14} t^{-\frac34}.
\end{align}

(\textit{ii})
Applying \eqref{eq:p:Holdt} for $ \alpha=\frac12 $ in \eqref{eq:pNrd} and in \eqref{eq:prd} yields
\begin{align}\label{eq:prd:bd}
	|\prd_K(t,x)|, \ |\prdY_K(t,x)| \leq C K^{-\frac14}.
\end{align}

(\textit{iii})
For the Brownian distribution function $ \erf(t,y) = \Pr(B(t)\leq y) $,
it is standard to show that $ t\mapsto |\erf(t_K,y) - \erf(t,y)| $ decreases in $ t $,
and that $ |\erf(\frac1K,y) - \erf(0,y)| \leq C \exp(-\sqrt{K}|y|) $.
Further,
fixing $ \alpha \in (2\gamma,\frac12) $ and
letting $ I_{x,\alpha} $ be as in \eqref{eq:Ixa}, 
we write 
\begin{align*}
	\exp(-\sqrt{K}|y-x|) 
	\leq
	\ind_{I_{x,\alpha}}(y) + \ind_{\bbR\setminus{}I_{x,\alpha}}(y) \exp(-\sqrt{K}|y-x|) 
	\leq 
	\ind_{I_{x,\alpha}}(y) + \exp(-K^{\alpha-\frac12}).
\end{align*}
From these bounds we conclude
\begin{subequations}
\label{eq:iii}
\begin{align}
	\label{eq:iii:0}
	|\erf(\tfrac1K,y-x) - \erf(0,y-x)|
	&\leq 
	C ( \ind_{I_{x,\alpha}}(y) + \exp(-K^{\alpha-\frac12}) ), 
\\
	\label{eq:iii:1}
	|\erf(t_K,y-x) - \erf(t,y-x)|
	&\leq 
	C ( \ind_{I_{x,\alpha}}(y) + \exp(-K^{\alpha-\frac12}) ), 
\\
	|\testf(\tfrac{1}{K},y,x) - \testf(0,y,x)|
	\label{eq:iii:2}
	&\leq
	C ( \ind_{I_{x,\alpha}\cup I_{-x,\alpha}}(y) + \exp(-K^{\alpha-\frac12}) ).
\end{align}
\end{subequations}
Recall the definition of $ \emrd(t,x) $ and $ \emrdY(u,t,x) $
from \eqref{eq:emrd} and \eqref{eq:emrdY}.
Applying $ \langle \emKY_{t}, \Cdot \rangle $
$ \langle \emKY_{0}, \Cdot \rangle $ and $ \langle \emK_{t}, \Cdot \rangle $
on both sides of \eqref{eq:iii:0}--\eqref{eq:iii:2}, respectively,
we obtain
\begin{subequations}
\label{eq:emrd:bd:}
\begin{align}
	|\emrdY_K(t,0,x)| 
	&\leq
	C \langle \emKY_{t}, \ind_{I_{x,\alpha}} \rangle
	+ C\exp(-K^{\alpha-\frac12}) \langle \emKY_{t}, \ind_{\bbR} \rangle,
\\
	|\emrdY_K(0,t,x)| 
	&\leq
	C \langle \emKY_{0}, \ind_{I_{x,\alpha}} \rangle
	+ C\exp(-K^{\alpha-\frac12}) \langle \emKY_{0}, \ind_{\bbR} \rangle,
\\
	|\emrd_K(t,x)|
	&\leq
	C \langle \emK_{t}, \ind_{I_{x,\alpha}} \rangle
	+ C \langle \emK_{t}, \ind_{I_{-x,\alpha}} \rangle
	+ C\exp(-K^{\alpha-\frac12}) \langle \emK_{t}, \ind_{\bbR} \rangle.
\end{align}
\end{subequations}
On the r.h.s.\ of \eqref{eq:emrd:bd:} sit two types of terms:
the `concentrated terms' that concentrate on the small interval $ \ind_{I_{\pm x,\alpha}} $;
and the `tail terms' with the factor $ \exp(-K^{\alpha-\frac12}) $.
For the tail terms,
writing
$ 
	\langle \emKY_{0}, \ind_{\bbR} \rangle 
	= 
	\langle \emKY_{t}, \ind_{\bbR} \rangle  
	=
	\frac{1}{\sqrt{K}}\#\{ Y^K_i(0) \}
$
and 
$ 
	\langle \emK_{t}, \ind_{\bbR} \rangle 
	\leq 
	\frac{1}{\sqrt{K}}\#\{ X^K_i(0) \}
$,
and using the bound~\eqref{eq:D*} and $ \#\{ X^K_i(0) \}=K $,
we bound the tail terms by
$ C\sqrt{K}\exp(-K^{\alpha-\frac12}) $,
with probability $ \geq 1-CK^{-n} $.
Further using $ \sqrt{K}\exp(-K^{\alpha-\frac12}) \leq CK^{-\gamma} $,
we have
\begin{subequations}
\label{eq:emrd:bd::}
\begin{align}
	|\emrdY_K(t,0,x)| 
	&\leq
	C \langle \emKY_{t}, \ind_{I_{x,\alpha}} \rangle
	+ CK^{-\gamma},
\\
	|\emrdY_K(0,t,x)| 
	&\leq
	C \langle \emKY_{0}, \ind_{I_{x,\alpha}} \rangle
	+ CK^{-\gamma},
\\
	|\emrd_K(t,x)|
	&\leq
	C \langle \emK_{t}, \ind_{I_{x,\alpha}} \rangle
	+C \langle \emK_{t}, \ind_{I_{-x,\alpha}} \rangle
	+ CK^{-\gamma},
\end{align}
\end{subequations}
with probability $ \geq 1-CK^{-n} $.
Next, to bound the concentrated terms,
we consider the covering
$
	\mathcal{X} := \{ I_{y,\alpha} : |y|\leq a+1 \}
$
of $ [-a-1,a+1] $.
With $ x\in[a,a+1] $, we clearly have that 
$ I_{\pm x, \alpha} \in \mathcal{X} $,
so by Lemma~\ref{lem:em} it follows that
\begin{align*}
	\langle \emKY_{t}, \ind_{I_{x,\alpha}} \rangle,
	\
	\langle \emKY_{0}, \ind_{I_{x,\alpha}} \rangle
	\leq
	C K^{-\gamma},
	\quad
	\langle \emKY_{t}, \ind_{I_{\pm x,\alpha}} \rangle
	\leq
	C K^{-\gamma}t^{-\frac34},
\end{align*}
with probability $ 1-CK^{-n} $.
Inserting this into \eqref{eq:emrd:bd::} gives
\begin{subequations}
\label{eq:emrd:bd}
\begin{align}
	&
	\Pr\big( 
		|\emrdY_K(t,0,x)| \leq C K^{-\gamma}, 
		\ \forall t\leq T, x\in [a,a+1] 
	\big)
	\geq 
	1 - C K^{-n},
\\
	&
	\Pr\big(  
		|\emrdY_K(0,t,x)| \leq C K^{-\gamma}, 
		\ \forall t\leq T, x\in [a,a+1] 
	\big)
	\geq 
	1 - C K^{-n},
\\
	&
	\Pr\big( 
		|\emrd_K(t,x)| \leq C K^{-\gamma} t^{-\frac{3}{4}}, 
		\ \forall t\leq T, x\in [a,a+1] 
	\big)
	\geq 
	1 - C K^{-n}.
\end{align}
\end{subequations}

(\textit{iv})
The strategy is to apply Lemma~\ref{lem:kol} for $ F(t,x) := K^{1/4} \mgY_K(t,x) $.
With $ \mgY_K(t,x) $ defined as in \eqref{eq:mgY},
for such $ F $ we have $ F(0,0)=0 $,
so the condition~\eqref{eq:kol:0} holds trivially.
Turning to verifying the condition~\eqref{eq:kol:holder}, 
we fix $ t<t' $ and $ x,x'\in\bbR $.
With $ \mgY_K(t,x) $ defined as in \eqref{eq:mgY},
we telescope $ F(t,x) - F(t',x') $ into $ F_1 + F_2 - F_3 $, where
\begin{align*}
	F_1 &:= K^{-1/4} \sum_{i} \int_{0}^{t} f_1(s,\YKi(s)) d\BKi(s),&
	&
	F_2 := K^{-1/4} \sum_{i} \int_{0}^{t} f_2(s,\YKi(s)) d\BKi(s),
\\
	F_3 &:= K^{-1/4} \sum_{i} \int_{t}^{t'} f_3(s,\YKi(s)) d\BKi(s),
\end{align*} 
$ f_1(s,y) := p(t_K-s,y-x) - p(t_K-s,y-x') $, 
$ f_2(s,y) := p(t_K-s,y-x') - p(t'_K-s,y-x') $
and $ f_3(s,y) := p( t'_K-s, y-x' ) $.
Similar to the way we obtained \eqref{eq:J3:bd1}, 
here by the \ac{BDG} inequality we have
\begin{align*}
	\Vert F_1 \Vert_{k}^2 
	&\leq 
	C(k) 
	\int_{0}^{t} \Vert \langle \emKY_{s}, f_1(s,\Cdot)^2 \rangle \Vert_{k/2} ds,&
	&
	\Vert F_2 \Vert_{k}^2 
	\leq 
	C(k) 
	\int_{0}^{t} \Vert \langle \emKY_{s}, f_2(s,\Cdot)^2 \rangle \Vert_{k/2} ds,
\\
	\Vert F_3 \Vert_{k}^2 
	&\leq 
	C(k) 
	\int_{t}^{t'} \Vert \langle \emKY_{s}, f_3(s,\Cdot)^2 \rangle \Vert_{k/2} ds,
\end{align*} 
for any fixed $ k>1 $.
On the r.h.s., the kernel functions $ f_1,f_2,f_3 $ appear in square (i.e.\ power of two).
We use \eqref{eq:p:Holdx}--\eqref{eq:p:Holdt}
to replace `one power' of them with 
$ C(t-s)^{-\frac34} |x-x'|^{\frac12} $, 
$ C(t-s)^{-\frac34} |t-t'|^{\frac14} $
and $ C(t-s)^{-\frac12} $, respectively,
and then use \eqref{eq:emPY:bd} for $ \alpha=\frac34 $ to bound 
$ \Vert \langle \emKY_{s}, f_j(s,\Cdot) \rangle \Vert_{k/2} $, $ j=1,2,3 $,
whereby obtaining 
\begin{align*}
	\Vert F_1 \Vert_{k}^2 
	&\leq 
	C(k)
	\int_{0}^{t} |x-x'|^{\frac12} ((t-s)^{-\frac78}+(t-s)^{-\frac34} s^{-\frac18})  ds
	\leq
	C(k) |x-x'|^{\frac12},
\\
	\Vert F_2 \Vert_{k}^2 
	&\leq 
	C(k)
	\int_{0}^{t} |t-t'|^{\frac14} ((t-s)^{-\frac78}+(t-s)^{-\frac34} s^{-\frac18}) ds
	\leq
	C(k) |t-t'|^{\frac14},
\\
	\Vert F_3 \Vert_{k}^2 
	&\leq 
	C(k)
	\int_{t}^{t'} ((t-s)^{-\frac58}+(t-s)^{-\frac12} s^{-\frac18}) ds
	\leq
	C(k) |t-t'|^{\frac38}.
\end{align*}
We have thus verify the condition~\eqref{eq:kol:holder} 
for $ (\alpha_1,\alpha_2) = (\frac18,\frac14) $.
Now apply Lemma~\ref{lem:kol} to obtain
$
	\Vert \norm \mgY_K \norm_{L^\infty([0,T]\times[a,a+1])} \Vert_{k} 
	\leq 
	C(k) K^{-\frac14}.
$
From this and Markov's inequality, we conclude
\begin{align}\label{eq:mgY:bd}
	\Pr \Big(  
		|\mgY_K(t,x)| \leq K^{-\gamma}, \ \forall t\leq T, x\in[a,a+1]
	\Big)
	\geq
	1 - C(k) K^{-k(1-\gamma)}
	\geq
	1 - C K^{-n}.
\end{align}
The term $ \mg_K(t,x) $ is bounded by similar procedures as in the preceding.
The only difference is that the estimate \eqref{eq:emP:bd},
unlike \eqref{eq:emPY:bd}, introduces a singularity of $ \mg_K(t,x) $ as $ t\to 0 $,
so we set $ F(t,x) := t^{\frac34} K^{1/4}\mg_K(t,x) $ (instead of $ F(t,x) := K^{1/4}\mg_K(t,x) $).
The extra prefactor $ t^{\frac34} $ preserves the moment estimate \eqref{eq:kol:holder}
since $ t^{\frac34} $ is $ \alpha $-H\"{o}lder continuous for all $ \alpha<\frac34 $.
Consequently, following the preceding argument we obtain
\begin{align}\label{eq:mg:bd}
	\Pr \Big(  
		\sup_{t\in[0, T],x\in[a,a+1]} (t^{\frac34}|\mg_K(t,x)|) \leq K^{-\gamma}
		\Big)
	\geq
	1 - C K^{-n}.
\end{align}

Now, combining the bounds 
\eqref{eq:Gc:bd}, \eqref{eq:prd:bd}, \eqref{eq:emrd:bd} and \eqref{eq:mgY:bd}--\eqref{eq:mg:bd} 
from (\textit{i})--(\textit{iv}), 
we conclude the desired results \eqref{eq:rd:bd} and \eqref{eq:rdY:bd}.
\end{proof}

\section{The Stefan Problem}
\label{sect:Stef}

In this section, we develop the necessary \ac{PDE} tools.
As stated in Remark~\ref{rmk:Stefan},
we take the integral identity and integral equations \eqref{eq:Ucs:move}--\eqref{eq:zeq}, 
instead of \eqref{eq:PDE>},
as the definition of the Stefan problem. 
To motivate such a definition, we first prove the following:
\begin{lemma}\label{lem:StefInt}
Let $ (u_2,z_2) $ be a classical solution to the following \ac{PDE}, i.e.,
\begin{subequations}\label{eq:Stef}
\begin{align}
	\notag
	&z_2 \in \Csp^1((0,\infty)) \cap \Csp([0,\infty)),
	\
	\text{nondecreasing}, \  z(0)=0, 
\\
	\notag
	&u_2 \in L^\infty(\barDsp)\cap L^1(\barDsp),
	\text{ and has a } \Csp^2\text{-extension onto a neighborhood of } \barDsp,
\\
	&
	\quad\quad\quad\text{where } \Dsp := \{(t,x): t>0, x \geq z_2(t) \},	
\\
	\label{eq:StefHE}
	&\partial_t u_2 = \tfrac12 \partial_{xx} u_2, \quad \forall~ 0< t < T, \ x>z(t),
\\
	&
	\label{eq:StefDiri}
	u_2(t,z_2(t)) = 2,	\quad \forall t\geq 0,
\\
	&
	\label{eq:Stefmv}
	2\tfrac{d~}{dt}z_2(t) + \tfrac12 \partial_x u_2(t,z_2(t)) = 0,	\quad \forall t> 0,
\end{align}
\end{subequations}
Define the tail distribution function of $ u_2 $ as 
$ \Uc_2(t,x) := \int_x^\infty u_2(t,y)dy $.
We have
\begin{align}
	\label{eq:Uc2:int}
	&\Uc_2(t,x) 
	= 
	\int_0^\infty \pN(t,y,x) \Uc_2(0,y) dy
	+ \int_0^t \pN(t-s,z_2(s),x) ds,
	\quad
	\forall t,x \in\bbR_+,
\\
	\label{eq:Stef:z}
	&\int_{0}^\infty p(t,z_2(t)-y) 
	\big( \Uc_2(0,0)-\Uc_2(0,y) \big) dy 
	= 
	\int_0^t p(t-s,z_2(t)-z_2(s)) ds.
\end{align}
\end{lemma}
\begin{proof}
Instead of the tail distribution function $ \Uc_2(t,x) $,
let us first consider the distribution function $ \U_2(t,x) := \int_{z_2(t)}^{x} u_2(t,y) dy $.
We adopt the convention that $ \U_2(t,x)|_{x<z_2(t)}:=0 $.
By \eqref{eq:StefHE}, \eqref{eq:StefDiri}--\eqref{eq:Stefmv},
$ \U_2(s,y) $ solves the heat equation in $ \{(s,y): s>0, y>z(s)\} $.
With this, 
for any fixed $ t>0 $ and $ x\in\bbR $,
we integrate Green's identity
\begin{align*}
	\tfrac12 \partial_y((\partial_y p)\U_2-p(\partial_y \U_2)) + \partial_s(p \U_2) =0,
	\
	\text{ where } p = p(t-s,x-y), \ \U_2=\U_2(s,y), 
\end{align*}
over $ \{(s,y):\e <s<t-\e,y>z(s)+\e\} $.
Letting $ \e\to 0 $,
and combining the result with $ \U_2(s,z_2(s))=0 $ and \eqref{eq:StefDiri}, we obtain
\begin{align}
	\label{eq:U2:int}
	\U_2(t,x) = \int_{0}^\infty p(t,x-y) \U_2(0,y)dy -\int_0^t p(t-s,x-z(s)) ds,
	\quad
	\forall t\in\bbR_+, \ x\in\bbR.
\end{align}
Note that the preceding derivation of~\eqref{eq:U2:int} applies to \emph{all} $ x \in\bbR $, including $ x <z_2(t) $.
Setting $ x=z_2(t) $ in \eqref{eq:U2:int},
on the l.h.s.\ we have $ \U_2(t,z_2(0))=0 $.
Further using $ \U_2(0,y) =\Uc_2(0,0)-\Uc_2(0,y) $,
we see that \eqref{eq:Stef:z} follows.

We now turn to showing \eqref{eq:Uc2:int}.
A straightforward differentiation, following by using~\eqref{eq:StefDiri}--\eqref{eq:Stefmv},
gives
\begin{align*}
	\partial_t \U_2(t,\infty) 
	=
	\partial_t \int_{z_2(t)}^\infty u_2(t,x) dx
	&=
	 -z'_2(t) u_2(t,z(t)) + \int_{z(t)} \frac{1}{2} \partial_{xx} u_2(t,x) dx
\\	 
	&
	 =-2z'_2(t) - \frac{1}{2} \partial_{x} u_2(t,z(t)) =0,
\end{align*}
so in particular $ \U_2(0,\infty)=\U_2(t,\infty) $.
Consequently, 
\begin{align}
	\label{eq:U22}
	\Uc_2(t,x) = \U_2(0,\infty)-\U_2(t,x).
\end{align}
Further, as $ U_2(t,x)|_{x\leq 0} =0 $,
\begin{align}
	\notag
	\U_2(t,x) 
	&= 
	\U_2(t,x) + \U_2(t,-x)
\\
	\label{eq:U2:pN}
	&=
	\int_{0}^\infty \pN(t,y,x) \U_2(0,y)dy -\int_0^t \pN(t-s,z(s),x) ds,
	\quad
	\forall t,x\in\bbR_+.
\end{align}
Inserting~\eqref{eq:U2:pN} into the last term in \eqref{eq:U22} yields
\begin{align*}
	\Uc_2(t,x) 
	= 
	\U_2(0,\infty)
	-\int_{0}^\infty \pN(t,y,x) \U_2(0,y)dy + \int_0^t \pN(t-s,z(s),x) ds,
	\quad
	\forall t,x\in\bbR_+.
\end{align*}
Further using $ \U_2(0,\infty) = \int_0^\infty \pN(t,y,x) \U_2(0,\infty) dy $
to write
\begin{align*}
	\U_2(0,\infty) - \int_{0}^\infty \pN(t,y,x) \U_2(0,y)dy
	=
	\int_{0}^\infty \pN(t,y,x) \Uc_2(0,y)dy,
\end{align*}
we see that \eqref{eq:Uc2:int} follows.
\end{proof}

We next turn to the well-posedness of \eqref{eq:Stef:z}.
The existence of a solution to \eqref{eq:Stef:z}
will be established in Lemma~\ref{lem:W:hydro}, Section~\ref{sect:mvbdy},
as a by-product of establishing the hydrodynamic limit of certain Atlas models.
Here we focus the uniqueness and stability of \eqref{eq:Stef:z}.
To this end,
we consider $ w\in \Csp([0,T]) $ that satisfies
\begin{equation} \label{eq:pStef:z}
	\int_{0}^\infty p(t, w(t)-y) \big( \Ucs(\tfrac12,0)-\Ucs(\tfrac12,y)\big) dy 
	= 
	f(t,w(t)) + \int_0^t p(t-s, w(t)-w(s)) ds,
\end{equation}
where $ f \in L^\infty([0,T]\times \bbR) $ is a generic perturbation.
Define a seminorm
\begin{equation} 
	\label{eq:seminorm}
	|w|'_{[0,T]} := \sup\{w(t)-w(t'),~0 \leq t \leq t' \leq T\}
\end{equation}
that measures how nondecreasing the given function is.
\begin{lemma} \label{lem:pStef}
Fixing $ T<\infty $ and $ f_1(t,x),f_2(t,x) \in L^\infty([0,T]\times \bbR) $,
we consider $ w_1 $ and $ w_2 $ satisfying \eqref{eq:pStef:z}
for $ f=f_1 $ and $ f=f_2 $, respectively.
Let $ L := \sup \{ |w_1(t)|, |w_2(t)| : t \leq T \}+1 $.
There exists $C_1=C_1(T,L)<\infty$ such that
\begin{align*}
	\sup_{0\leq t \leq T}(w_1(t)-w_2(t)) 
	\leq 
	C_1 \sum_{i=1,2} 
	\big( |w_i(0)|+ \norm f_i\norm_{L^\infty([0,T] \times\bbR)} +|w_i|'_{[0,T]} \big),
\end{align*}
for all $ f_1,f_2 \in L^\infty([0,T]\times\bbR) $ 
satisfying 
$ 
	\sum_{i=1,2} 
	\big( |w_i(0)|+ \norm f_i\norm_{L^\infty([0,T] \times\bbR)} +|w_i|'_{[0,T]} \big) 
	\leq
	\frac{1}{C_1}.
$
\end{lemma}
\noindent
Indeed, when $ f_1=f_2=0 $, Lemma~\ref{lem:pStef} yields
\begin{corollary}\label{cor:unique}
The solution to \eqref{eq:zeq} is unique.
\end{corollary}
\begin{proof}[Proof of Lemma~\ref{lem:pStef}]
To simplify notations, let 
$ \e:= |f_1|_{L^\infty([0,T] \times\bbR)} + |f_2 |_{L^\infty([0,T] \times\bbR)} $, 
$ \e':=|w_1|'_{[0,T]}+|w_2|'_{[0,T]} $
and $ \e'':= |w_1(0)|+|w_2(0)| $.

Let 
\begin{align}
	\label{eq:LHSof}
	\Lambda(t,z)
	&:=
	\int_0^{\infty} p(t,z-y) (\Ucs(\tfrac12,0)-\Ucs(\tfrac12,y)) dy
\\
	&=
	\int_{-\infty}^{z} p(t,y) (\Ucs(\tfrac12,0)-\Ucs(\tfrac12,z-y)) dy
\end{align}
denote the expression on the l.h.s.\ of \eqref{eq:pStef:z}.
From the explicit expressions \eqref{eq:U1}--\eqref{eq:u1},
we have that $ \partial_x (-\Ucs(\frac12,z-y)) = u_1(\frac12,z-y) >0 $, $ \forall y\leq z $,
so $ \partial_z \Lambda(t,z) > 0$, $ \forall z \geq 0 $.
Consequently, there exists $c_1 = c_1(T,L)>0$ such that
\begin{equation} \label{eq:lbpartialF}
	\partial_z \Lambda(t,z) \geq c_1 >0, \quad \forall 0 \leq  z \leq L, \ 0 \leq t \leq T.
\end{equation}
Setting $ C_1:=\frac{4}{c_1}\vee 1 $
and $ \delta:=C_1 (\e+\e'+\e'') \leq 1 $,
we write $ w_2^{\delta}(t):=w_2(t)+\delta $ to simplify notations,
and consider the first time $t^{*}:=\inf\{t \leq T: w_1(t) \geq w_2^{\delta}(t) \} $
when $ w_1 $ hits $ w^\delta_2 $.
Indeed, since $ C_1 \geq 1 $, we have $ w_1(0) \leq w_2(0)+|w_2(0)-w_1(0)| < w_2(0)+\delta $,
so in particular $ t^* >0 $.
Taking the difference of \eqref{eq:pStef:z} for $ (t,f)=(t^*,f_1) $ and for $ (t,f)=(t^*,f_2) $, 
we obtain
\begin{equation}\label{eq:ubdelta}
	\Lambda(t^{*},w_1(t^{*})) - \Lambda(t^{*},w_2(t^{*})) 
	=
	\Lambda(t^{*},w_2^{\delta}(t^{*})) - \Lambda(t^{*},w_2(t^{*})) 
	\leq 
	\varepsilon + \int_0^{t^{*}} g^{*}(s) ds,
\end{equation}
where $ g^{*}(s):=p(t^{*}-s,w_1(t^{*})-w_1(s)) - p(t^{*}-s,w_2(t^{*})-w_2(s)) $.
Next, using $ w_1(s) \leq w_2(s)+\delta $, $ \forall s\leq t^* $, we have
\begin{align}
	\label{eq:w12}
	w_1(t^*)-w_1(s) = w_2(t^*)+\delta-w_1(s)
	\geq 
	w_2(t^*) - w_2(s).
\end{align}
To bound the function $ g^*(s) $, 
we consider the separately cases \textit{i}) $ w_2(t^*) - w_2(s) \geq 0 $;
and \textit{ii}) $ w_2(t^*) - w_2(s) < 0 $.
For case (\textit{i}),
by \eqref{eq:w12} we have 
$ |w_1(t^{*})-w_1(s))| \geq |w_2(t^{*})-w_2(s)| $,
so $ g^*(s) \leq 0 $.
For case (\textit{ii}),
using $ 0 > w_2(t^*) - w_2(s) \geq -\e' $ 
we have $ g^*(s) \leq p(t^*-s,0) -p(t^*-s, \varepsilon') $.
Combining these bounds with the readily verified identity
\begin{align}
	\label{eq:pId}
	\int_0^t p(t-s,x) ds = 2 t p(t,x)- 2|x| \erfc(t,|x|),
\end{align}
we obtain
\begin{align}
	\notag
	\int_0^{t^*} g_{*}(s)ds 
	&\leq
	\int_0^{t^*} (p(t^*-s,0) -p(t^*-s, \varepsilon')) ds
	\leq 
	2t_*p(t^*,0) -2t_*p(t^*, \varepsilon') + 2|\varepsilon'|
\\
	\label{eq:estgstar}
	&=
	\sqrt{\tfrac{2t^*}{\pi}} (1-\exp(-\tfrac{{\e'}^{2}}{2t^*}))+2 
	\varepsilon' 
	<
	4 \varepsilon',
\end{align}
where we used $ (1-e^{-\xi}) \leq 2\sqrt{\xi} $, $ \forall \xi \in\bbR_+ $,
in the last inequality.
Now, if $ t^* \leq T $,
combining \eqref{eq:estgstar} with \eqref{eq:ubdelta} and \eqref{eq:lbpartialF}
yields $ \delta c_1 < \e + 4 \e'$, 
leading to a contradiction. 
Consequently, we must have $ t^* >T $.
\end{proof}

We next establish a property of $ (\Ucs(t,x),\z(t)) $,
that will be used toward the proof of Theorems~\ref{thm:aldous} and \ref{thm:hydro}.
\begin{lemma}\label{lem:Us}
For any $ (\Ucs(t,x),\z(t); t\geq\frac12) $ satisfying \eqref{eq:Ucs:move}--\eqref{eq:zeq},
we have
\begin{align}
	\label{eq:Usc:cnsv}
	\Ucs(t,\z(t)) = \tfrac{4}{\sqrt{\pi}},
	\quad
	\forall t\geq \tfrac12.
\end{align}
\end{lemma}
\begin{remark}
As $ \Ucs(t,x) $ represents the hydrodynamic limit of 
the scaled tail distribution function~$ \Uc_K(t,x) := \frac{1}{\sqrt{K}}\#\{ \XKi(t) >x \} $,
equation~\eqref{eq:Usc:cnsv} is a statement of \emph{conservation of particles}
within the moving boundary phase, in the hydrodynamic limit.
\end{remark}
\begin{proof}
Fixing such $ (\Ucs(t,x),\z(t)) $,
we \emph{define}
\begin{align*}
	\Us(t+\tfrac12,x) 
	:= 
	\int_0^\infty p(t,x-y) (\Ucs(\tfrac12,0)-\Ucs(\tfrac12,y)) dy
	- \int_0^t p(t-s,x-\z(s+\tfrac12)) ds.
\end{align*}
From this expression,
it is straightforward to verify that, for any fixed $ T<\infty $,
$ \Us(\Cdot+\frac12,\Cdot) \in \Csp([0,T]\times\bbR) \cap L^\infty([0,T] \times\bbR) $
solves the heat equation on $ \{(t,x) : t>0, x< \z(s) \} $.
Further, $ \Us(\frac12,x)|_{x\leq 0}=0 $, and, by \eqref{eq:zeq},
\begin{align}
	\label{eq:Us:z=0}
	\Us(t+\tfrac12,\z(t+\tfrac12))=0.
\end{align}
From these properties of $ \Us(t+\frac12,x) $,
by the uniqueness of the heat equation on the domain $ \{(t,x):t\in\bbR, x< \z(t)\} $, 
we conclude that $ \Us(t+\frac12,x)|_{x\leq\z(t+\frac12)}=0 $.
Therefore,
\begin{align}
	\notag
	&
	\Us(\tfrac12+t,x) 
	= \Us(\tfrac12+t,x) +\Us(\tfrac12+t,-x) 
\\
	&
	\notag
	= 
	\int_0^\infty \pN(t,y,x) (\Ucs(\tfrac12,0)-\Ucs(\tfrac12,y)) dy
	-
	\int_0^t \pN(t-s,\z(s+\tfrac12),x) ds,
\\
	\label{eq:Us:refl}
	&
	= 
	\Ucs(\tfrac12,0)
	-\int_0^\infty \pN(t,y,x) \Ucs(\tfrac12,y) dy
	-
	\int_0^t \pN(t-s,\z(s+\tfrac12),x) ds,
	\quad
	\forall x \in\bbR_+.
\end{align}
Next, set $ t=\frac12 $ in \eqref{eq:Ucs:move},
and write $ 2p(\frac12,y)=\pN(t,0,y) $ to obtain
\begin{align}
	\label{eq:Ucs:itr}
	\Ucs(\tfrac12,y) = \pN(\tfrac12,0,y) + \int_0^\frac12 \pN(\tfrac12-s,\z(s),y) ds.
\end{align}
Inserting this expression~\eqref{eq:Ucs:itr} of $ \Ucs(\frac12,y) $ 
into \eqref{eq:Us:refl},
followed by using the semigroup property 
$ \int_0^\infty \pN(t,y,x)\pN(\tfrac12,z,y) dy = \pN(t+\frac12,z,y) $,
we obtain
\begin{align*}
	\Us(t+\tfrac12,x) 
	&= 
	\Ucs(\tfrac12,0)
	- \pN(t+\tfrac12,0,x)
	- \int_0^{t+\frac12} \pN(t-s,\z(s+\tfrac12),x) ds
\\
	&= \Ucs(\tfrac12,0) - \Ucs(t+\tfrac12,x),
	\quad
	\forall x \in\bbR_+.
\end{align*}
Setting $ x=\z(t+\frac12) $ and using \eqref{eq:Us:z=0} on the l.h.s.,
we conclude the desired identity~\eqref{eq:Usc:cnsv}.
\end{proof}

As will be needed toward the proof of Theorem~\ref{thm:aldous} and \ref{thm:hydro},
we next show that $ \z(t) $ grows quadratically near $ t=\frac12 $.

\begin{lemma}\label{lem:zqd}
For any solution $ \z(\Cdot+\frac12) $ to the integral equation~\eqref{eq:zeq},
we have
\begin{align}
	\label{eq:zdq}
	\lim_{t\downarrow 0} \{ t^{-2}\z(t+\tfrac12) \} = \tfrac{2}{\sqrt{\pi}}.
\end{align}
\end{lemma}

\begin{remark}
For sufficiently smooth solutions to the \ac{PDE} \eqref{eq:PDE>},
one can easily calculate 
$ 
	\frac{d^2~}{dt^2}\z(\frac12) = -\frac18 \partial_{xxx} u_1(\frac12,0) 
	= \frac{2}{\sqrt{\pi}}
$
by differentiating \eqref{eq:StefBC} and \eqref{eq:HE>}.
Here, as we take the integral equation \eqref{eq:zeq} 
as the definition of the Stefan problem,
we prove Lemma~\ref{lem:zqd} by a different, indirect method,
which does not assume the smoothness of of $ \z $.
\end{remark}
\begin{proof}
We begin by deriving a useful identity.
Write
$ 	
	\int_0^t p(t-s,x) ds
	=
	-\int_0^t \int_{-\infty}^x \partial_{yy} \erf(t-s,y) ds,
$
use $ -\partial_{yy} \erf(t-s,y) = 2\partial_s \erf(t-s,y) $,
swap the double integrals, and integrate over $ s\in[0,t] $. 
With $ \erf(0,y) = \ind_{[0,\infty)}(y) $,
we obtain 
\begin{align}
	\label{eq:useful}
	\int_0^t p(t-s,x) ds 
	= 
	2\int_{-\infty}^{x} (\erf(t,y)-\ind_{[0,\infty)}(y)) ds
	=
	2\int_{-\infty}^{-|x|} \erf(t,y) ds.	
\end{align}

We now begin the proof of \eqref{eq:zdq}.
Let $ \Lambda(t,x) $ be as in \eqref{eq:LHSof}.
Recall from \eqref{eq:U1} that 
$ \partial_y \Ucs(\frac12,y) = -u_1(\frac12,y) $,
for $ u_1(\frac12,y) $ defined in \eqref{eq:u1}.
Integrating by parts followed by a change of variable $ y \mapsto \frac{y}{\sqrt{t}} $ yields
\begin{align}
	\notag
	\Lambda(t,x) 
	&= 
	-\int_0^\infty \partial_y\big(\Ucs(\tfrac12,0)-\Ucs(\tfrac12,y)\big) \erf(t,x-y) dy
\\
	\label{eq:LHSof:}
	&=
	\sqrt{t} \int_0^\infty u_1(\tfrac12,\sqrt{t} y) \erf(1,\tfrac{x}{\sqrt{t}}-y) dy.
\end{align}
From the explicit expression~\eqref{eq:u1} of $ u_1(\frac12,y) $,
we see that $ u_1(\frac12,\Cdot) \in \Csp^\infty(\bbR_+) \cap L^\infty(\bbR_+) $,
and that $ u_1(\frac12, 0)=2 $, $ \partial_y u_1(\frac12, 0) = \partial_{yy}u_1(\frac12,0) =0 $,
and $  -\partial_{yyy} \u(0) = \frac{16}{\sqrt{\pi}} =: a_3 $.
Using these properties
to Taylor-expand $ u_1(\frac12,\sqrt{t}y) $ in \eqref{eq:LHSof:} yields
\begin{align}
	\label{eq:LHS:Taylor}
	\Lambda(t,x) 
	= 
	t^{\frac12} \Lambda_0(\tfrac{x}{\sqrt{t}}) 
	- t^{2} \Lambda_3(\tfrac{x}{\sqrt{t}}) 
	+ t^{\frac52} \Lambda_4(t,x),
\end{align}
where $ \Lambda_4(t,x) $ is a \emph{bounded} remainder function 
in the sense that $ \lim\limits_{t\downarrow 0} \sup\limits_{x\in\bbR}|\Lambda_4(t,x)| <\infty $, 
and
\begin{align}
	\label{eq:Lambda0}
	\Lambda_0(x) &:= 2 \int_0^\infty \erf(1,x-y) dy
	=
	2 \int_{-\infty}^x \erf(1,y) dy,
\\
	\label{eq:Lambda3}
	\Lambda_3(x) &:=  \frac{a_3}{6} \int_0^\infty y^3 \erf(1,x-y) dy.
\end{align}
Insert the expression~\eqref{eq:LHS:Taylor} into \eqref{eq:zeq},
we obtain
\begin{align}
	\label{eq:qd:ineq:}
	t^{\frac12} \Lambda_0(\tfrac{w(t)}{\sqrt{t}}) 
	- t^{2} \Lambda_3(\tfrac{w(t)}{\sqrt{t}}) 
	+ t^{\frac12} \Lambda_4(t,w(t))
	=
	\int_0^t p(t-s,w(t)-w(s)) ds.
\end{align}

The strategy of the proof is to extract
upper and lower bounds on $ \frac{w(t)}{\sqrt{t}} $
from \eqref{eq:qd:ineq:}.
We begin with the upper bound.
On the r.h.s.\ of \eqref{eq:qd:ineq:},
using $ p(t-s,w(t)-w(s)) \leq p(t-s,0) $,
followed by applying the identity~\eqref{eq:useful}, we have that
\begin{align}
	\label{eq:qd:ineq}
	t^{\frac12} \Lambda_0(\tfrac{w(t)}{\sqrt{t}}) 
	- t^{2} \Lambda_3(\tfrac{w(t)}{\sqrt{t}}) 
	+ t^{\frac52} \Lambda_4(t,x)
	\leq 
	t^{\frac12} \Lambda_0(0).
\end{align}
Dividing \eqref{eq:qd:ineq} by $ t^{\frac12} $ and letting $ t\downarrow 0 $,
we conclude that $ \lim_{t\downarrow 0} \Lambda_0(\frac{w(t)}{\sqrt{t}}) = \Lambda_0(0) $.
As $ x\mapsto \Lambda_0(x) $ is strictly increasing,
we must have $ \lim_{t\downarrow 0} \frac{w(t)}{\sqrt{t}} = 0 $.
Now, dividing both sides of \eqref{eq:qd:ineq} by $ t^2 $,
and letting $ t\downarrow 0 $ using $ \lim_{t\downarrow 0} \frac{w(t)}{\sqrt{t}} = 0 $,
we further deduce that
\begin{align}
	\label{eq:wLambda3<}
	\lim_{t\downarrow 0} t^{-\frac32} \big(\Lambda_0(\tfrac{w(t)}{\sqrt{t}})-\Lambda_0(0) \big) 
	- \Lambda_3(0) 
	\leq 0.
\end{align}
From the explicit expression~\eqref{eq:Lambda0} of $ \Lambda_0(x) $,
we have that 
\begin{align}
	\label{eq:Lambda0:diff}
	\tfrac{d~}{dx} \Lambda_0(0) =1.
\end{align}
Using~\eqref{eq:Lambda0:diff} to Taylor-expanding 
the expression $ \Lambda_0(\tfrac{w(t)}{\sqrt{t}}) $ in \eqref{eq:wLambda3<}
to the first order,
we thus conclude the desired upper bound
$ 
	\limsup_{t\downarrow 0} \frac{w(t)}{t^2} \leq \Lambda_3(0) = \frac{a_3}{8} = \frac{2}{\sqrt{\pi}} 
$.

Having established the desired upper bound on $ \frac{w(t)}{t^2} $,
we now turn to the lower bound.
Let $ b:=\liminf_{t\downarrow 0} \frac{w(t)}{t^2} $.
Since $ 0\leq b\leq \frac{2}{\sqrt{\pi}}<\infty $,
there exists $ t_n\downarrow 0 $ such that
\begin{align}
	\label{eq:liminf:b}
	|\tfrac{w(t_n)}{t_n^2} - b| <\tfrac1n,
	\quad
	\tfrac{w(s)}{s^2} > b-\tfrac1n, \ \forall s\in (0,t_n].
\end{align}
As $ t\mapsto w(t) $ is non-decreasing,
by \eqref{eq:liminf:b} we have
\begin{align*}
	|w(t_n)-w(s)|
	=
	w(t_n)-s(s)
	\leq
	(bt^2_n+\tfrac{t^2_n}{n}) - (bs^2-\tfrac{s^2}{n})
	\leq
	b (t^2_n-s^2) + \tfrac{2t^2_n}{n},
	\quad
	\forall s\leq t_n.
\end{align*}
Taking the square of the preceding inequality further yields
\begin{align}
	\label{eq:wts}
	|w(t_n)-w(s)|^2 
	\leq 
	(\tfrac{2t_n^2}n)^2 + \tfrac{4bt_n^2(t^2_n-s^2)}n + b^2(t_n^2-s^2)^2,
	\quad
	\forall s\leq t_n.
\end{align}
Use this inequality~\eqref{eq:wts} to write
\begin{align*}
	p(&t_n-s,w(t_n)-w(s)) 
	\geq 
	p(t_n-s,\tfrac{2t_n^2}n)
	\exp\Big( 
		-\tfrac{1}{2(t_n-s)} \Big(\tfrac{4bt_n^2(t^2_n-s^2)}n + b^2(t^2_n-s^2)^2 \Big)
	\Big)
\\
	&=
	p(t_n-s,\tfrac{2t_n^2}n)
	\exp\Big( 
		-\tfrac{2bt_n^2(t_n+s)}n \Big)
	\exp\Big(
		-\tfrac{b^2}{2}(t_n-s)(t_n+s)^2
	\Big)
\\
	&\geq
	p(t_n-s,\tfrac{2t_n^2}n) 
	\exp\Big( -\tfrac{2bt_n^2(2t_n)}n \Big)
	\exp\Big( - \tfrac{b^2}{2}(t_n-s)(2t_n)^2 \Big).
\end{align*}
Within the last expression,
using $ e^{-\xi} \geq 1-\xi $
for $ \xi=\tfrac{b^2}{2}(t_n-s)(2t_n)^2 $,
and using 
$ 
	|p(t_n-s,\tfrac{2t_n^2}n) 
	\exp( -\tfrac{2bt_n^2(2t_n)}n )
	\leq
	\frac{1}{\sqrt{t_n-s}},
$
we obtain
\begin{align}
	\label{eq:pwts}
	p(t_n-s,w(t_n)-w(s)) 
	\geq
	p(t_n-s,\tfrac{2t_n^2}n) 
	\exp\big(-\tfrac{2bt_n^2(2t_n)}n \big)
	- \tfrac{b^2}{2}(t_n-s)^{\frac12}(2t_n)^2.
\end{align}
Now, integrate \eqref{eq:pwts} over $ s\in[0,t_n] $,
using the identity~\eqref{eq:useful} to obtain
\begin{align}
	\notag
	\int_0^{t_n} p(t_n-s,w(t_n)-w(s)) ds
	&\geq
	e^{-\frac{2bt_n^2(2t_n)}n}
	\int_0^{t_n} p(t_n-s,\tfrac{2t_n^2}n) ds
	-
	4b^2(t_n)^{\frac52}
\\
	\notag
	&=e^{-\frac{4bt_n^3}n} \sqrt{t_n} \Lambda_0( -\tfrac{2t_n^2}{n\sqrt{t_n}} )
	-
	4b^2(t_n)^{\frac52}
\\
	\label{eq:llllll}
	&\geq
	\sqrt{t_n} \Lambda_0( -\tfrac{2t_n^2}{n\sqrt{t_n}} ) - C(t_n)^{\frac52},
\end{align}
for some constant $ C<\infty $.
Now, set $ t=t_n $ in \eqref{eq:qd:ineq}
and combine the result with \eqref{eq:llllll}.
After dividing both sides of the result by $ t_n^2 $ and letting $ n\to\infty $,
we arrive at
\begin{align}
	\label{eq:rrrrr}
	\lim_{n\to\infty} 
	(t_n)^{-\frac32} \big( 
		\Lambda_0(\tfrac{w(t_n)}{\sqrt{t_n}})
		-\Lambda_0(-\tfrac{2}n (t_n)^\frac32 )
	\big)
	-
	\Lambda_3(0) \geq 0.
\end{align}
Using \eqref{eq:Lambda0:diff} to Taylor expand
the expressions $ \Lambda_0(\tfrac{w(t_n)}{\sqrt{t_n}}) $ 
and $ \Lambda_0(-\tfrac{2}n (t_n)^\frac32 ) $,
with\\
$ \liminf_{n\to\infty} \frac{w(t_n)}{t_n^2} = b $ (by~\eqref{eq:liminf:b}),
we conclude the desired lower bound $ b\geq \Lambda_3(0) =\frac{2}{\sqrt{\pi}} $.
\end{proof}

\section{Proof of Theorem~\ref{thm:hydro}}
\label{sect:hydro}
Equipped with the tools developed previously,
in this section prove Theorem~\ref{thm:hydro}.
To this end,
throughout this section we specialize $ (\phi^K_i(t):t\geq 0) $ 
to the push-the-laggard strategy~\eqref{eq:pushlaggard}.
Recalling that $ \tau_i $ denote the absorption of the $ i $-th particle $ X_i(t) $,
we let 
\begin{align}
	\label{eq:extt}
	\tau_\text{ext} := \max\nolimits_{i}{\tau_i},
	\quad
	\tau^K_\text{ext} := K^{-1}\tau_\text{ext}
\end{align}
denote the extinction times (unscaled and scaled).
Under the push-the-laggard strategy,
Proposition~\ref{prop:intXY}\ref{enu:intX} gives
\begin{align}
	\label{eq:int:abs:}
	\Uc_K(t,x) = 
	\Gc_K(t,x) + \int_{0}^{t\wedge\tau^K_\text{ext}} \pN(t-s,Z_K(s),x) ds + \rd_K(t,x),
\end{align}
We first establish a lower bound on the extinction time.

\begin{lemma}\label{lem:extT}
For any fixed $ T,n<\infty $, there exists $ C=C(T,n)<\infty $
such that
\begin{align}
	\label{eq:extT}
	\Pr( \tau^K_\text{ext} > T ) \geq 1 -CK^{-n}.
\end{align}
\end{lemma}
\begin{proof}
Consider the modified process $ (X^\text{ab}_i(t);t \geq 0)_{1\leq i\leq K} $
consisting of $ K $ independent Brownian motions starting at $ x=1 $ 
and absorbed once they reach $ x=0 $,
and let $ \tau'_\text{ext} := \inf \{ t : X_i(t)=0,\forall i \} $
denote the corresponding extinction time.
Under the natural coupling of $ (X^\text{ab}_i(t))_i $ and $ (X_i(t))_i $ 
(by letting them sharing the underlying Brownian motions),
we clearly have $ \tau_\text{ext} \geq \tau'_\text{ext} $.
For the latter, it is straightforward to verify that
\begin{align*}
	\Pr( \tau'_\text{ext} \leq KT ) 
	=
	\Big( \Pr\Big( \inf_{t\leq T} (B(Kt)+1) \leq 0 \Big) \Big)^K
	\leq 
	\exp(-\tfrac{1}{C(T)}K^{1/2}),
\end{align*}
where $ B(\Cdot) $ denotes a standard Brownian motion.
From this the desired result follows.
\end{proof}
\noindent
By Lemma~\ref{lem:extT}, toward the end of proving Theorem~\ref{thm:hydro},
without loss of generality we remove the localization 
$ \Cdot\wedge\tau^K_\text{ext} $ in \eqref{eq:int:abs:}.

Next, using the expression \eqref{eq:Gc:pN} of $ \Gc_K(t,x) $,
from the heat kernel estimate \eqref{eq:p:Holdx} we have
\begin{align}
	&
	\label{eq:GK:est}
	|\Gc_K(t,x) - 2p(t,x)| \leq	C(\alpha) K^{-\frac{\alpha}{2}} t^{-\frac{1+\alpha}{2}},&
	&
	\forall\alpha\in(0,1),
\\
	&
	\label{eq:pN:est}
	\int_0^t|\pN(t-s,x,z)-\pN(t-s,x,z')| ds
	\leq
	C(\alpha') |z-z'|^{\alpha'} t^{\frac{1-\alpha'}{2}},&
	&
	\forall \alpha' \in (0,1).
\end{align}
For any fixed $ \gamma\in (0,\frac14) $ and $ \alpha'\in(0,1) $,
taking the difference of \eqref{eq:Ucs} and \eqref{eq:int:abs:},
followed by using the estimates \eqref{eq:GK:est}--\eqref{eq:pN:est}
and \eqref{eq:rd:bd},
we obtain
\begin{align*}
	&|\Uc_K(t,x)-\Ucs(t,x)|
\\
	\leq&
	|\Gc_K(t,x)-2p(t,x)| 
	+ 
	\int_0^t|\pN(t-s,x,Z_K(t))-\pN(t-s,x,z')| ds + |\rd_K(t,x)|
\\
	\leq&
	C(\gamma) t^{-\frac34} K^{-\gamma}
	+ 
	C(\gamma,\alpha') \sup_{s\leq T}|Z_K(s)-z(s)|^{\alpha'},
	\quad
	\forall x\in\bbR, t\leq T,
\end{align*}
with probability $ \geq 1-C(n,T)K^{-n} $.
From this we see that the hydrodynamic limit~\eqref{eq:hydro:U} of $ \Uc_K(t,x) $
follows immediately from the hydrodynamic limit~\eqref{eq:hydro:Z} of $ Z_K $.
Focusing on proving \eqref{eq:hydro:Z} hereafter,
in the following we settle \eqref{eq:hydro:Z} 
in the absorption phase and the moving boundary phase separately.
For technical reasons, instead of using $ \t=\frac12 $ 
as the separation of these two phases,
in the following we use $ \frac{1}{2}+ \frac17 K^{-2\gamma} $ for the separation of the two phases,
where $ \gamma\in(0,\frac{1}{96}) $ is fixed.
More precisely, 
the desired hydrodynamic result~\eqref{eq:hydro:Z} follows immediately
from the following two propositions
(by setting $ \beta=\gamma $ in Part\ref{enu:Zhy:abs}):
\begin{proposition}
\label{prop:Zhy}
For any fixed $ \gamma<\gamma_1\in(0,\frac{1}{96}) $ and $ n<\infty $, 
there exists $ C=C(\gamma,\gamma_1,n)<\infty $ such that
\begin{enumerate}[label=(\alph*)]
\item \label{enu:Zhy:abs} for all $ \beta \leq 4\gamma_1 $ and $ K<\infty $,
\begin{align}
	\label{eq:Zhydro:abs:}
	\Pr \Big( 
			|Z_K(t)-\z(t)| \leq CK^{-\beta},
			\
			\forall t \in[0, \tfrac{1}{2}+ \tfrac17 K^{-2\beta}]
	\Big) 
		\geq 1 - CK^{-n};
\end{align}
\item \label{enu:Zhy:mvbdy}
for all $ K<\infty $,
\begin{align}
\label{eq:Zhy:mvbdy}
	\Pr \Big( 
			|Z_K(t)-\z(t)| \leq CK^{-\gamma},
			\
			\forall t\in[\tfrac12 + \tfrac17 K^{-2\gamma}, T] 
	\Big) 
		\geq 1 - CK^{-n}.
\end{align}
\end{enumerate}
\end{proposition}
\noindent
We settle Proposition~\ref{prop:Zhy}\ref{enu:Zhy:abs}--\ref{enu:Zhy:mvbdy}
in Sections~\ref{sect:absrb}--\ref{sect:mvbdy} in the following,
respectively.
To this end, we fix $ \gamma<\gamma_1\in(0,\frac{1}{96}) $, 
$ n<\infty $ and $ T<\infty $,
and, to simplify notations,
 use $ C<\infty $ to denote a generic constant that depends only on $ \gamma,\gamma_1,n,T $.
\subsection{Proof of Proposition~\ref{prop:Zhy}\ref{enu:Zhy:abs}}
\label{sect:absrb}
Fix $ \beta \leq 4\gamma_1 $.
We begin with a reduction. 
Since $ \z(t)|_{t\leq\frac12}=0 $, by Lemma~\ref{lem:zqd},
we have $ \sup_{t \leq \frac{1}{2}+ \frac17 K^{-2\beta}} |\z(t)| \leq CK^{-4\beta} \leq CK^{-\beta} $.
From this, we see that is suffices to prove
\begin{align}
\label{eq:Zhydro:abs}
	\Pr \Big( 
			Z_K(t) \leq K^{-\beta},
			\
			\forall t \leq \tfrac{1}{2}+ \tfrac17 K^{-2\beta}
	\Big) 
		\geq 1 - CK^{-n}.
\end{align}

To the end of showing \eqref{eq:Zhydro:abs},
we recall the following classical result from \cite{feller71}.
\begin{lemma}[{\cite[Chapter X.5, Example (c)]{feller71}}] \label{Felconf}
Let $ (B(t); t \geq 0) $ be a standard Brownian motion (starting from $ 0 $), and let $ 0<a<b<\infty $. 
Defining $ \rho(t,a,b) := \Pr ( 0< B(s)+a < b,\forall s \leq t) $, we have
\begin{equation}\label{eigenexp}
	\rho(t,a,b) 
	= 
	\sum_{n=0}^{\infty} 
	\frac{4}{(2n+1) \pi} 
	\sin \Big( \frac{(2n+1) \pi a}{b} \Big) 
	\exp \Big( -\frac{(2n+1)^2 \pi^2}{2b^2}t \Big).
\end{equation}
\end{lemma}
\noindent
With $ \beta\leq4\gamma_1<\frac{1}{24} $,
we have $ 4\beta<\frac12 - 8\beta $.
Fixing $ \alpha \in (4\beta, \frac12 - 8\beta) \subset (0,\frac12) $,
we begin with a short-time estimate:
\begin{lemma} \label{lem:smtime}
There exists $ C<\infty $ such that 
\begin{align}
	\label{eq:sm}
	\Pr \Big(  Z_K(t) \leq K^{-\alpha}, \ \forall t \leq K^{-2\alpha} \Big) 
	\geq 1-C K^{-n}.
\end{align}
\end{lemma}
\begin{proof}
We consider first the modified process 
$ (\widehat{X}^{\text{ab}}_i(t); t \geq 0)_{i=1}^K $,
which consists of $ K $ independent Brownian motions starting at $ x=1 $, 
and absorbed at $ x=0 $ and $ x=\frac{1}{2} K^{\frac12-\alpha} $. 
Let $ \hat{X}^{K,\text{ab}}_i(t) := \frac{1}{\sqrt{K}} \hat{X}^{\text{ab}}_i(Kt) $
denote the scaled process, and consider 
\begin{align}
	\label{eq:hatN}
	\widehat{N}^\text{ab} 
	:= 
	\# \set{ i : 0<\widehat{X}^{K,\text{ab}}_i(t)<\tfrac12 K^{-\alpha}, \ \forall t \leq K^{-2\alpha} },
\end{align}
the number of surviving $ \widehat{X}^{K,\text{ab}} $-particles of up to time $ K^{-2\alpha} $.
Let
\begin{align*}
	\rho^*_K := \rho(K^{-2\alpha},K^{-\frac12},\tfrac12 K^{-\alpha})
	:=
	\Pr\big( 0<\tfrac{1}{\sqrt{K}}(B(Kt)+1)<\tfrac12 K^{-\alpha}, \ \forall t \leq K^{-2\alpha} \big).
\end{align*}
From the definition~\eqref{eq:hatN}
we see that $ \widehat{N}^\text{ab} $ 
is the sum of i.i.d.\ Bernoulli$(\rho^*_K)$ random variables.
Hence, by the Chernov bound we have
\begin{equation}\label{Chernoff}
	\Pr (\widehat{N}^\text{ab} \geq \tfrac{1}{2}K \rho^*_K )
	\geq  
	1- \exp \big( -\tfrac{1}{8}K \rho^*_K \big).
\end{equation}
Specialize \eqref{eigenexp} at 
$ (t,a,b) = (K^{-2\alpha}, K^{-\frac12}, \frac12 K^{-\alpha}) $ to obtain
\begin{align*}
	\rho^*_K 
	= 
	\sum_{n=0}^{\infty} \rho'_{K,n}
	\exp \big( -2(2n+1)^2 \pi^2 \big),
	\
	\text{ where }
	\rho'_{K,n} := \frac{4}{(2n+1) \pi} \sin ( 2(2n+1)\pi K^{\alpha-\frac12} ).
\end{align*}
With $ \alpha<\frac12 $, we have
$ \lim_{K\to\infty} (K^{\frac12-\alpha}\rho'_{K,n}) = 8 $ and $ |\rho'_{K,n}| \leq 8 K^{\alpha-\frac12} $,
and it is straightforward to show that
\begin{align*}
	\lim_{K\to\infty} (K^{\frac12-\alpha}\rho^*_K)
	= 
	8 \sum_{n=0}^{\infty} \exp \big( -2(2n+1)^2 \pi^2 \big) > 0.
\end{align*}
Consequently, $ \rho^*_K \geq \frac{1}{C} K^{\alpha-\frac{1}{2}} $.
Inserting this into \eqref{Chernoff}, we arrive at
\begin{align*}
	\Pr (\widehat{N}^\text{ab} \geq \tfrac{1}{C}K^{\alpha+\frac{1}{2}} ) 
	\geq  
	1- \exp(-\tfrac{1}{C} K^{\alpha+\frac{1}{2}} )
	\geq
	1 - CK^{-n}.
\end{align*}

Next, we consider the process $ (X^{\text{ab}}_i(t);t\geq 0)_{i=1}^K $,
consisting of $ K $ independent Brownian motions starting at 
$ x= 1 $ and absorbed only at $ x=0 $,
coupled to $ (\widehat{X}^{\text{ab}}_i(t))_i $ by the natural coupling 
that each $ i $-th particle share the same underlying driving Brownian motion.
Let 
$ X^{\text{ab},K}_i(t) := \frac{1}{\sqrt{K}} X^{\text{ab}}_i(Kt) $
denote the scaled process, let
$ 
	\Gamma := \{X^{\text{ab},K}_i(K^{-2\alpha}) : 
	0<X^{\text{ab},K}_i(t)<\frac12 K^{-\alpha}, \forall t \leq K^{-2\alpha}\} 
$
denote the set of all $ X^{\text{ab},K} $-particles 
that stay within $ (0,\frac12 K^{-\alpha}) $
for all $ t \leq K^{-2\alpha} $, 
and let $ N^{\text{ab}} := \#\Gamma $.
We clearly have $ N^{\text{ab}} \geq \widehat{N}^\text{ab} $, and therefore
\begin{equation} \label{compfZ}
	\Pr (N^{\text{ab}}(K^{-2\alpha}) \geq \tfrac{1}{C} K^{\alpha+\frac{1}{2}}) \geq 1 - C K^{-n}.
\end{equation}

Now, couple $ (X^{\text{ab},K}(t)) $ and $ (X^K(t)) $ 
by the aforementioned natural coupling.
On the event $\{N^{\text{ab}} \geq \frac{1}{C} K^{\alpha+\frac{1}{2}}\}$, 
to move all $ X^K $-particles in $ \Gamma $ to the level $ x=K^{-\alpha} $
requires at least a drift of 
$
	N^{\text{ab}} (\frac{1}{2}K^{-\alpha}) \geq \tfrac{1}{C} K^{\frac{1}{2}},
$
while the total amount of (scaled) drift at disposal is $ K^{-2\alpha+\frac{1}{2}} $.
This is less than $ \tfrac{1}{C} K^{\frac{1}{2}} $ for all large enough $ K $.
Consequently, the desired result \eqref{eq:sm} follows from \eqref{compfZ}.
\end{proof}

Equipped with the short-time estimate~\eqref{eq:sm},
we now return to showing \eqref{eq:Zhydro:abs}.
Consider the threshold function
\begin{align}
	z^*(t) =  K^{-\alpha} \ind_{\{t\leq K^{-2\alpha}\}} 
	+ (\sqrt{t} K^{-\beta}) \ind_{\{t>K^{-2\alpha}\}},
\end{align}
and the corresponding hitting time 
$ \tau:=\inf\{t \in\bbR_+ : Z_K(t) \geq z^*(t) \} $.
It suffices to show $ \Pr( \tau > \frac{1}{2} + \frac{1}{7} K^{-2\beta} ) \geq 1 - CK^{-n} $.
To this end, by Lemma~\ref{lem:smtime}, without loss of generality we assume
$ \tau \in (K^{-2\alpha},1) $.
As the trajectory of $Z_{K}$ is continuous except when it hits $ 0 $, we have $Z_K(\tau) \geq z^{*}(\tau)$.
Hence at time $ \tau $, no particle exists between $ 0 $ and $ z^*(\tau) $, 
or equivalently $ \Uc_K(\tau,z^{*}(\tau))=\Uc_K(\tau,0). $
With this, taking the difference of \eqref{eq:int:abs:} at $ x=z^*(\tau) $ and at $ x=0 $,
and multiplying the result by $ \sqrt{ \frac{\pi \tau}{2} } $, we obtain
\begin{align*}
	h_1 = h_2 + \sqrt{ \tfrac{\pi \tau}{2} }(\rd_K(\tau,z^{*}(\tau)) - \rd_K(\tau,0)),
\end{align*}
where $ h_1 := \sqrt{ \frac{\pi \tau}{2} }( \Gc_K(\tau,0)-\Gc_K(\tau,z^*(\tau))) $,
$ h_2 := \sqrt{ \frac{\pi \tau}{2} } \int_0^{\tau} f_2(s,Z_K(s),z^*(\tau)) ds $, and
\begin{align}\label{eq:h2}
	f_2(s,z,z') := \pN(\tau-s,z,z')-\pN(\tau-s,z,0).
\end{align}
Further using \eqref{eq:rd:bd},
for fixed $ \delta\in (0,\frac{1-2\alpha}{4} - 4\beta) $,
to control the remainder term $ (\rd_K(\tau,z^{*}) - \rd_K(\tau,0)) $,
with $ \tau \geq K^{-2\alpha} $, we have 
\begin{align}
	\label{eq:lrc}
	h_1 \leq h_2 + C K^{ \frac{2\alpha-1}{4}+\delta },
\end{align}
with probability $ 1 -CK^{-n} $.
Given the inequality~\eqref{eq:lrc},
the strategy of the proof is to extract the bound
$ \tau \geq \frac12 + \frac{1}{7} K^{-2\beta} $ from \eqref{eq:lrc}.
To this end,
we next derive a lower bound on $ h_1 $ and an upper bound on $ h_2 $.

With $ \Gc_K(t,x) $ defined as in \eqref{eq:Gc}, we have
\begin{align*}
	h_1 = h_1(K^{-\beta}),
	\
	\text{ where }
	h_1(a)
	= 
	\sqrt{K\tau}
	\int_0^\frac{1}{\sqrt{K\tau}} 
	\big( e^{-\frac{y^2}{2}} - \tfrac12 e^{-\frac{(y+a)^2}{2}} - \tfrac12 e^{-\frac{(y-a)^2}{2}} \big) dy.
\end{align*}
Taylor-expanding $ h_1(a) $ to the fifth order gives
$
	h_1(a) \geq  a^2 h_{12} + a^4 h_{14} -  Ca^6,
$
where 
$ h_{12} :=	\sqrt{K\tau} \int_0^\frac{1}{\sqrt{K\tau}} e^{-y^2/2}(\frac12-\frac12 y^2) dy $
and $ h_{14} := \sqrt{K\tau}	\int_0^\frac{1}{\sqrt{K\tau}} e^{-y^2/2}(-\frac{1}{8}+\frac{1}{4}y^2-\frac{1}{24}y^4) dy $.
Further Taylor-expanding $ h_{12} $ and $ h_{14} $ in $ \frac{1}{\sqrt{K\tau}} $ yields
$ h_{12} \geq \frac12 - \frac{C}{K\tau} $ and $ h_{14} \geq -\frac{1}{8} - \frac{C}{K\tau} $,
and therefore
\begin{align}\label{eq:h1bd}
	h_1 
	\geq 
	\tfrac12 a^2 - \tfrac{1}{8} a^4 - C a^2 ( \tfrac{1}{K\tau} + a^4)
	\geq 
	\tfrac12 a^2 - \tfrac{1}{8} a^4 - C a^2 ( K^{2\alpha-1} + a^4),	
	\text{ for }
	a= K^{-\beta}.
\end{align}

Turning to estimating $ h_2 $,
we first observe that the function $ f_2(s,z,z') $ 
as in \eqref{eq:h2} increases in $ z $, $ \forall z \leq z' $,
as is readily verified by taking derivative as follows:
\begin{align*}
	&
	\sqrt{2\pi(\tau-s)^3} \partial_z f_2(s,z,z') 
\\
	=& z' (e^{-(z-z')^2/2}-e^{-(z+z')^2/2}) - z (e^{-(z-z')^2/2}+e^{-(z+z')^2/2}) + 2z e^{-z^2/2}
\\
	\geq&
	z (e^{-(z-z')^2/2}-e^{-(z+z')^2/2}) - z (e^{-(z-z')^2/2}+e^{-(z+z')^2/2}) + 2z e^{-z^2/2} \geq 0.
\end{align*}
Now, since $ t\mapsto z^*(t) $ is increasing for all $ t \geq K^{-2\alpha} $,
to obtain an upper bound on $ h_2 $ we replace $ Z_K(s) $ with $ z^*(\tau) $ for $ s \geq K^{-2\alpha} $.
Further, with $ \int_0^{K^{-2\alpha}} \pN(\tau-s,z,z') ds \leq C K^{-\alpha} $,
we obtain $ h_2 \leq C K^{-\alpha} + \sqrt{ \frac{\pi \tau}{2} } \int_0^{\tau} f_2(s,z^*(\tau),z^*(\tau)) ds $.
With $ z^*(\tau) = K^{-\beta} \sqrt{\tau} $,
the last integral is evaluated explicitly by using \eqref{eq:pId},
yielding
\begin{align*}
	h_{2} &\leq
	\tau h_2(K^{-\beta}) + CK^{-\alpha},
	\quad
	\text{ where }
	h_2(a)
	= 1 + e^{-2a^2}-2e^{-a^2/2}
		+ 2 a \int_{a}^{2a} e^{-y^2/2} dy.
\end{align*}
Taylor-expanding $ h_2(a) $ to the fifth order,
we further obtain 
\begin{align}\label{eq:h2bd}
	h_{2} \leq \tau (a^2 - \tfrac{7}{12} a^4 + Ca^6) + C K^{-\alpha},
	\quad
	\text{ for } a= K^{-\beta}.
\end{align}

Now, combining \eqref{eq:lrc}--\eqref{eq:h2bd},
we arrive at
\begin{align}
	\label{eq:t*bd}
	\tau \geq 
	\frac{
			\frac12 - \tfrac{1}{8} a^2 - C ( K^{2\alpha-1} + a^4
			+a^{-2}K^{-\alpha} + a^{-2}K^{\frac{2\alpha-1}{4}+\delta})
		}{
		1 - \tfrac{7}{12} a^2 + Ca^4
		},
	\
	\text{ for } a=K^{-\beta}.
\end{align}
With $ \alpha $ and $ \delta $ chosen as in the preceding,
it is now straightforward to check that, for $ a=K^{-\beta} $,
\begin{align*}
	&\frac{
			\frac12 - \tfrac{1}{8} a^2 - C ( K^{2\alpha-1} + a^4
			+a^{-2}K^{-\alpha} + a^{-2}K^{\frac{2\alpha-1}{4}+\delta})
		}{
		1 - \tfrac{7}{12} a^2 + Ca^4
		}
\\
	=&
	\frac{ \frac12 - \tfrac{1}{8} a^2 }{ 1 - \tfrac{7}{12} a^2 } + (\text{ higher order terms })
\\
	=&
	\tfrac12 + \tfrac{1}{6} K^{-2\gamma} + (\text{ higher order terms }).
\end{align*}
From this the we conclude the desired result: $ \tau>\frac12 + \frac{1}{7} K^{-2\gamma} $,
with probability $ \geq 1-CK^{-n} $.

\subsection{Proof of Proposition~\ref{prop:Zhy}\ref{enu:Zhy:mvbdy}}
\label{sect:mvbdy}
To simplify notations, 
we let $ \sigma_K := \frac12+\frac17 K^{-2\gamma} $.
Define the scaled distribution function of surviving $ X $-particles as
\begin{align}
	\label{eq:U}
	\U_K(t,x) := \tfrac{1}{\sqrt{K}} \# \{ 0<X^K_i(t) \leq x \}
	=
	\langle \emK_t, \ind_{(0,x]} \rangle,
\end{align}
and, to simplify notations, we let
\begin{align}
	\label{eq:Uss}
	\Uss(x) := \Ucs(\tfrac12,0) - \Ucs(\tfrac12,x) = \int_0^x u_1(\tfrac12,y) dy,
\end{align}
where $ u_1(t,y) $ is defined in \eqref{eq:u1}.
Recall that $ \gamma<\gamma_1\in(0,\frac{1}{96}) $ are fixed.
Fix furthering $ \gamma_3<\gamma_2 \in (\gamma,\gamma_1) $,
we begin with an estimate on $ \U_K(t,x) $:

\begin{lemma}\label{lem:Ubd}
There exists $ C<\infty $ such that
\begin{align}
	\label{eq:Ubd12}
	\Pr\Big( 
		|\U_K(\tfrac12,x)-\Uss(x)| \leq CK^{-4\gamma_2},
		\
		\forall x\in\bbR
	\Big) 
	\geq 1 - CK^{-n},
\\
	\label{eq:Ubd}
	\Pr\Big( 
		|\U_K(\sigma_K,x)-\Uss(x)| \leq K^{-\gamma},
		\
		\forall x\in\bbR
	\Big) 
	\geq 1 - CK^{-n}.
\end{align}
\end{lemma}
\begin{proof}
With $ \U_K(t,x)=\Uc_K(t,0)-\Uc_K(t,x) $ 
and $ \Uss(x) $ defined in \eqref{eq:Uss},
we have that
$
	|\U_K(t,x)-\Uss(x)| \leq 2 \sup_{y\in\bbR} |\Uc_K(t,y)-\Ucs(\frac12,y)|.
$
To bound the r.h.s.,
we take the difference of the integral identities \eqref{eq:int:abs:} and \eqref{eq:Ucs}
to obtain
\begin{align}
	\notag
	|\Uc_K(t,&x)-\Ucs(\tfrac12,x)|
\\
	\label{eq:Uc12:esti:1}
	\leq&
	|\Gc_K(t,x)-2p(t,x)| + |2p(t,x)-2p(\tfrac12,x)|
\\
	\label{eq:Uc12:esti:2}
	&
	+\int_0^{t} |\pN(t-s,Z_K(s),x) - \pN(t-s,\z(s),x) | ds
	+
	\int_\frac12^t \pN(t-s,\z(s),x) ds
\\
	\label{eq:Uc12:esti:3}
	&
	+ |\rd_K(t,x)|.
\end{align}
We next bound the terms in \eqref{eq:Uc12:esti:1}--\eqref{eq:Uc12:esti:3} in sequel:
\begin{itemize}
\item[-] Using \eqref{eq:GK:est} for $ \alpha= 2\gamma_2 $ yields
	$ |\Gc_K(t,x)-2p(t,x)| \leq CK^{-4\gamma_2} $;
\item[-] Using \eqref{eq:p:Holdt} for $ \alpha=1 $, gives
	$	|2p(t,x)-2p(\tfrac12,x)| \leq C|\tfrac12-t|, $
	$ \forall t \geq \frac12 $;
\item[-] Using \eqref{eq:pN:est} for $ \alpha'=\frac{\gamma_2}{\gamma_1} $
	and \eqref{eq:Zhydro:abs:} for $ \beta=4\gamma_1 $, we have\\
	$  
		\int_0^{t} |\pN(t-s,Z_K(s),x) - \pN(t-s,\z(s),x) | ds
		\leq
		C \sup_{s\leq t} |Z_K(s)-\z(s)|^{\alpha'}
		\leq
		C K^{-4\gamma_2},
	$\\
	with probability $ \geq 1-CK^{-n} $;
\item[-] Using $ \pN(t-s,\z(s),s) \leq \frac{2}{\sqrt{2\pi(t-s)}} $, we obtain
	$ \int_\frac12^t \pN(t-s,\z(s),x) ds \leq \sqrt{\frac{2}{\pi}|t-\frac12|} $.
\item[-] Using \eqref{eq:rd:bd}, we have
	$  |\rd_K(t,x)| \leq CK^{-4\gamma_2} $, $ \forall t\in[\frac12,\sigma_K] $, $ x\in\bbR $,
	with probability $ \geq 1-CK^{-n} $.
\end{itemize}
\noindent
Combining these bounds yields
\begin{align}
	\label{eq:UcUcs}
	|\Uc_K(t,x)-\Ucs(\tfrac12,x)|
	\leq 
	C K^{-4\gamma_2}+C|t-\tfrac12| + \sqrt{\tfrac{2}{\pi}|t-\tfrac12|},
	\quad
	\forall t\in [\tfrac12,\sigma_K],
\end{align}
with probability $ \geq 1-CK^{-n} $.
Substituting in $ t=\frac12 $ in \eqref{eq:UcUcs} yields \eqref{eq:Ubd12}.
Similarly,
substituting in $ t=\sigma_K $ in \eqref{eq:UcUcs},
with $ |\sigma_K-\tfrac12| = \frac{K^{-2\gamma}}7 $,
we have, with probability $ \geq 1-CK^{-n} $,
\begin{align*}
	|\Uc_K(\sigma_K,x)-\Ucs(\tfrac12,x)|
	\leq 
	C K^{-4\gamma_2}+CK^{-2\gamma} + \sqrt{\tfrac{2}{7\pi}} K^{-\gamma}
	<
	K^{-\gamma},
\end{align*}
for all $ K $ large enough.
This concludes \eqref{eq:Ubd}.
\end{proof}

Recall the definition of Atlas models from the beginning of Section~\ref{sect:Int}.
Our strategy of proving Proposition~\ref{prop:Zhy}\ref{enu:Zhy:mvbdy}
is to \emph{reduce} the problem of the particle system $ (X(s);s\geq \sigma_K) $
to a problem of certain Atlas models $ (\Yup(t):t\geq 0) $ and $ (\Ylw(t):t\geq 0) $, 
constructed as follows.
To construct such Atlas models,
recalling the expression of $ u_1(\frac12,x) $ from \eqref{eq:u1},
we define
\begin{align}
	\label{eq:uup}
	\uup(x) &:= 
	\left\{\begin{array}{l@{,}l}
		u_1(\tfrac12, x) &	\text{ when } x\geq K^{-\gamma},
	\\
		0 &	\text{ when }  x < K^{-\gamma},
 	\end{array}\right.
\\
	\label{eq:ulw}
	\ulw(x) &:= 
	\left\{\begin{array}{l@{,}l}
		u_1(\tfrac12, x) &	\text{ when } x>0,
	\\
		u_1(\tfrac12,0) = 2 &	\text{ when }  -K^{-4\gamma_3} \leq x \leq 0,
	\\
		0	&	\text{ when } x < -K^{-4\gamma_3}.
 	\end{array}\right.
\end{align}
Adopting the notation $ \PPP(f(x)) $ for a Poisson point process on 
$ \bbR $ with density $ f(x) $,
for each $ K<\infty $
we let $ (\Yup(t;K):t\geq 0) $ and $ (\Ylw(t;K):t\geq 0) $ be Atlas models
starting from the following initial conditions
\begin{align}
	\label{eq:Yic:}
	(\Yup_i(0;K))_{i} \sim \PPP\big( \uup(\tfrac{x}{\sqrt{K}}) \big),
	\quad
	(\Ylw_i(0;K))_{i} \sim \PPP\big( \ulw(\tfrac{x}{\sqrt{K}}) \big),
\end{align}
and let $ \Wup(t;K) :=\min_{i} \Yup_i(t;K) $ and $ \Wlw(t;K) := \min_{i} \Ylw_i(t;K) $
denote the corresponding laggards.
\begin{remark}
\label{rmk:Y:Kdep}
The notations $ \Yup_i(t;K) $, etc., are intended to highlight
the dependence on $ K $ of the processes,
as is manifest from \eqref{eq:Yic:}.
To simplify notations, however,
hereafter we omit the dependence, and write $ \Yup_i(t;K)= \Yup_i(t) $, etc., 
unless otherwise noted.
\end{remark}
\noindent
We let $ \YKup_i(t) := \frac{1}{\sqrt{K}} \Yup_i(Kt) $,
denote the scaled process,
and similarly for $ \YKlw_i(t)$, $ \Wup_K(t) $ and $ \Wlw_K(t) $.
Under these notations, equation~\eqref{eq:Yic:} translates into
\begin{align}
	\label{eq:Yic}
	(\YKup_i(0))_{i} \sim \PPP\big( \uup(x) \big),
	\quad
	(\YKlw_i(0))_{i} \sim \PPP\big( \ulw(x) \big).
\end{align}
We let $ \Vup_K(t,x) := \frac{1}{\sqrt{K}} \#\{ \YKup_i(t) \leq x \} $
and $ \Vlw_K(t,x) := \frac{1}{\sqrt{K}} \#\{ \YKlw_i(t) \leq x \} $
denote the corresponding scaled distribution functions.

Having introduced the Atlas models $ \Yup $ and $ \Ylw $,
we next establish couplings that \emph{relate} these models
to the relevant particle system $ X $.
Recall the definition of the extinction time $ \tau^K_\text{ext} $ from \eqref{eq:extt}.
We let
\begin{align}
	\label{eq:abs>12}
	\tau^K_\text{abs} := \inf\{ t>\sigma_K : Z_K(t)=0 \}
\end{align}
denote the first absorption time (scaled by $ K^{-1} $) after $ \sigma_K $.

\begin{lemma} \label{lem:couple}
There exists a coupling of $ (X^K(s+\frac12);s\geq 0) $ and $ (\Ylw^K(s);s\geq 0) $ 
under which
\begin{align} 
	\label{eq:cplLw}
	\Pr \big( 
		\Wlw_K(s) \leq Z_K(\tfrac{1}{2} + s),
		\
		\forall s+\tfrac12 < \tau^K_\text{ext} 
		\big) 
	\geq 
	1 - CK^{-n}.
\end{align}
Similarly,
there exists a coupling of $ (X^K(s+\sigma_K);s\geq 0) $ and $ (\Yup^K(s);s\geq 0) $ 
under which
\begin{equation} 
	\label{eq:cplUp}
	\Pr \big( 
		\Wup_K(s) \geq Z_K(s+\sigma_K),
		\
		\forall s+\sigma_K < \tau^K_\text{ext} \wedge \tau^K_\text{abs}
		\big) 
	\geq 
	1 - CK^{-n}.
\end{equation}
\end{lemma}
\noindent
The proof requires a coupling result from~\cite{sarantsev15}:
\begin{lemma}[{\cite[Corollary~$3.9$]{sarantsev15}}]
\label{lem:sar:cmp}
Let $ (Y_i(s);s\geq 0)_{i=1}^{m} $ and $ (Y'_i(s);s\geq 0)_{i=1}^{m'} $ be Atlas models,
and let $ \W(s) $ and $ \W'(s) $ denote the corresponding laggards.
If $ Y'(0) $ dominates $ Y(0) $ componentwisely in the sense that
\begin{align}
	\label{eq:domin}
	m' \leq m,
	\quad
	Y'_{i}(0) \geq  Y_i (0),
	\
	i=1,\ldots,m',
\end{align}
then there exists a coupling of $ Y $ and $ Y' $ (for $ s>0 $)
such that the dominance continues to hold for $ s>0 $, i.e.\
$ Y'_{i}(s) \geq  Y_i (s) $, $ i=1,\ldots,m' $.
In particular, $ \W'(s) \geq W(s) $.
\end{lemma}

\begin{proof}[Proof of Lemma~\ref{lem:couple}]
As will be more convenient for the notations for this proof,
we work with \emph{unscaled} processes $ X(s+\frac12K) $, $ X(s+\sigma_KK) $ and $ Y(s) $, 
and construct the coupling accordingly.

We consider first $ \Ylw $ and prove \eqref{eq:cplLw}.
At $ s=0 $, order the particles as
$ (\Wlw(0)=\Ylw_{1}(0) \leq \Ylw_{2}(0) \leq \cdots) $, 
and $ (Z(\frac12K)= X_1(\frac{1}{2}K) \leq X_2(\frac{1}{2}K) \leq \cdots) $.
We claim that, regardless of the coupling, 
the following holds with probability $ 1-CK^{-n} $:
\begin{align} 
	\label{eq:desired}
	\#\{ \Ylw_{i}(0)\} \geq  \#\{X_i(\tfrac{1}{2}K)\},
	\quad
	\text{and }
	\Ylw_{i}(0) \leq X_i (\tfrac{1}{2}K),
	\ 
	\forall 1\leq i \leq \#\{ X_j(\tfrac{1}{2}K) \}.
\end{align}
Recalling from \eqref{eq:U} that $ \U_K(t,x) $ 
denotes the \emph{scaled} distribution function of $ X(t) $,
with $ \U_K(t,x)|_{x<0}=0 $,
we see that \eqref{eq:desired} is equivalent to the following
\begin{align}
\label{eq:gaol}
	\Pr\big(
		\Vlw_K(0,x) \geq \U_K(\tfrac{1}{2},x), 
		\
		\forall x\in \bbR_+
	\big)
	\geq 1-C K^{-n}.
\end{align}
To see why \eqref{eq:gaol} holds,
with $ (\YKlw_i(0))_i $ distributed in \eqref{eq:Yic},
we note that $ x \mapsto \sqrt{K}\Vlw_K(0,x) $, $ x \in [-K^{-4\gamma_3},\infty) $
is an inhomogeneous Poisson process with density $ \sqrt{K} \ulw(x) $.
From this,
it is standard (using Doob's maximal inequality and the \ac{BDG} inequality)
to show that
\begin{align}
	\label{eq:Vlw:bd:}
	\Big\Vert \sup_{x\in\bbR} 
	\Big| \Vlw_K(0,x) - \int_{-K^{-4\gamma_3}}^{x} \ulw(y) dy \Big| 
	\ \Big\Vert_m
	\leq
	C(m) K^{-\frac14},
	\quad
	\forall m \geq 2.
\end{align}
Further, with $ \ulw $ defined in \eqref{eq:ulw},
we have 
\begin{align*}
	\int_{-K^{-4\gamma_3}}^{x} \ulw(y) dy 
	= 	
	\Uss(x) + 2K^{-4\gamma_3},
	\quad
	\forall x \geq 0.
\end{align*}
Inserting this into \eqref{eq:Vlw:bd:},
followed by using Markov's inequality 
$ \Pr( |\xi| > K^{-\frac18} ) \leq K^{-\frac{m}{8}}\Ex(|\xi|^m) $
for $ m=8n $,
we arrive at
\begin{align}
	\label{eq:Vlw:bd}
	\Pr \Big( 
		| \Vlw_K(0,x) - \Uss(x) - 2K^{-4\gamma_3} |
		\leq
		K^{-\frac18},
		\
		\forall x \in \bbR_+
	\Big)
	\geq 1-CK^{-n}.
\end{align}
Combining \eqref{eq:Vlw:bd} and \eqref{eq:Ubd12} yields
\begin{align}
	\label{eq:VlwU}
	\Vlw_K(0,x) - \U_K(\tfrac12,x)
	\geq
	-K^{-\frac18}- CK^{-4\gamma_2}+2K^{-4\gamma_3},
	\quad
	\forall x \in \bbR_+,
\end{align}
with probability $ \geq 1-CK^{-n} $.
With $ \gamma_3<\gamma_2<\frac1{96} $,
the r.h.s.\ of \eqref{eq:VlwU} is positive for all $ K $ large enough,
so \eqref{eq:gaol} holds.

Assuming the event~\eqref{eq:desired} holds,
we proceed to construct the coupling for $ s>0 $.
Let $\tau_1:=\inf\{ t \geq \frac12K: Z(t)=0 \} $ 
be the first absorption time after time $ \frac12K $. 
For $ s \in [0,\tau_1-\frac12K) $, both processes $ \Ylw(s) $ and $ X(s+\frac12K) $ 
evolve as Atlas models.
Hence, by Lemma~\ref{lem:sar:cmp} for $ (Y(s),Y'(s))=(\Ylw(s),X(s+\frac12K)) $, 
we have a coupling such that
\begin{equation*}
	\Ylw_{i}(s) \leq X_{i} (s+\tfrac12K) 
	\quad 
	\forall 
	1 \leq i \leq \#\{X_j(\tfrac{1}{2}K)\},
	\
	s \in [0,\tau_1-\tfrac12K).
\end{equation*}
At time $ t=\tau_1 $,
the system $ X $ \emph{loses} a particle, 
so by reorder $ (\Ylw_i(\tau_1-\frac12K))_i $ and $ (X_i(\tau_1))_i $,
we retain the type of dominance as in \eqref{eq:desired}.
Based on this we iterate the prescribed procedure to the second absorption
$ \tau_2 := \inf\{ s>\tau_1: Z(s) =0 \} $.
As absorption occurs at most $ K $ times, 
the iteration procedure yields the desired coupling
until the extinction time $ \tau_\text{ext} $.
We have thus constructed a coupling of $ (\Ylw(s);s\geq 0) $ and $ (X(s+\frac12K):s\geq 0) $
under which \eqref{eq:cplLw} holds.

We now turn to $ \Yup $ and construct
the analogous coupling of $ (\Ylw(s);s\geq 0) $ and $ (X(s+\sigma_KK):s\geq 0) $.
Similarly to \eqref{eq:Vlw:bd:}, for $ \Vup_K(0,x) $ we have that
\begin{align}
	\label{eq:Vlw:::}
	\Pr\Big(	
		\Big| \Vup_K(0,x) - \int_{K^{-\gamma}}^x u_1(\tfrac12,y) dy \Big|
		\leq
		K^{-\frac18},
		\
		\forall x \geq K^{-\gamma}
	\Big)
	\geq 1-CK^{-n}.
\end{align}
As seen from the expression~\eqref{eq:u1},
$ u_1(\frac12,0)=2 $ and
$ x\mapsto u_1(\frac12,x) $ is smooth with bounded derivatives,
so in particular
\begin{align*}
	\Big| 
		\int_{K^{-\gamma}}^x u_1(\tfrac12,y)dy 
		- (\Uss(x)-2K^{-\gamma})
	\Big|
	\leq
	C K^{-2\gamma},
	\quad
	\forall x \geq K^{-\gamma}.
\end{align*}
Inserting this estimate into \eqref{eq:Vlw:::},
and combining the result with \eqref{eq:Ubd},
we obtain that, with probability $ \geq 1-CK^{-n} $,
\begin{align*}
	\Vup_K(0,x) 
	\leq 
	\U_K(\sigma_K,x) - 2K^{-\gamma} + K^{-\gamma} + CK^{-2\gamma}
	\leq
	\U_K(\sigma_K,x),
	\quad
	\forall x \geq K^{-\gamma},
\end{align*}
for all $ K $ large enough.
This together with $ \Vup_K(0,x)|_{x<K^{-\gamma}}=0 $ yields
the following dominance condition:
\begin{align} 
	\label{eq:desired:}
	\#\{ \Yup_{i}(0)\} \leq \#\{X_i(\sigma_KK)\},
	\quad
	\text{and }
	\Yup_{i}(0) \geq X_i (\sigma_KK),
	\ 
	\forall 1\leq i \leq \#\{ \Yup_j(0) \},
\end{align}
with probability $ \geq 1 - CK^{-n} $.
Based on this, we construct the coupling for 
$ \Yup $ and $ X $ similarly to the proceeding.
Unlike in the proceeding, however, 
when an absorption occurs, 
dominance properties of the type~\eqref{eq:desired:} may be destroyed.
Hence here we obtain the coupling
with the desired property only up to the first absorption time,
as in \eqref{eq:cplUp}.
\end{proof}

We see from Lemma~\ref{lem:couple} 
that $ \Wup_K $ and $ \Wlw_K $ serve as suitable upper and lower bounds for $ Z_K $.
With this, we now turn our attention to the Atlas models $ \Yup $ and $ \Ylw $,
and aim at establishing the hydrodynamic limits of $ \Wup_K $ and $ \Wlw_K $.
To this end,
recalling from \eqref{eq:seminorm} the definition of $ |\Cdot|'_{[0,T]} $
and that $ T<\infty $ is fixed,
we begin by establishing the follow estimates on 
$ |\Wup_K|'_{[0,T]} $ and $ |\Wlw_K|'_{[0,T]} $.
\begin{lemma}\label{lem:Znondecr}
There exists $ C<\infty $
such that 
\begin{align}
	\label{eq:Wup:nond}
	\Pr( |\Wup_K|'_{[0,T]} \leq CK^{-\frac18} ) &\geq 1 -CK^{-n},
\\
	\label{eq:Wlw:nond}
	\Pr( |\Wlw_K|'_{[0,T]} \leq CK^{-\frac18} ) &\geq 1 -CK^{-n},
\end{align}
\end{lemma}
\begin{proof}
The proof of \eqref{eq:Wup:nond}--\eqref{eq:Wlw:nond} are similar,
and we work out only the former here.

At any given time $ s\in\bbR_+ $,
let us \emph{order} the $ \Yup $-particles as
$ \W(s) = \Yup_1(s) \leq \Yup_2(s) \leq \ldots \leq \Yup_{N}(s) $,
where $ N:= \#\{\Yup_j(0)\} $,
and let $ \Gup_i(s) := \Yup_{i+1}(s)-\Yup_{i}(s) $ 
denote the corresponding gap process.
We adopt the convention that $ \Gup_i(s) := \infty $ if $ i+1 > N $,
so that $ \Gup(s) := (\Gup_i(s))_{i=1}^\infty $ is $ [0,\infty]^\infty $-valued.

We begin with a stochastic comparison of the gap process $ \Gup(s) $.
More precisely,
given any $ [0,\infty]^\infty $-valued random vectors $ \xi $ and $ \zeta $,
we say $ \xi $ stochastically dominate $ \zeta $,
denoted $ \xi \succeq \zeta $,
if there exists a coupling of $ \xi $ and $ \zeta $ under which
$ \xi_i \geq \zeta_i $, $ i=1,2,\ldots $.
Since $ (\Yup_i(0)) $ is distributed as in \eqref{eq:Yic},
with $ \uup(x) \leq 2 $, $ \forall x\in\bbR $, we have that
\begin{align}
	\label{eq:G0>G}
	\Gup(0) \succeq  \bigotimes_{i=1}^\infty \Exp(2).
\end{align}
By \cite[Theorem~4.7]{sarantsev14},
for any Atlas model satisfying the dominance property~\eqref{eq:G0>G},
the dominance will continue to hold for $ s>0 $, i.e., 
$ \Gup(s) \succeq \bigotimes_{i=1}^\infty \Exp(2) $.
(Theorem~4.7 of \cite{sarantsev14} does not state 
$ \Gup(s) \succeq \bigotimes_{i=1}^\infty \Exp(2) $ explicitly, 
but the statement appears in the first line of the proof,
wherein $ \pi = \bigotimes_{i=1}^\infty \Exp(2) $,
c.f., \cite[Example~1]{sarantsev14} and \cite{pal08}.)

Having established the stochastic comparison of $ \Gup(s) $,
we now return to estimating $ |\Wup_K|'_{[0,T]} $.
The seminorm $ |\Cdot|'_{[0,T]} $,
defined in \eqref{eq:seminorm}, measures how nondecreasing the given function is.
To the end of bounding $ |\Wup_K|'_{[0,T]} $, we fix $ s_*\in[0,T] $,
and begin by bounding the quantity
\begin{align}
	\label{eq:Wlw:nond:}
	\sup_{s\in[s_*,T]} (\Wup_K(s_*) - \Wup_K(s)).
\end{align}
We consider an \textbf{infinite Atlas model} $ (Y^*_i(s);t\geq 0)_{i=1}^\infty $,
which is defined analogously to \eqref{eq:alt} via the following stochastic differential equations
\begin{align}
	\label{eq:infAtlas}
	dY^*_i(s)= \ind_\Set{ Y^*_i(s)=\W^*(s)} dt + dB_i(s), \quad i=1,2,\ldots,
	\quad
	\W^*(s) := \min\nolimits_{i=1}^\infty\{ Y^*_i(s) \},
\end{align}
with the following initial condition
\begin{align}
	\label{eq:infAtl:ic}
	Y^*_1(0) := \Ylw_1(Ks_*),
	\text{ and, independently }
	(Y^*_{i+1}(0) - Y^*_{i}(0))_{i=1}^\infty \sim \bigotimes_{i=1}^\infty \Exp(2).
\end{align}
General well-posedness conditions 
for \eqref{eq:infAtlas} are studied in \cite{ichiba13,shkolnikov11}.
In particular,
the distribution~\eqref{eq:infAtl:ic} is an admissible initial condition,
and is in fact a stationary distribution of the gaps \cite{pal08}.
Under such stationary gap distribution,
the laggard $ \W^*(s) $ remains very close to a constant under diffusive scaling.
More precisely, 
letting $ \W^*_K(s) := \frac{1}{\sqrt{K}} \W^*(sK) $,
by \cite[Proposition~$2.3$, Remark~$2.4$]{dembo15},
we have
\begin{align}
	\label{eq:DT}
	\Pr \Big( \sup_{ s\in[s_*, T] }|\W^*_K(s)-\W^*_K(0)| \leq K^{-\eta} \Big) 
	\geq 1-C(\eta)K^{-n-2},
\end{align}
for any fixed $ \eta\in(0,\frac14) $.
In view of the bound \eqref{eq:DT},
the idea of bounding the quantity~\eqref{eq:Wlw:nond:}
is to couple $ (Y^*(s);s\geq 0) $ and $ (\Yup(s+s_*);s\geq 0) $.
As we showed previously $ \G(Ks_*) \succeq \bigotimes_{i=1}^\infty \Exp(2) $.
With $ (Y^*_i(0))_i $ distributed in \eqref{eq:infAtl:ic},
we couple $ (\Yup_i(s_*))_i $ and $ (Y^*_i(0))_i $ in such a way that
\begin{align}
	\label{eq:domin:}
	Y^*_i(0) \leq \Yup_i(Ks_*),
	\quad
	i=1,2,\ldots, \#\{ \Yup_i(Ks_*) \}.
\end{align}
Equation~\eqref{eq:domin:}
gives a generalization of the dominance condition~\eqref{eq:domin}
to the case where $ m=\infty $.
For such a generalization we have the analogous
coupling result from \cite[Corollary~$3.9$, Remark~$9$]{sarantsev15},
which gives a coupling of $ (Y^*(s);s\geq 0) $ and $ (\Yup(s+s_*K);s\geq 0) $
such that
\begin{align}
	\label{eq:Wcmp}
	\W^*(s) \leq \Wup(s+s_*K), \ \forall s \in\bbR_+.
\end{align}
Combining \eqref{eq:DT} for $ \eta=\frac18 $ and \eqref{eq:Wcmp}, 
together with $ \W^*(0)=\Wup(s_*K) $ (by \eqref{eq:infAtl:ic}), 
we obtain
\begin{align}
	\label{eq:nondecr:Z}
	\Pr \Big( 
		\sup_{s\in [s_*,T] } 
		(\Wup_K(s_*) - \Wup_K(s))  \leq K^{-\frac18 } 
	\Big)  
	\geq 1-CK^{-n-2}.
\end{align}

Having established the bound \eqref{eq:nondecr:Z}
for fixed $ s_*\in[0,T] $,
we now take the union bound of \eqref{eq:nondecr:Z} over
$ s_*=s_\ell:=K^{-2}T\ell $, $ 1 \leq \ell \leq K^{2} $, to obtain
\begin{equation} \label{eq:discreteZK}
	\Pr \Big(
		\sup_{s\in [s',T] } 
		(\Wup_K(s') - \Wup_K(s))  \leq K^{-\frac18 },
		\ 
		\forall s'=s_1,s_2,\ldots
	\Big) 
	\geq 
	1- C K^{-n}.
\end{equation}
To pass from the `discrete time' $ s'=s_1,s_2,\ldots $
to $ s'\in[0,T] $,
adopting the same procedure we used for obtaining \eqref{eq:driftBMest}, 
we obtain the following continuity estimate:
\begin{align} \label{eq:gapZK}
	\Pr \Big(
		\sup_{s \in [s_\ell,s_{\ell+1}]} |\Wup_K(s)-\Wup_K(s_\ell)| 
		\leq
		K^{-\frac18},
		\
		1 \leq \ell \leq K^{2}
	\Big) 
	\geq 
	1 -C K^{-n}.
\end{align}
Combining \eqref{eq:discreteZK}--\eqref{eq:gapZK} yields
\begin{align*}
	\Pr \Big(
		\sup_{s'<s\in[0,T]} (\Wup_K(s')-\Wup_K(s))
		\leq
		2K^{-\frac18}
	\Big) 
	\geq 
	1 -C K^{-n}.
\end{align*}
This concludes the desired result \eqref{eq:Wup:nond}. 
\end{proof} 

We next establish upper bonds on $ |\Wup_K| $ and $ |\Wlw_K| $.

\begin{lemma}\label{lem:Zupbd}
There exists $ C<\infty $ and a constant $ L=L(T)<\infty $
such that
\begin{align}
	\label{eq:Wup:bd}
	\Pr( |\Wup_K(t)| \leq L, \ \forall t\leq T ) \geq 1 -CK^{-n},
\\
	\label{eq:Wlw:bd}
	\Pr( |\Wlw_K(t)| \leq L, \ \forall t\leq T ) \geq 1 -CK^{-n}.
\end{align}
\end{lemma}
\begin{proof}
We first establish \eqref{eq:Wup:bd}.
The first step is to derive an integral equation for $ \Wup_K $.
Recalling that $ \Vup_K(t,x) $ denote the scaled distribution function of $ \Yup $,
we apply Lemma~\ref{lem:int}\ref{enu:intY} for $ Y=\Yup $ 
to obtain the following integral identity
\begin{align}
	\label{eq:Yup:int}
	\Vup_K(t,x) 
	= 
	\int_0^\infty p(t,x-y) \Vup_K(0,y) dy
	- \int_{0}^{t}  p(t-s,x-\W_K(s)) ds
	+ \rdY_K(t,x).
\end{align}	
Note that the conditions~\eqref{eq:D*}--\eqref{eq:Dan}
hold for $ \Yup(0) $, which is distributed as in \eqref{eq:Yic}.
Using the approximating \eqref{eq:Vlw:bd},
we have
\begin{align}
	\label{eq:Vupint}
	\Big| \int_0^\infty p(t,x-y) \Vup_K(0,y) dy - \int_0^t p(t,x-y) \Uss(y) dy \Big|
	\leq
	C K^{-4\gamma_2},
\end{align}
with probability $ \geq 1 - CK^{-n} $.
Using \eqref{eq:Vupint} and \eqref{eq:rd:bd} in \eqref{eq:Yup:int},
we rewrite the integral identity as
\begin{align}
	\label{eq:Yup:int:}
	\Vup_K(t,x) 
	= 
	\int_0^\infty p(t,x-y) \Uss(y) dy
	- \int_{0}^{t}  p(t-s,x-\W_K(s)) ds
	+ \overline{F}'_K(t,x),
\end{align}	
for some $ \overline{F}'_K(t,x) $ such that 
\begin{align}
	\label{eq:Yup:Fk'}
	\Pr\Big( |\overline{F}'_K|_{L^\infty([0,T]\times\bbR)} \leq C K^{-4\gamma_2} \Big)
	\geq 1-CK^{-n}.
\end{align}
Further, with $ (\Yup^K_i(0)) $ distributed as in \eqref{eq:Yic},
it is standard to verify that 
\begin{align}
	\label{eq:Yup:Wuploc}
	\Pr\Big( |\Wup_K(0)| \leq CK^{-4\gamma_3} \Big)
	\geq 1-CK^{-n}.
\end{align}
By definition, 
$
	\Vup_K(t,\Wup_K(t))
	= \frac{1}{\sqrt{K}} \#\big\{ \YKup_i(t)\in(-\infty,\Wup_K(t)] \big\}
	= \frac1{\sqrt{K}},
$
so setting $ x=\Wup_K(t) $ in \eqref{eq:Yup:int:}
we obtain the follow integral equations
\begin{align}
	\label{eq:Wup:intEq}
	\int_0^\infty p(t,\Wup_{K}(t)-y) \Uss(y) dy
	=
	\int_{0}^{t}  p(t-s,\Wup_{K}(t)-\Wup_{K}(s)) ds
	+ \overline{F}_{K}(t,\Wup_{K}(t)),
\end{align}
where $ \overline{F}_K(t,x) := \frac{1}{\sqrt{K}}-\overline{F}'_K(t,x) $,
which, by \eqref{eq:Yup:Fk'}, satisfies
\begin{align}
	\label{eq:Yup:Fk}
	\Pr\Big( |\overline{F}_K|_{L^\infty([0,T]\times\bbR)} \leq C K^{-4\gamma_3} \Big)
	\geq 1-CK^{-n}.
\end{align}

Having derive the integral equation~\eqref{eq:Wup:intEq} for $ \Wup_K $,
we proceed to showing \eqref{eq:Wup:bd} based on \eqref{eq:Wup:intEq}.
To this, we define $ w^*(t) := \Wup_K(0) + at $, for some $ a\in\bbR_+ $ to be specified later,
and consider the first hitting time $ \tau := \inf \{ t: \Wup_K(t) \geq w^*(t) \} $.
As $ w \mapsto \int_0^\infty p(\tau,w-y) \Uss(y) dy $ is nondecreasing, 
by \eqref{eq:Yup:Wuploc} we have
\begin{align}
	\label{eq:f1t}
	\int_0^\infty p(\tau,\Wup_K(0)+a\tau-y) \Uss(y) dy
	&\geq 
	\int_0^\infty p(\tau,1+a\tau-y) \Uss(y) dy := f_1(\tau),
\end{align}
with probability $ \geq 1- CK^{-n} $.
Using $ \Wup_K(\tau)-\Wup_K(s) \geq a(\tau-s) $, $ \forall s\leq \tau $, we obtain
\begin{align}
	\label{eq:f2a}
	\int_{0}^{\tau}  p(\tau-s,\Wup_K(\tau)-\Wup_K(s)) ds
	&
	\leq
	\int_0^\infty p(s,as) ds := f_2(a).
\end{align}
For the functions $ f_1 $ and $ f_2 $, 
we clearly have $ \inf_{t\leq T} f_1(t) := f_*>0 $
and $ \lim_{a\to\infty} f_2(a) =0 $.
With this, we now fix some large enough $ a $ with $ f_2(a) < \frac12 f_* $,
and insert the bounds \eqref{eq:Yup:Fk}--\eqref{eq:f2a} into \eqref{eq:Wup:intEq} to obtain
\begin{align*}
	\Pr\big( \{f_* \leq \tfrac12 f_* + K^{-4\gamma_3} \} \cap \{ \tau \geq T \} \big)
	\geq
	1-C K^{-n}.
\end{align*}
Since $ f_* >0 $, the event $ \{f_* \leq \tfrac12 f_* + K^{-4\gamma_3} \} $
is empty for all large enough $ K $, so 
\begin{align*}
	\Pr\Big( \Wup_K(t) \leq \Wup_K(0)+aT , \ \forall t\leq T \Big) 
	\geq 1 -CK^{-n}.
\end{align*}
This together with \eqref{eq:Yup:Wuploc}
gives the upper bound
$ \Pr(\Wup_K(t)\leq L,\forall t\leq T) \geq 1-CK^{-n} $ for $ L:=1+aT $.
A lower bound $ \Pr(\Wup_K(t)\geq-L,\forall t\leq T) \geq 1-CK^{-n} $ 
follows directly from \eqref{eq:nondecr:Z} for $ s_*=0 $.
From these we conclude the desired result~\eqref{eq:Wup:bd}.

Similarly to \eqref{eq:Wup:intEq}, for $ \Vlw_K(t,x) $ we have
\begin{align}
	\label{eq:Wlw:intEq}
	\int_0^\infty p(t,\Wlw_{K}(t)-y) \Uss(y) dy
	=
	\int_{0}^{t}  p(t-s,\Wlw_{K}(t)-\Wlw_{K}(s)) ds
	+ \underline{F}_{K}(t,\Wlw_{K}(t)),
\end{align}
for some $ \underline{F}_{K}(t,x) $ satisfying
\begin{align}
	\label{eq:Ylw:Fk}
	\Pr\Big( |\underline{F}_K|_{L^\infty([0,T]\times\bbR)} \leq C K^{-\gamma} \Big)
	\geq 1-CK^{-n}.
\end{align}
From this, the same argument in the proceeding
gives the desired bound~\eqref{eq:Wlw:bd} for $ L=1+aT $.
\end{proof}

We now establish the hydrodynamic limit of $ \Wup_K $ and $ \Wlw_K $.
\begin{lemma}
\label{lem:W:hydro}
There exists $ \z(\Cdot+\frac12)\in\Csp(\bbR_+) $ that solves \eqref{eq:zeq}
(which is unique by Corollary~\ref{cor:unique}).
Furthermore, for some $ C<\infty $, we have
\begin{align}
	\label{eq:Wlw:hydro}
	\Pr\Big(
		|\Wlw_K(s)-\z(\tfrac12+s)| \leq CK^{-4\gamma_3}, \ \forall s\in[0,T]
	\Big)
	&\geq
	1 - CK^{-n},
\\
	\label{eq:Wup:hydro}
	\Pr\Big(
		|\Wup_K(s)-\z(s+\tfrac12)| \leq CK^{-\gamma}, \ \forall s\in[0,T]
	\Big)
	&\geq
	1 - CK^{-n}.
\end{align} 
\end{lemma}
\begin{proof}
The strategy of the proof is 
to utilize the fact that $ \Wup_K $ and $ \Wlw_K $
satisfy the integral equations \eqref{eq:Wup:intEq} and \eqref{eq:Wlw:intEq},
respectively,
and apply the stability estimate Lemma~\ref{lem:pStef}
to show the convergence of $ \Wup_K $ and $ \Wlw_K $.
Given the estimates
\eqref{eq:Yup:Fk} and \eqref{eq:Ylw:Fk}, 
\eqref{eq:Wup:nond}--\eqref{eq:Wlw:nond},
and \eqref{eq:Wup:bd}--\eqref{eq:Wlw:bd},
the proof of \eqref{eq:Wlw:hydro} and \eqref{eq:Wup:hydro} are similar,
and we present only the former.

Such a $ \z $ will be constructed as the unique limit point of $ \Wup_K $.
We begin by showing the convergence of $ \Wup_K $.
To this end, we fix $ K_1<K_2 $, and consider the processes $ \Wup_{K_1} $ and $ \Wup_{K_2} $.
Since they satisfy the integral equation~\eqref{eq:Wup:intEq},
together with the estimates
\eqref{eq:Yup:Wuploc}, \eqref{eq:Yup:Fk}, \eqref{eq:Wup:nond} and \eqref{eq:Wup:bd},
we apply Lemma~\ref{lem:pStef} for $ (w_1,w_2)=(\Wup_{K_1},\Wup_{K_2}) $ to obtain
\begin{align}
	\label{eq:Wup:cauchy}
	\Pr( |\Wup_{K_1}(s)-\Wup_{K_2}(s)| \leq CK_1^{-4\gamma_3}, \forall s\in[0,T] )
	\geq
	1 - C{K_1}^{-n}.
\end{align}
We now consider the subsequence $ \{\Wup_{K_m}\}_{m=1}^\infty $,
for $ K_m:=2^m $.
Setting $ (K_1,K_2)=(2^{m},2^{m+j}) $ in \eqref{eq:Wup:cauchy},
and taking union bound of the result over $ j\in\bbN $,
we obtain
\begin{align}
	\label{eq:Wup:cauchy:geo}
	\Pr\Big( 
		\sup_{t\in[0,T]}|\Wup_{K_{m}}(t)-\Wup_{K_{m'}}(t)| \leq CK_m^{-4\gamma_3}, \forall m'>m
	\Big)
	\geq
	1 - CK_m^{-n}.
\end{align}
From this, we conclude that $ \{ \Wup_{K_m} \}_m $ is almost surely Cauchy in $ \Csp([0,T]) $,
and hence converges to a possibly random limit $ W\in\Csp([0,T]) $.
Now, letting $ K\to\infty $ in \eqref{eq:Wup:intEq},
with \eqref{eq:Yup:Fk}, we see that $ W $ must solve \eqref{eq:zeq}.
Further, by \eqref{eq:Wup:nond} and \eqref{eq:Yup:Wuploc},
$ t\mapsto W(t) $ is nondecreasing with $ W(0)=0 $.
Since, by Corollary~\ref{cor:unique}, the solution to \eqref{eq:zeq} is unique,
$ W(t) =: \z(t+\frac12) $ must in fact be deterministic.
Now, letting $ m'\to\infty $ in \eqref{eq:Wup:cauchy:geo} yields
\begin{align}
	\label{eq:Wup:cauchy:geo:}
	\Pr\Big( 
		\sup_{s\in[0,T]}|\Wup_{K_{m}}(s)-\z(s+\tfrac12)| \leq CK_m^{-4\gamma_3}
	\Big)
	\geq
	1 - CK_m^{-n}.
\end{align}
Combining \eqref{eq:Wup:cauchy} and \eqref{eq:Wup:cauchy:geo:},
we concludes the desired result~\eqref{eq:Wlw:hydro}.
\end{proof}

Having established the hydrodynamic limit of the laggards
$ \Wup_K $ and $ \Wlw_K $ of the Atlas models $ \Yup $ and $ \Ylw $,
we now return to proving Proposition~\ref{prop:Zhy}\ref{enu:Zhy:mvbdy},
i.e.\ proving the hydrodynamic limit \eqref{eq:Zhy:mvbdy} 
of $ Z_K(t) $ for $ t \in[\sigma_K,T] $.
We recall from Lemma~\ref{lem:couple} that we have a coupling
of $ Z_K $ and $ \Wlw_K $ under which
$ \Wlw_K(t-\frac12) \leq Z_K(t) $, $ \forall t\in[\frac12,\tau^K_\text{ext}) $,
with probability $ \geq 1-CK^{-n} $.
By using the lower bound \eqref{eq:extT} on the scaled extinction time $ \tau^K_\text{ext} $,
we have that 
\begin{align*} 
	\Pr \big( 
		\Wlw_K(t-\tfrac12) \leq Z_K(t),
		\
		\forall t\in[\tfrac12,T]
		\big) 
	\geq 
	1 - CK^{-n}.
\end{align*}
Combining this with \eqref{eq:Wlw:hydro} yields
\begin{align}
	\label{eq:Zhy:lw}
	\Pr\Big(  
		Z_K(t) \geq \z(t) - CK^{-4\gamma_3}, \forall t\in[\tfrac12,T] 
	\Big)
	\geq 
	1 - CK^{-n}.
\end{align}
Equation~\eqref{eq:Zhy:lw} gives the desired lower bound on $ Z_K $.
Further, it provides a lower bound on the absorption time 
$ \tau^K_\text{abs} $ (as defined in \eqref{eq:abs>12}).
To see this, we use \eqref{eq:Zhy:lw} to write
\begin{align}
	\label{eq:Zhy:lw:}
	\Pr\Big(
		\inf_{t\in[\sigma_K,T]} Z_K(t) \geq \inf_{t\in[\sigma_K,T]} \z(t) - CK^{-4\gamma_3}
	\Big)
	\geq
	1-CK^{-n}.
\end{align}
With $ \sigma_K=\frac12+\frac17K^{-2\gamma} $
and $ t\mapsto \z(t) $ being non-decreasing,
the quadratic growth \eqref{eq:zdq} of $ \z(t) $
near $ t=\frac12 $ gives
\begin{align}
	\label{eq:zlw}
	\inf_{t\in[\sigma_K,T]} \z(t) = \z(\sigma_K) \geq \tfrac{1}{C}K^{-4\gamma}.
\end{align}
Combining \eqref{eq:zlw} with \eqref{eq:Zhy:lw:}, followed by using $ \gamma<\gamma_3 $,
we obtain
\begin{align}
	\label{eq:tabs:bd}
	\Pr\Big( \inf_{t\in[\sigma_K,T]} Z_K(t) >0 \Big)
	=
	\Pr\big( \tau^K_\text{abs} > T \big)
	\geq
	1 -CK^{-n}.
\end{align}
Using the bounds \eqref{eq:tabs:bd} and \eqref{eq:extT}
on $ \tau^K_\text{abs} $ and $ \tau^K_\text{ext} $ 
withing the coupling \eqref{eq:cplUp},
we have that $ Z_K(t) \leq \Wup_K(t) $, $ \forall t\in[\sigma_K,T] $,
with probability $ \geq 1-CK^{-n} $.
From this and \eqref{eq:Wup:hydro}, we conclude
\begin{align}
	\label{eq:Zhy:up}
	\Pr\Big(  
		Z_K(t) \leq \z(t) + CK^{-\gamma}, \forall t\in[\sigma_K,T] 
	\Big)
	\geq 
	1 - CK^{-n}.
\end{align}
As $ 4\gamma_3>\gamma $,
the bounds \eqref{eq:Zhy:lw} and \eqref{eq:Zhy:up} 
conclude the desired hydrodynamic limit~\eqref{eq:Zhy:mvbdy} of $ Z_K(t) $.

\section{Proof of Theorem~\ref{thm:aldous}}
\label{sect:aldous}

We first settle Part\ref{enu:aldous:upbd}.
To this end,
we fix an arbitrary strategy $ \phi(t) = (\phi_i(t))_{i=1}^K $, 
fix $ \gamma\in(0,\frac14) $ and $ n<\infty $,
and use $ C=C(\gamma,n)<\infty $ to denote a generic constant
that depends only on $ \gamma,n $, and \emph{not} on the strategy in particular.
Our goal is to establish 
an upper on $ \Uc_K(\infty) := \lim_{t\to\infty} \Uc_K(t,Z_K(0)) $,
the total number of ever-surviving particles, scaled by $ \frac{1}{\sqrt{K}} $.
To this end, with $ \Uc_K(\infty) \leq \Uc_K(\frac12,0) $,
we set $ t=\frac12 $ in \eqref{eq:int:abs} to obtain
\begin{align}
	\label{eq:int:aldous}
	\Uc_K(\infty) \leq \Gc_K(\tfrac12,0)
	+ \sum_{i=1}^K \int_{0}^{\frac12\wedge\tau_i^K} \phi^K_i(s) \pN(\tfrac12-s,\XKi(s),0) ds
	+ \rd_K(\tfrac12,x).
\end{align}
On the r.h.s.\ of \eqref{eq:int:aldous}, we
\begin{itemize}
\item[-] use \eqref{eq:Gc:bd} to approximate $ \Gc_K(\tfrac12,0) $ with $ 2p(t,x) $;
\item[-] use $ \pN(\tfrac12-s,\XKi(s),0) \leq 2p(\tfrac12-s,0) $
	and $ \sum_{i=1}^K \phi^K_i(s) \leq 1 $ to bound the integral term;
\item[-] use \eqref{eq:rd:bd} to bound the remainder term $ \rd_K(\tfrac12,x) $.
\end{itemize}
\indent
We then obtain
\begin{align}
	\label{eq:int:aldous:}
	\Uc_K(\infty) 
	\leq 
	2p(\tfrac12,0) + \int_{0}^{\frac12} 2p(\tfrac12-s,0) ds
	+ CK^{-\gamma},
\end{align}
with probability $ \geq 1-CK^{-n} $.
Comparing the r.h.s.\ of \eqref{eq:int:aldous:} 
with the r.h.s.\ of \eqref{eq:Ucs:abs},
followed by using $ \Ucs(\frac12,0)=\frac{4}{\sqrt{\pi}} $ (from \eqref{eq:Usc:cnsv})
we obtain
\begin{align*}
	\Uc_K(\infty) 
	\leq 
	\Ucs(\tfrac12,0) + CK^{-\gamma}
	=
	\tfrac{4}{\sqrt{\pi}} + CK^{-\gamma}.
\end{align*}
with probability $ \geq 1-CK^{-n} $.
This concludes the desired result~\eqref{eq:aldous:upbd} of Part\ref{enu:aldous:upbd}.

We now turn to the proof of Part\ref{enu:aldous:optl}.
Fix $ \gamma\in(0,\frac{1}{96}) $ and $ n<\infty $, and
specialize $ \phi(t) $ to the push-the-laggard strategy hereafter.
Using Theorem~\ref{thm:hydro} for $ T=1 $,
with $ \Uc_K(t) := \Uc_K(t,Z_K(t)) $,
we have that 	
$ 	
	\sup_{t\in[\frac12,1]} |\Uc_K(t) - \Ucs(t,\z(t)) |
	\leq
	CK^{-\gamma},
$
with probability $ \geq 1 -CK^{-n} $.
Combining this with \eqref{eq:Usc:cnsv} yields
\begin{align}
	\label{eq:aldous:hydro01}
	\Pr\Big( 
		|\Uc_K(t) -\tfrac{4}{\sqrt{\pi}}| \leq CK^{-\gamma} ,
		\
		\forall t\in [\tfrac12,1]
	\Big)
	\geq
	1 - CK^{-n}.
\end{align}

Having established \eqref{eq:aldous:hydro01},
we next establish that
\begin{align}
	\label{eq:Wlw>0}
	\Pr \Big(  \inf_{t\in[1,\infty)} \Wlw_K(t-\tfrac12) >0 \Big) \geq 1 -CK^{-n}.
\end{align}
We claim that \eqref{eq:Wlw>0} 
is the desired property in order to complete the proof.
To see this,
recall from Lemma~\ref{lem:couple} that we have a coupling under which \eqref{eq:cplLw} holds,
and, by Lemma~\ref{lem:extT},
we assume without lost of generality that $ \tau^K_\text{ext} >1 $.
Under this setup, 
the event in \eqref{eq:Wlw>0} 
implies $ Z_K(t) \geq \Wlw_K(t-\frac12) >0 $, $ \forall t\in [1,\tau^K_\text{ext}) $
which then forces $ \tau_\text{ext}=\infty $.
That is, the statement~\eqref{eq:Wlw>0} implies $ \Pr(Z_K(t) >0 , \forall t>1)\geq 1-CK^{-n} $,
and hence $ \Pr(\Uc_K(t)=\Uc_K(1),\forall t\geq 1) \geq 1-CK^{-n} $.
This together with \eqref{eq:aldous:hydro01} concludes \eqref{eq:aldous:optl}.

Returning to the proof of \eqref{eq:Wlw>0},
we recall from Remark~\ref{rmk:Y:Kdep}
that the Atlas model $ \Ylw(t) $ as well as its laggard $ \Wlw(t) $
actually depend on $ K $,
which we have omitted up until this point for the sake of notations.
Here we restore such a dependence and write $ \Ylw(t;K) $ and $ \Wlw(t;K) $, etc.
Recall that the initial condition of the Atlas model $ (\Ylw_i(0;K))_i $
is sampled from the Poisson point process 
$ \PPP(\ulw(\frac{x}{\sqrt{K}})) $ in \eqref{eq:Yic:}.
From the definition~\eqref{eq:ulw} of $ \ulw $
and the explicit formula \eqref{eq:u1} of $ u_1(\frac12,x) $,
it is straightforward to verify that the density function 
$ x \mapsto \ulw(\frac{x}{\sqrt{K}}) $ is nonincreasing on its support
$ [-K^{\frac12-4\gamma_3},\infty) $.
Consequently, fixing $ K_1<K_2 $,
we have 
\begin{align*}
	\ulw(\tfrac{1}{\sqrt{K_1}}(x+{K_1}^{\frac12-4\gamma_3})) 
	\leq \ulw(\tfrac{1}{\sqrt{K_2}}(x+{K_2}^{\frac12-4\gamma_3})),
	\quad
	\forall x\in\bbR.
\end{align*}
With this, 
it is standard to construct a coupling of $ \Ylw(0;K_1) $ and $ \Ylw(0;K_2) $ 
under which
\begin{align*}
	&\#\{ \Ylw_i(0;K_1) \} \leq \#\{ \Ylw_i(0;K_2) \},
\\	
	&\Ylw_i(0;K_2)+{K_2}^{\frac12-\gamma_3} 
	\leq 
	\Ylw_i(0;K_1)+{K_1}^{\frac12-\gamma_3},	
	\
	\forall i=1,\ldots, \#\{ \Ylw_i(0;K_1) \} .
\end{align*}
By Lemma~\ref{lem:sar:cmp},
such a dominance coupling at $ s=0 $ is leveraged into a dominance coupling for all $ s>0 $,
yielding
\begin{align}
	\label{eq:W:RG12}
	\Wlw(s;K_1)
	\geq
	\Wlw(s;K_2)+{K_2}^{\frac12-4\gamma_3}-{K_1}^{\frac12-4\gamma_3}
	\geq
	\Wlw(s;K_2),
	\quad
	\forall s \geq 0.	
\end{align}
Now, fix $ K<\infty $ and consider the geometric subsequence $ L_m := K2^{m} $.
We use the union bound to write
\begin{align*}
	\Pr\Big(  \inf_{s\in[\frac12K,\infty)} \Wlw(s;K) \leq 0 \Big)
	\leq
	\sum_{m=1}^\infty 
	\Pr\Big(  \inf_{s\in[L_{m-1},L_{m}]} \Wlw(s;K) \leq 0 \Big).
\end{align*}
Within each $ m $-th term in the last expression,
use the coupling \eqref{eq:W:RG12} for $ (K_1,K_2)=(L_{m-1},L_m) $ to obtain
\begin{align}
	\label{eq:Wlw:rg:1}
	\Pr\Big(  \inf_{s\in[L_{m-1},L_{m}]}\Wlw(s;K) \leq 0 \Big)
	\leq
	\Pr\Big(  \inf_{s\in[L_{m-1},L_{m}]} \Wlw(s;L_m) \leq 0 \Big).
\end{align}
Next, set $ T=1 $ and $ K=L_m $ in \eqref{eq:Wlw:hydro} 
and rewrite the resulting equation in in the pre-scaled form as
\begin{align}
	\label{eq:Wlw:rg:2}
	\Pr\Big( 
		|\Wlw(s;L_m) - \sqrt{L_m}\z(\tfrac{s}{L_m}+\tfrac12)| \leq CL_m^{\frac12-4\gamma_3}, 
		\forall s\in[0,L_m] 
	\Big)
	\geq
	1 -CL_m^{-n}.
\end{align}
Further, by Lemma~\ref{lem:zqd} and the fact that $ t\mapsto \z(t) $ is nondecreasing,
we have that
\begin{align}
	\label{eq:Wlw:rg:3}
	\inf_{s\in[L_{m-1},L_m]} \z(\tfrac{s}{L_m}+\tfrac12)
	=
	\inf_{t\in[\frac12,1]} \z(t)
	=
	\z(\tfrac12) >0.
\end{align}
Combining \eqref{eq:Wlw:rg:2}--\eqref{eq:Wlw:rg:3} yields 
$ \Pr( \inf_{s\in[L_{m-1},L_{m}]} \Wlw(s,L_m) \leq 0 ) \leq CL_m^{-n} $.
Inserting this bound into \eqref{eq:Wlw:rg:1},
and summing the result over $ m $,
we arrive at
\begin{align*}
	\Pr\Big(  \inf_{s\in[\frac12K,\infty)} \Wlw(s;K) \leq 0 \Big)
	\leq
	C \sum_{m=1}^\infty 
	L_m^{-n}
	=
	C K^{-n}.
\end{align*}
This concludes~\eqref{eq:Wlw>0}
and hence complete the proof of Part\ref{enu:aldous:optl}.

\bibliographystyle{alphaabbr}
\bibliography{upriver}
\end{document}